\documentclass[10pt]{article}

\usepackage{amsmath}    
\usepackage{amsthm}     
\usepackage{amssymb}    
\usepackage{geometry}   
\usepackage{amssymb}
\usepackage{stmaryrd}
\usepackage{xcolor}
\usepackage{hyperref}
\usepackage{graphicx}
\usepackage{algorithm}
\usepackage{subcaption}
\usepackage{algpseudocode}
\usepackage{authblk}
\usepackage[normalem]{ulem}
\DeclareMathOperator*{\argmax}{arg\,max}
\DeclareMathOperator*{\argmin}{arg\,min}

\newtheorem{theorem}{Theorem}[section]
\newtheorem{lemma}[theorem]{Lemma}

\newtheorem{corollary}[theorem]{Corollary}

\newtheorem{example}[theorem]{Example}

\begin{document}

\title{Information-theoretic \textcolor{black}{coordinate} subset \textcolor{black}{and partition} selection of multivariate Markov chains via submodular optimization}
\author[1]{Zheyuan Lai\thanks{Email: zheyuan\_lai@u.nus.edu}}
\author[1]{Michael C.H. Choi\thanks{Email: mchchoi@nus.edu.sg, corresponding author}}
\affil[1]{Department of Statistics and Data Science, National University of Singapore, Singapore}
\date{\today}
\maketitle

\begin{abstract}
We study the problem of optimally projecting the transition matrix of a finite ergodic multivariate Markov chain onto a lower-dimensional state space, \textcolor{black}{as well as the problem of finding an optimal partition of coordinates such that the factorized Markov chain gives minimal information loss compared to the original multivariate chain.} Specifically, we seek to construct a Markov chain that optimizes various information-theoretic criteria under cardinality constraints. These criteria include entropy rate, information-theoretic distance to factorizability, independence, and stationarity. We formulate these tasks as best subset \textcolor{black}{or partition} selection problems over multivariate Markov chains and leverage the \textcolor{black}{($k$-)}submodular (or \textcolor{black}{($k$-)}supermodular) structures of the objective functions to develop efficient greedy-based algorithms with theoretical guarantees. \textcolor{black}{Along the way,} we introduce a generalized version of the distorted greedy algorithm, which may be of independent interest. Finally, we illustrate the theory and algorithms through extensive numerical experiments with publicly available code on multivariate Markov chains associated with the Bernoulli--Laplace and Curie--Weiss \textcolor{black}{models}.
\end{abstract}

\section{Introduction}

Consider a multivariate ergodic Markov chain with transition matrix $P$ that admits a stationary distribution $\pi$ on a finite product state space $\mathcal{X}$ with $d \in \mathbb{N}$ coordinates. Given a subset $S \subseteq \{1,\ldots,d\}$, let $P^{(S)}$ denote the projected transition matrix onto $S$, and let $\Pi$ be the transition matrix where each row is given by $\pi$. These notations are formally defined in Section \ref{sec:prelim} and \ref{sec:dist2stat}.  

A number of natural and interesting information-theoretic optimization problems arise in the context of subset selection for multivariate Markov chains. For example, among all subsets $S$ with $|S| \leq m$, which projected transition matrix $P^{(S)}$ maximizes the entropy rate, thereby identifying the most ``random'' coordinates? Similarly, which choice of $S$ minimizes the Kullback-Leibler (KL) divergence between $P^{(S)}$ and $\Pi^{(S)}$, making it closest to stationarity and highlighting subsets nearest to equilibrium? 
\textcolor{black}{More generally, which choice of partition $\mathbf{S}$ of coordinates approximates a complex transition matrix $P$ by a simpler model that factorizes the groups of coordinates with respect to $\mathbf{S}$ with minimal information losses?}
These insights have direct applications in Markov chain Monte Carlo (MCMC), where selecting and analyzing optimal subsets or \textcolor{black}{partitions} may lead to the design of accelerated MCMC samplers. \textcolor{black}{For example, in Section \ref{subsec:MCMC} below, we make use of the identified subset from the output of submodular optimization algorithms to propose a heuristics that leads to an improved sampler for the Curie--Weiss model.}

Model reduction problems for Markov chains have been studied previously using spectral theory \cite{DMM11}. While such approach is powerful, combinatorial approaches to subset \textcolor{black}{or partition} selection in the context of Markov chains remain underexplored, a gap we address in this paper by leveraging submodular optimization.
\textcolor{black}{Spectral approaches rely on the knowledge of eigenvalues and eigenvectors of the underlying transition matrices, which might not be tractable especially under high-dimensional settings. On the other hand, the combinatorial approach given in this paper relies on the knowledge of the transition matrix, which is usually known or has a tractable structure in most models of interests.}
Our work builds on recent efforts, such as \cite{choi2024ratedistortionframeworkmcmcalgorithms}, and adapts a combinatorial lens to develop efficient greedy-based algorithms for finding approximate optimizers.
\textcolor{black}{In this work, we are interested in selecting subset or partition of the coordinates of multivariate Markov chains, which is different from aggregating the states of Markov chains in the investigation of lumpability \cite{GT14}.}

{\color{black}Our main contributions are as follows. Throughout, we use several information-theoretic ``distance'' functionals that compare the original chain to a reference model: \emph{distance to stationarity} compares a projected Markov chain transition matrix and its projected stationary kernel, \emph{distance to independence} quantifies how far a multivariate Markov chain is from behaving independently in an information-theoretic sense, and \emph{distance to factorizability} measures the information loss when approximating a multivariate Markov chain by a model that factorizes according to a coordinate partition. Precise definitions are given in Sections \ref{sec:prelim} and \ref{sec:dist2stat}.

\begin{itemize}
    \item \textbf{Identification of new submodular functions in Markov chain theory.} 
    We show that several natural set functionals arising from the above information-theoretic distances exhibit a \emph{diminishing returns} (i.e. submodular) or \emph{increasing returns} (i.e. supermodular) structure: under suitable assumptions, the distance-to-stationarity functional is \emph{supermodular}, the distance-to-stationarity functional for the complement set is \emph{submodular}, whereas the distance-to-independence functional for the complement set is \emph{submodular}. This extends the line of work initiated in \cite{choi2024ratedistortionframeworkmcmcalgorithms}.  
    
    \item \textbf{Greedy subset/partition selection for Markov chains via submodular optimization.} 
    Leveraging the above diminishing returns properties, we adapt existing greedy-style algorithms with theoretical guarantees to approximately optimize these information-theoretic objectives under practical constraints (e.g., a budget on the number of selected coordinates, or a prescribed number of groups in a partition). To the best of our knowledge, this is the first systematic use of greedy submodular optimization for subset/partition selection problems tailored to multivariate Markov chains.  
    
    \item \textbf{A generalized distorted greedy algorithm for $k$-submodular maximization under cardinality constraints.} 
    We generalize the distorted greedy algorithm of \cite{pmlr-v97-harshaw19a} into $k$-submodular setting, where the decision variable is a \emph{partition} of coordinates up to $k$ labeled groups, and we provide theoretical guarantees under cardinality-type constraints. This algorithmic result is of independent interest in $k$-submodular maximization problems beyond the Markov chain applications considered here.
    
    \item \textbf{Numerical validation on structured multivariate Markov chains.} 
    We conduct experiments on multivariate Markov chains associated with the Bernoulli--Laplace level model and the Curie--Weiss model, demonstrating that the proposed greedy procedures efficiently identify subsets/partitions that optimizes the targeted information-theoretic criteria.  
\end{itemize}}

{\color{black}
\paragraph{Relation to \cite{choi2024ratedistortionframeworkmcmcalgorithms}.}
Although \cite{choi2024ratedistortionframeworkmcmcalgorithms} also considers information-theoretic distances to factorizability or independence, in this paper we tackle this and other related problems from a submodular optimization perspective which is fundamentally different.
In \cite{choi2024ratedistortionframeworkmcmcalgorithms}, the goal is to \emph{construct or approximate} a transition matrix that is ``closest'' to a given multivariate chain under structural constraints (such as factorization). 
In contrast, in this manuscript we \emph{do not modify the original chain} \(P\). Instead, we treat \(P\) as fixed and perform \emph{combinatorial selection} over the coordinate set: we choose a subset \(S\subseteq\{1,\ldots,d\}\) (or a partition \(\mathbf{S}\)) and evaluate the induced projected (or factorized-by-partition) dynamics. 
This leads to a different class of problems, that is to optimize information-theoretic criteria \emph{over subsets/partitions} subject to natural cardinality-type constraints. 
Examples include selecting coordinates whose projection has large entropy rate, or selecting subsets/partitions whose induced dynamics are closest (in Kullback-Leibler divergence) to stationarity, independence, or factorizability. 
To the best of our knowledge, this subset/partition-selection formulation and its greedy submodular optimization approach have not been developed in the Markov chain literature.}

The remainder of this paper is organized as follows. \textcolor{black}{Section \ref{subsec:infotheoryMC} discusses the information-theoretic properties of multivariate Markov chains.} Section \ref{subsec:submodular} provides a review of submodularity and $k$-submodularity \textcolor{black}{and some related examples in Markov chain theory}. Section \ref{subsec:distortgreedy} introduces the distorted greedy algorithm and presents our generalized version with theoretical guarantees. We then explore optimization problems related to entropy rate (Section \ref{sec:entropy_rate}), distance to factorizability (Section \ref{sec:dist2fact}), distance to independence (Section \ref{sec:dist2indp}), distance to stationarity (Section \ref{sec:dist2stat}), and distance to factorizability over a fixed set (Section \ref{sec:dist2fact_fixed}). Finally, we illustrate the algorithms through numerical experiments in Section \ref{sec:numericalexp}.  

\section{Preliminaries}\label{sec:prelim}
\subsection{\textcolor{black}{Information theory of multivariate Markov chains}}
\label{subsec:infotheoryMC}

Throughout this paper, we consider a finite $d$-dimensional state space described by $\mathcal{X}= \mathcal{X}^{(1)} \times \ldots \times \mathcal{X}^{(d)}$. For $S\subseteq \llbracket d \rrbracket$, we write $\mathcal{X}^{(S)} = \times_{i \in S} \mathcal{X}^{(i)}$. 
\textcolor{black}{One of our goals is to choose a subset $S \subseteq \llbracket d \rrbracket$ that yields, in suitable senses, an optimal lower-dimensional description of a given multivariate Markov chain.}
We denote by $\mathcal{L}(\mathcal{X})$ to be the set of transition matrices on $\mathcal{X}$, and $\mathcal{P}(\mathcal{X}) = \{\pi;~\min_{x \in \mathcal{X}} \pi(x) > 0\}$ to be the set of probability masses with support on $\mathcal{X}$. 
We say that $P\in \mathcal{L}(\mathcal{X})$ is $\pi$-stationary if it satisfies $\pi = \pi P$.

We now recall the definition of the tensor product of transition matrices and probability masses, see e.g. Exercise 12.6 of \cite{levin2017markov}. 
Define, for $M_l \in \mathcal{L}(\mathcal{X}^{(l)})$, $\pi_l \in \mathcal{P}(\mathcal{X}^{(l)})$, $x^l,y^l \in \mathcal{X}^{(l)}$ for $l \in \{i,j\}, i \neq j \in \llbracket d \rrbracket,$
\begin{align*}
    (M_i \otimes M_j)((x^i, x^j), (y^i, y^j)) &:= M_i(x^i, y^i) M_j(x^j, y^j), \\
    (\pi_i \otimes \pi_j)(x^i, x^j) &:= \pi_i(x^i) \pi_j(x^j).
\end{align*}

A transition matrix $P\in \mathcal{L}(\mathcal{X})$ is said to be in a product form if there exists $M_i \in \mathcal{L}(\mathcal{X}^{(i)})$ for $i\in \llbracket d \rrbracket$ such that \textcolor{black}{$P$ can be expressed as a $d$-fold tensor product, that is,  $P = \otimes_{i=1}^d M_i$.}
\textcolor{black}{A product form $P$ describes a Markov chain where the $d$ coordinates evolve independently and in parallel.}
A probability mass $\pi$ is said to be in a product form if there exists $\pi_i \in \mathcal{P}(\mathcal{X}^{(i)})$ such that $\pi = \otimes_{i=1}^d \pi_i$, \textcolor{black}{which means that the coordinates are independent under $\pi$.}

We then recall the definition of leave-$S$-out and keep-$S$-in transition matrices of a given transition matrix $P$, see Section 2.2 of \cite{choi2024ratedistortionframeworkmcmcalgorithms}. Let \(\pi \in P(\mathcal{X})\), \(P \in \mathcal{L} (\mathcal{X})\), and \(S \subseteq \llbracket d \rrbracket\). For any \((x^{(-S)}, y^{(-S)}) \in \mathcal{X}^{(-S)} \times \mathcal{X}^{(-S)}\), we define the $\textbf{leave}$-$S$-\textbf{out} transition matrix to be $P_\pi^{(-S)}$ with entries given by
    \[
    P_\pi^{(-S)}(x^{(-S)}, y^{(-S)}) :=
    \frac{\sum_{(x^{(S)}, y^{(S)}) \in \mathcal{X}^{(S)} \times \mathcal{X}^{(S)}} \pi(x^1, \dots, x^d) P((x^1, \dots, x^d), (y^1, \dots, y^d))}
    {\sum_{x^{(S)} \in \mathcal{X}^{(S)}} \pi(x^1, \dots, x^d)}.
    \]
The \textbf{keep}-\(S\)-\textbf{in} transition matrix of \(P\) with respect to \(\pi\) is 
    \[
    P_\pi^{(S)} := P_\pi^{(-\llbracket d \rrbracket \setminus S)} \in \mathcal{L}(\mathcal{X}^{(S)}).
    \]
In the special case of \(S = \{i\}\) for \(i \in \llbracket d \rrbracket\), we write
    \[
    P_\pi^{(-i)} = P_\pi^{(-\{i\})}, \quad P_\pi^{(i)} = P_\pi^{(\{i\})}.
    \]
When \(P\) is \(\pi\)-stationary, we omit the subscript $\pi$ and write directly \(P^{(-S)}, P^{(S)}\). We also apply the convention of $P^{(\emptyset)} = P^{(-\llbracket d \rrbracket)} = 1$.

{\color{black}To understand leave-$S$-out and keep-$S$-in intuitively, first let us imagine that we only care about coordinates \emph{outside} $S$ and simply do not track the coordinates in $S$. Then the leave-$S$-out transition matrix is the effective transition probabilities rule for how the remaining coordinates move in one step, obtained by averaging over the forgotten coordinates in $S$ according to the reference distribution $\pi$. On the other hand, suppose we keep only the coordinates in $S$ and treat everything else in $S^c$ as hidden background.
Then the keep-$S$-in transition matrix is the induced one-step evolution law on $S$, again formed by averaging over the hidden background $S^c$ using $\pi$.}

We proceed to recall the Shannon entropy of a probability distribution and the entropy rate of the transition matrix, see Section 1 of~\cite{Polyanskiy_Wu_2025}. \textcolor{black}{For a probability distribution $\pi$ on $\mathcal{X}$}, its \textbf{Shannon entropy} is defined as
    $$H(\pi) := -\sum_{x\in\mathcal{X}} \pi(x) \ln{\pi(x)},$$
where the standard convention of $0 \ln 0 := 0$ applies. For $\pi$-stationary $P\in \mathcal{L}(\mathcal{X})$, the \textbf{entropy rate} of $P$ is defined as
$$H(P) := -\sum_{x\in \mathcal{X}} \sum_{y\in \mathcal{X}} \pi(x) P(x, y) \ln{P(x, y)}.$$

{\color{black}Intuitively, the Shannon entropy $H(\pi)$ of $\pi$ measures the typical ``surprise'' (information content) of a draw from $\pi$. It is large when $\pi$ is spread out (high uncertainty) and small when $\pi$ concentrates on a few states (low uncertainty). Similarly, for a $\pi$-stationary $P$, its entropy rate $H(P)$ is the average one-step randomness in the next state given the current state. It quantifies how unpredictable the trajectory is per time step: deterministic transitions give $H(P)=0$, while more diffusive transition rows yield larger entropy rate.}

We shall also recall the definition of KL divergence between Markov chains (Definition 2.1 of~\cite{choi2024ratedistortionframeworkmcmcalgorithms}) and the distance to independence (Definition 2.2 of~\cite{choi2024ratedistortionframeworkmcmcalgorithms}). For given \(\pi \in \mathcal{P}(\mathcal{X})\) and transition matrices \(M, L \in \mathcal{L} (\mathcal{X})\), we define the \textbf{KL divergence} from \(L\) to \(M\) with respect to \(\pi\) as
\[
D_\mathrm{KL}^{\pi}(M\|L) := \sum_{x \in \mathcal{X}} \pi(x) \sum_{y \in \mathcal{X}} M(x, y) \ln \left( \frac{M(x, y)}{L(x, y)} \right).
\]
where the convention of $0 \ln \frac{0}{a} := 0$ applies for $a \in [0,1]$. 
Note that \(\pi\) need not be the stationary distribution of \(L\) or \(M\). In particular, when \(M,L\) are assumed to be \(\pi\)-stationary, we write \(D(M\|L) := D_\mathrm{KL}^{\pi}(M\|L)\), which can be interpreted as the KL divergence rate from \(L\) to \(M\). Given $P\in \mathcal{L}(\mathcal{X})$, we define the \textbf{distance to independence} of $P$ with respect to $D_\mathrm{KL}^\pi$ to be $$\mathbb{I}^\pi (P):= \min_{L_i \in \mathcal{L}(\mathcal{X}^{(i)}),~ \forall i \in \llbracket d \rrbracket} D_\mathrm{KL}^\pi (P \| \otimes_{i=1}^d L_i) = D^\pi_\mathrm{KL}(P \| \otimes_{i=1}^d P_\pi^{(i)}).$$
We write $\mathbb{I} (P) = \mathbb{I}^\pi (P)$ if $P$ is $\pi$-stationary.

{\color{black}Intuitively, the KL divergence $D^\pi_\mathrm{KL}(M \| L)$ from $L$ to $M$ measures the row-wise discrepancies of $M,L$ as weighted by $\pi$. The quantity $\mathbb{I}^\pi(P)$ measures how far a multivariate $P$ is from behaving like $d$ independent coordinate-wise Markov chain, in the sense of KL divergence under $\pi$. The minimizer $\otimes_{i=1}^d P_\pi^{(i)}$ is the closest product approximation to $P$ obtained by matching each coordinate's induced dynamics, so $\mathbb{I}^\pi(P)=0$ exactly when the transition factorizes as $P=\otimes_{i=1}^d P_\pi^{(i)}$.}

{\color{black}One major difference between this manuscript and \cite{choi2024ratedistortionframeworkmcmcalgorithms} is that in the latter, one is interested in a single minimization problem of seeking the closest product form Markov chain, that is,
\begin{align*}
    \mathbb{I}^\pi (P) = \min_{L_i \in \mathcal{L}(\mathcal{X}^{(i)}),~ \forall i \in \llbracket d \rrbracket} D_\mathrm{KL}^\pi (P \| \otimes_{i=1}^d L_i),
\end{align*}
while in Section \ref{sec:dist2indp} of this manuscript, we shall investigate a double minimization problem where we seek a projected Markov chain being closest to independence with given cardinality constraint, that is, for $m \in \llbracket d \rrbracket$,
\begin{align*}
    \min_{S\subseteq \llbracket d \rrbracket;~ |S| = m} \mathbb{I}^{\pi^{(S)}}(P^{(S)}) &= \min_{S\subseteq \llbracket d \rrbracket;~ |S| = m} \min_{L_i \in \mathcal{L}(\mathcal{X}^{(i)}),~ \forall i \in S} D^{\pi^{(S)}}_\mathrm{KL}(P^{(S)} \| \otimes_{i \in S} L_i)\\
    &= \min_{S\subseteq \llbracket d \rrbracket;~ |S| = m} D^{\pi^{(S)}}_\mathrm{KL}(P^{(S)} \| \otimes_{i \in S} P^{(i)}).
\end{align*}}

Finally, we recall the partition lemma for KL divergence of Markov chains, \textcolor{black}{which shows that the KL divergence between the original multivariate Markov chains is at least as large as the projected coordinate subset of the two chains.}
\begin{theorem}[Partition lemma \textcolor{black}{(Theorem 2.4 of \cite{choi2024ratedistortionframeworkmcmcalgorithms})}]\label{thm:part_lem}
    Let $\pi \in \mathcal{P}(\mathcal{X})$, $P, L \in \mathcal{L}(\mathcal{X})$ and suppose $S \subseteq \llbracket d \rrbracket$, we have:
    $$D^\pi_\mathrm{KL}(P \| L) \geq D^{\pi^{(S)}}_\mathrm{KL} (P^{(S)} \| L^{(S)}).$$
\end{theorem}

\subsection{Definition and properties of submodular functions}\label{subsec:submodular}

We first recall the definition of a submodular function \cite{MR3549595}. Given a finite nonempty ground set \( U \), a set function \( f : 2^U \to \mathbb{R} \)  defined on subsets of \( U \) is called \textbf{submodular} if for all \( S, T \subseteq U \),
    \[
    f(S) + f(T) \geq f(S \cap T) + f(S \cup T).
    \]
$f$ is said to be \textbf{supermodular} if $-f$ is submodular, and $f$ is said to be \textbf{modular} if $f$ is both submodular and supermodular.

Next, we recall a result that states the complement of a submodular function is still submodular:

\begin{lemma}\label{lem:complement_submodularity}
    If $S \mapsto f(S)$ is submodular, then $S\mapsto f(U \backslash S)$ is submodular.
\end{lemma}
\begin{proof}
    We choose $S\subseteq T\subseteq U$ and $e\in U\backslash T$, then
    \begin{align*}
        \big(f(U &\backslash (S\cup \{e\})) - f(U \backslash S)\big) - \big(f(U \backslash (T\cup \{e\})) - f(U \backslash T)\big) \\
        &= \big(f(U \backslash T) - f(U \backslash (T\cup \{e\}))\big) - \big(f(U \backslash S) - f(U \backslash (S\cup \{e\}))\big) \geq 0
    \end{align*}
    since $S \mapsto f(S)$ is submodular and $U\backslash T \subseteq U\backslash S$, and hence $S\mapsto f(U \backslash S)$ is submodular.
\end{proof}

We call a \textcolor{black}{set} function $f: 2^U \to \mathbb R$ \textbf{symmetric} if $f(A) = f(U\backslash A)$ for all $A\subseteq U$. \textcolor{black}{A set function \( f \) is said to be \textbf{monotonically non-decreasing} (resp.~\textbf{non-increasing}) if  
    \[
    f(S) \leq (\textrm{resp.}\, \geq)\,  f(T) \quad \forall \,S \subseteq T \subseteq U.
    \]}

\textcolor{black}{We proceed to recall some examples of submodular functions in multivariate Markov chains as discussed in Proposition 2.6 of~\cite{choi2024ratedistortionframeworkmcmcalgorithms}. These submodular or supermodular structures motivate us to design and apply submodular optimization algorithms for seeking optimal subset or partition of the coordinates of Markov chains in subsequent sections.}

\begin{theorem}[\textcolor{black}{Submodularity of some functions in Markov chain theory (Proposition 2.6 of \cite{choi2024ratedistortionframeworkmcmcalgorithms})}]\label{thm:submod_mc}
Let \( S \subseteq \llbracket d\rrbracket \). Let \( P \in \mathcal{L}(\mathcal{X}) \) be a $\pi$-stationary transition matrix. We have
\begin{itemize}
    \item (Submodularity of the entropy rate of \( P \)) The mapping \( S \mapsto H(P^{(S)}) \) is submodular.
    \item (Submodularity of the distance to \( (S, \llbracket d \rrbracket \backslash S) \)-factorizability of \( P \)) The mapping \( S \mapsto D(P\|P^{(S)} \otimes P^{(-S)}) \) is submodular.
    \item (Supermodularity and monotonicity of the distance to independence) The mapping \( S \mapsto \mathbb I(P^{(S)}) \) is monotonically non-decreasing and supermodular.
\end{itemize}
\end{theorem}

{\color{black}Intuitively, submodularity (resp.~supermodularity) captures the idea of diminishing (resp.~increasing) returns. $f$ is submodular (resp.~supermodular) if adding an item to a set gives less (resp.~more) additional value of $f$ when the set is already large than when the set is small. In the context of Markov chains and in view of Theorem \ref{thm:submod_mc}, we see that, for $S \subseteq T \subseteq \llbracket d \rrbracket$ and $i \notin T$, we have
\begin{align*}
    H(P^{(S \cup \{i\})}) - H(P^{(S)}) &\geq H(P^{(T \cup \{i\})}) - H(P^{(T)}), \\
    D(P\|P^{(S \cup \{i\})} \otimes P^{(-(S \cup \{i\}))}) - D(P\|P^{(S)} \otimes P^{(-S)}) &\geq D(P\|P^{(T \cup \{i\})} \otimes P^{(-(T \cup \{i\}))}) - D(P\|P^{(T)} \otimes P^{(-T)}), \\
    \mathbb I(P^{(S \cup \{i\})}) - \mathbb I(P^{(S)}) &\leq \mathbb I(P^{(T \cup \{i\})}) - \mathbb I(P^{(T)}).
\end{align*}
In words, adding a coordinate $i$ to a lower-dimensional transition matrix $P^{(S)}$ gives at least as high a marginal gain of entropy rate as adding to a higher-dimensional transition matrix $P^{(T)}$. Similarly, adding a coordinate $i$ to a transition matrix $P$ gives at least as high a marginal gain of distance to $(S,\llbracket d \rrbracket \backslash S)$-factorizability as adding to the distance to $(T,\llbracket d \rrbracket \backslash T)$-factorizability. Finally, the marginal gain of distance to independence of adding a coordinate $i$ to a lower-dimensional transition matrix $P^{(S)}$ gives at most the marginal gain of adding the same coordinate to a higher-dimensional transition matrix $P^{(T)}$. Furthermore, the distance to independence is monotonically non-decreasing, meaning that the larger the set is, the further away from independence of the multivariate Markov chain projection:
\begin{align*}
    \mathbb I(P^{(S)}) \leq \mathbb I(P^{(T)}).
\end{align*}
}

Next, we investigate the map \( S \mapsto \mathbb I(P^{(-S)}) \), and show that it is monotonically non-increasing and supermodular.

\begin{theorem}[\textcolor{black}{Supermodularity and monotonicity of the distance to independence of $P^{(-S)}$}]\label{thm:dist2indp_complement}
    The mapping $S \mapsto \mathbb{I} (P^{(-S)})$ is monotonically non-increasing and supermodular.
\end{theorem}

\begin{proof}
    We first prove the monotonicity. Suppose $S \subseteq T \subseteq \llbracket d \rrbracket$, then $\llbracket d \rrbracket \backslash T \subseteq \llbracket d \rrbracket \backslash S$, hence according to the partition lemma (\textcolor{black}{Theorem} \ref{thm:part_lem}), we have:
    $$\mathbb I(P^{(-S)}) = D(P^{(-S)} \| \otimes_{i\in \llbracket d \rrbracket \backslash S} P^{(i)}) \geq D(P^{(-T)} \| \otimes_{i\in \llbracket d \rrbracket \backslash T} P^{(i)}) = \mathbb I(P^{(-T)}),$$ therefore, $S \mapsto \mathbb{I} (P^{(-S)})$ is monotonically non-increasing.

    We then look into the supermodularity of this map. Since $$\mathbb I(P^{(-S)}) = \sum_{i \in \llbracket d \rrbracket \backslash S} H(P^{(i)}) - H(P^{(-S)}),$$ then $\mathbb I(P^{(-S)})$ is supermodular because $H(P^{(-S)})$ is submodular in view of Lemma \ref{lem:complement_submodularity}.
\end{proof}

{\color{black}
    A classical method of submodular optimization is the heuristic greedy algorithm (Section 4 of \cite{MR503866}). However, its $(1 - e^{-1})$-approximation guarantee requires that the objective function is monotonically non-decreasing. For our example of submodular structures in Markov chain theory, we note that the maps $S \mapsto H(P^{(S)})$ and $S \mapsto D(P \| P^{(S)} \otimes P^{(-S)})$ are not monotone. This motivates us to introduce the following result, that we shall apply in subsequent sections, to transform a non-monotone submodular $f$ to a monotonically non-decreasing submodular $g$:
    
    \begin{theorem}[Transform a non-monotone submodular $f$ to a monotone submodular $g$ (Proposition 14.18 of \cite{korte2011combinatorial})] \label{thm:monotonize}
        Let $f: 2^U \to \mathbb{R}$ be a submodular function and $\beta \in \mathbb R$, then $g: 2^U \to \mathbb R$ defined by 
        $$g(S) := f(S) -\beta + \sum_{e \in S}(f(U\backslash \{e\}) - f(U))$$
        is submodular and monotonically non-decreasing.
    \end{theorem}
}

A multivariate generalization of submodularity is known as $k$-submodularity \cite{ene2022streaming} where $k \in \mathbb{N}$. In particular, $1$-submodular function is equivalent to submodular function. Let \( f : (k + 1)^U \to \mathbb{R} \) be a set function. The function \( f \) is said to be \( k \)-\textbf{submodular} if  
\[
f(\mathbf{S}) + f(\mathbf{T}) \geq f(\mathbf{S} \sqcap \mathbf{T}) + f(\mathbf{S} \sqcup \mathbf{T}) \quad \forall \, \mathbf{S}, \mathbf{T} \in (k + 1)^U,
\]
where \(\mathbf{S} \sqcap \mathbf{T}\) is the \( k \)-tuple whose \( i \)-th set is \( S_i \cap T_i \) and \(\mathbf{S} \sqcup \mathbf{T}\) is the \( k \)-tuple whose \( i \)-th set is \((S_i \cup T_i) \setminus \left(\bigcup_{j \neq i} (S_j \cup T_j)\right)\). A function $f$ is said to be $k$-\textbf{supermodular} if $-f$ is $k$-submodular.

For $\mathbf{S} = (S_1, \ldots, S_k),\mathbf T = (T_1, \ldots, T_k) \in (k+1)^U$, we write $\mathbf S \preceq \mathbf T$ if and only if $S_i \subseteq T_i$ $\forall i \in \llbracket k \rrbracket$, where $\llbracket k \rrbracket := \{1,2,\ldots,k\}.$ A function \( f \) is said to be \textbf{monotonically non-decreasing} (resp.~\textbf{non-increasing}) if  
    \[
    f(\mathbf{S}) \leq (\textrm{resp.}\, \geq)\,  f(\mathbf{T}) \quad \forall \,\mathbf{S} \preceq \mathbf{T}.
    \]

{\color{black}
Intuitively, $k$-submodularity (resp. $k$-supermodularity) serves as a natural generalization to multivariate settings of submodular (resp. supermodular) functions, which exhibits the similar property of ``diminishing (resp. increasing) returns''.
The following sensor placement problem serves as an example of $k$-submodularity.
\begin{example}[Sensor placement problem (Section 5.2 of \cite{NIPS2015_f770b62b})] \label{eg:sensor_placement}
    Let $V$ be a finite set of locations and let $\llbracket k \rrbracket$ be the indices of $k$ kinds of sensors. For each $e\in V$ and $i\in\llbracket k\rrbracket$, let $X_e^i$ be a discrete random variable representing the observation collected by placing a sensor of type $i$ at location $e$, and let \[\Omega := \{X_e^i : e\in V,\ i\in \llbracket k\rrbracket\}.\]
    A sensor placement is encoded by a labeling $x\in\{0,1,\dots,k\}^V$, where $x(e)=0$ means that no sensor is placed at $e$, and $x(e)=i\in \llbracket k\llbracket$ means that a sensor of type $i$ is placed at $e$. For such an $x$, define
\[
    \mathcal{X}(x) := \{X_e^{x(e)} : e\in V,\ x(e)\neq 0\}\quad\text{and}\quad f(x) := H(\mathcal{X}(x)).
\]
    Here, the function $f$ is nonnegative, monotone, and $k$-submodular on $(k+1)^V$.
\end{example}
}

Let \(\Delta_{e, i} f(\mathbf{S})\) be the marginal gain of adding \(e\) to the \(i\)-th set of \(\mathbf{S}\):
    \[
    \Delta_{e, i} f(\mathbf{S}) := f(S_1, \ldots, S_i \cup \{e\}, \ldots, S_k) - f(S_1, \ldots, S_i, \ldots, S_k).
    \]
    Note that \( f \) being monotonically non-decreasing is equivalent to \( \Delta_{e,i} f(\mathbf{S}) \geq 0 \) for all \( \mathbf S \in (k+1)^U \), \( i \in \llbracket k\rrbracket \), and \( e \notin \mathrm{supp}(\mathbf{S}) \), where we define $\mathrm{supp}(\mathbf{S}) := \cup_{i=1}^k S_i$.
A function \( f \) is said to be \textbf{pairwise monotonically non-decreasing} (resp.~\textbf{non-increasing}) if  
    \[
    \Delta_{e, i} f(\mathbf S) + \Delta_{e, j} f(\mathbf S) \geq (\textrm{resp.}\, \leq)\, 0
    \]
    for all \( \mathbf S \in (k+1)^U \), \( e \notin \mathrm{supp}(\mathbf S) \), and \( i, j \in \llbracket k\rrbracket \) such that \( i \neq j \).
A function \( f \) is said to be \textbf{orthant submodular} (resp.~\textbf{orthant supermodular}) if
\begin{align}\label{eq:ort_submod}
    \Delta_{e,i}f(\mathbf S) \geq (\textrm{resp.}\, \leq) \Delta_{e,i}f(\mathbf T)
\end{align}
for all \( i \in \llbracket k\rrbracket \) and \( \mathbf S, \mathbf T \in (k+1)^U \) such that \( \mathbf S \preceq \mathbf T \), \( e \notin \mathrm{supp}(\mathbf T) \).

The following result that we recall characterizes $k$-submodularity.

\begin{theorem}[Characterization of $k$-submodularity (Theorem $7$ of \cite{MR3549595})]\label{thm:k-submod}
    A function $f$ is $k$-submodular (resp.~$k$-supermodular) if and only if $f$ is both orthant submodular (resp.~supermodular) and pairwise monotonically non-decreasing (resp.~non-increasing).
\end{theorem}

The next two results \textcolor{black}{relate} the sum of individually supermodular or submodular functions to $k$-supermodularity or $k$-submodularity respectively \textcolor{black}{using Theorem~\ref{thm:k-submod}.}

\begin{lemma}\label{lem:sum_nonincr_supmod}
    Let $F: (k+1)^U \to \mathbb{R}$ defined to be
    $$F(\mathbf{S}) = F(S_1, \ldots, S_k) := \sum_{i=1}^k F_i(S_i)$$
    be the sum of $k$ monotonically non-increasing and supermodular functions $(F_i)_{i=1}^k$ with $F_i : 2^U \to \mathbb{R}$ for all $i \in \llbracket k \rrbracket$. Then $F$ is $k$-supermodular.
\end{lemma}
\begin{proof}
    Throughout this proof, let $i \neq j \in \llbracket k \rrbracket$. First, we seek to prove that $F$ is pairwise monotonically non-increasing, in which case we aim to show $\Delta_{e, i} F(\mathbf S) + \Delta_{e, j} F(\mathbf S) \leq 0$ for $e\notin \mathrm{supp}(\mathbf{S})$:
    $$\Delta_{e, i} F(\mathbf S) + \Delta_{e, j} F(\mathbf S) = (F_i(S_i \cup \{e\}) - F_i(S_i)) + (F_j(S_j \cup \{e\}) - F_j(S_i)) \leq 0, $$ given that $F_i, F_j$ are both monotonically non-increasing. Next, we seek to show that $F$ is orthant supermodular, in which case we aim to show that $\Delta_{e,i}F(\mathbf S) \leq \Delta_{e,i}F(\mathbf T)$ for any $\mathbf{S} \preceq \mathbf{T}$ and $e\notin \mathrm{supp}(\mathbf{T})$:
    $$\Delta_{e,i}F(\mathbf S) - \Delta_{e,i}F(\mathbf T) = (F_i(S_i \cup \{e\}) - F_i(S_i)) - (F_i(T_i \cup \{e\}) - F_i(T_i)) \leq 0,$$
    given that $F_i$ is supermodular.
    Therefore, $F$ is $k$-supermodular given that it is pairwise monotonically non-increasing and orthant supermodular using Theorem \ref{thm:k-submod}. 
\end{proof}

\begin{corollary}\label{lem:sum_nondecr_submod}
    Let $G: (k+1)^U \to \mathbb{R}$ defined to be
    $$G(\mathbf{S}) = G(S_1, \ldots, S_k) := \sum_{i=1}^k G_i(S_i)$$
    be the sum of $k$ monotonically non-decreasing and submodular functions $(G_i)_{i=1}^k$ with $G_i : 2^U \to \mathbb{R}$ for all $i \in \llbracket k \rrbracket$. Then $G$ is $k$-submodular.
\end{corollary}

\begin{proof}
    By applying Lemma \ref{lem:sum_nonincr_supmod} to $-G$, we see that $-G$ is $k$-supermodular, which is equivalent to $G$ being $k$-submodular.
\end{proof}

\textcolor{black}{In the following sections, we shall show that some natural orthant submodular structures arise in the information theory of multivariate Markov chains. Hence,}
we aim to prove a generalized version of Theorem \ref{thm:monotonize}, that transforms a given constrained orthant submodular function into a \textcolor{black}{monotonically non-decreasing} $k$-submodular function. \textcolor{black}{By applying such transformation, we shall propose Algorithm~\ref{alg:generalized_distorted_greedy} in the subsequent subsection to approximately solve the $k$-submodular optimization problem with theoretical guarantee.} \textcolor{black}{More precisely,} suppose that we are given $\mathbf{V} \in (k+1)^U$, then, constrained to $\mathbf{V}$, we can transform an orthant submodular function into a $k$-submodular function.
\begin{theorem}[\textcolor{black}{Transforming an orthant submodular function to a $k$-submodular function}]\label{thm:monotonize-k}
    Let $f: (k+1)^U \to \mathbb{R}$ be an orthant submodular function, $\beta \in \mathbb{R}$ and $\mathbf{V} \in (k+1)^U$. then $g: (k+1)^U \preceq \mathbf{V} \to \mathbb{R}$ with $$g(\mathbf{S}) := f(\mathbf{S}) - \beta + \sum_{i=1}^k \sum_{e\in S_i} \left( f(V_1, \ldots, V_i \backslash \{e\}, \ldots, V_k) - f(V_1, \ldots, V_i, \ldots, V_k) \right)$$ is $k$-submodular and monotonically non-decreasing.
\end{theorem}

\begin{proof}
    Suppose that $\mathbf{S} \preceq \mathbf{T}$, $i \in \llbracket k \rrbracket$, and $e \in V_i \backslash T_i$. Since $f$ is orthant submodular, we have $\Delta_{e, i} f(\mathbf{S}) \geq \Delta_{e, i} f(\mathbf{T})$, and hence
    \begin{align*}
        \Delta_{e, i} g(\mathbf{S}) &= \Delta_{e, i} f(\mathbf{S}) + f(V_1, \ldots, V_i \backslash \{e\}, \ldots, V_k) - f(V_1, \ldots, V_i, \ldots, V_k) \\
        & \geq \Delta_{e, i} f(\mathbf{T}) + \Delta_{e, i} \sum_{j=1}^k \sum_{u \in T_j} \left(f(V_1, \ldots, V_j \backslash \{u\}, \ldots, V_k) - f(V_1, \ldots, V_j, \ldots, V_k)\right)\\
        &= \Delta_{e, i} g(\mathbf{T}).
    \end{align*}
    This gives $g$ is orthant submodular. 

    To prove the orthant monotonicity, we choose $\mathbf{S} \in (k+1)^U$, $i \in \llbracket k \rrbracket$, and $e \in V_i \backslash S_i$. From the orthant submodularity of $f$, since $S_i \subseteq V_i \backslash \{e\}$, we have 
    $$\Delta_{e, i} g(\mathbf{S}) = \Delta_{e, i} f(\mathbf{S}) - \left(f(V_1, \ldots, V_i, \ldots, V_k) - f(V_1, \ldots, V_i \backslash \{e\}, \ldots, V_k)\right) \geq 0.$$
    Therefore $g$ is monotonically non-decreasing, which implies that $g$ is pairwise monotonically non-decreasing, and hence $g$ is $k$-submodular.
\end{proof}

In the remaining of this subsection and also the coming subsection, we recall a few classical submodular optimization algorithms.

To maximize a monotonically non-decreasing submodular function, one can apply a heuristic greedy algorithm (see Section 4 of \cite{MR503866}) with $(1 - e^{-1})$-approximation guarantee. For non-monotone submodular functions, we recall a local search algorithm (Algorithm~\ref{alg:local_search}) that comes along with an approximation guarantee (Theorem \ref{thm:lb_local_search}).

\begin{algorithm}
\caption{\textbf{Local Search Algorithm (Section 3.1 of \cite{feige2011maximizing})}}
\label{alg:local_search}
\begin{algorithmic}[1]
\Require Ground set $U$ with $|U|=d$, submodular function $f$, positive $\epsilon > 0$
\State Initialize $S \gets \{e\}$, where $f(\{e\})$ is the maximum over all singletons $e \in U$
\While{ $\exists \, a \in U\backslash S$ such that $f(S \cup \{a\}) \geq (1 + \epsilon / d^2) f(S)$}
    \State $S \gets S\cup \{a\}$
\EndWhile
\If{$\exists \, a\in S$ such that $f(S \backslash \{a\}) \geq (1 + \epsilon / d^2) f(S)$}
    \State $S \gets S\backslash \{a\}$
    \State Go back to line 2
\EndIf
\State \textbf{Output}: $f(S)$ and $f(U \backslash S)$
\end{algorithmic}
\end{algorithm}

\begin{theorem}[Approximation guarantee of Algorithm~\ref{alg:local_search} \textcolor{black}{(Theorem 3.4 of \cite{feige2011maximizing})}]\label{thm:lb_local_search}
    Algorithm~\ref{alg:local_search} is a $\left(\frac{1}{3} - \frac{\epsilon}{d}\right)$-approximation algorithm for maximizing non-negative submodular functions, and $\left(\frac{1}{2} - \frac{\epsilon}{d}\right)$-approximation algorithm for maximizing non-negative symmetric submodular functions. The time complexity of Algorithm~\ref{alg:local_search} is $\mathcal{O} \left(\frac{1}{\epsilon} d^3 \log d\right)$.
\end{theorem}

\subsection{Distorted greedy algorithms to maximize the difference of a submodular function and a modular function}\label{subsec:distortgreedy}

In this paper, \textcolor{black}{in view of Theorem~\ref{thm:monotonize},} some functions we are interested in optimizing can be written as a difference of a \textcolor{black}{monotonically non-decreasing} submodular function and a modular function. In this section, we shall consider maximizing the difference of a monotonically non-decreasing submodular $g$ and a modular $c$ on the ground set $U$ with cardinality constraint being at most $m \in \mathbb{N}$. Precisely, we consider the problem
\begin{align*}
    \max_{S \subseteq U;~ |S| \leq m} g(S) - c(S),
\end{align*}
and 
\begin{align*}
    \mathrm{OPT} = \mathrm{OPT}(g,c,U, m) := \argmax_{S \subseteq U;~ |S| \leq m} g(S) - c(S). 
\end{align*}

In this setting, a distorted greedy algorithm (Algorithm \ref{alg:distorted_greedy}) has been proposed along with a theoretical lower bound \cite{pmlr-v97-harshaw19a}.

\begin{algorithm}
\caption{\textbf{Distorted greedy algorithm for maximizing the difference between a monotonically non-decreasing submodular function and a modular function}}\label{alg:distorted_greedy}
\begin{algorithmic}[1]
\Require monotonically non-decreasing submodular $g$ with $g(\emptyset) \geq 0$, non-negative modular $c$, cardinality $m$, ground set $U$
\State Initialize $S_0 \gets \emptyset$
\For{$i = 0$ to $m-1$}
    \State $e_i \gets \argmax\limits_{e \in U} \left\{ \left( 1 - \frac{1}{m} \right)^{m - (i+1)} (g(S_i \cup \{e\}) - g(S_i)) - c(\{e\}) \right\}$
    \If{$\left( 1 - \frac{1}{m} \right)^{m - (i+1)} (g(S_i \cup \{e_i\}) - g(S_i)) - c(\{e_i\}) > 0$}
        \State $S_{i+1} \gets S_i \cup \{e_i\}$
    \Else
        \State $S_{i+1} \gets S_i$
    \EndIf
\EndFor
\State \textbf{Output:} $S_m$. 
\end{algorithmic}
\end{algorithm}

\begin{theorem}[Lower bound for distorted greedy algorithm \textcolor{black}{(Theorem 3 of \cite{pmlr-v97-harshaw19a})}]\label{thm:lb_distgrdy}
    Algorithm \ref{alg:distorted_greedy} provides the following lower bound:
    $$g(S_m) - c(S_m) \geq (1 - e^{-1})g(\mathrm{OPT}) - c(\mathrm{OPT}),$$ where $S_m$ is the final output set.
\end{theorem}

\textcolor{black}{We now generalize Algorithm~\ref{alg:distorted_greedy} into multivariate settings in view of Theorem \ref{thm:monotonize-k}.} Let $\mathbf{V} \in (k+1)^U$, and consider maximizing the difference of a monotonically non-decreasing $k$-submodular $g$ and a modular $c$ on the ground set $U$ with cardinality constraint being at most $m \in \mathbb{N}$. Precisely, we consider the problem
\begin{align}\label{eq:mathbfOPT}
    \max_{\mathbf{S} \preceq \mathbf{V};~ |\mathrm{supp}(\mathbf{S})| \leq m} g(\mathbf{S}) - c(\mathbf{S}),
\end{align}
and 
\begin{align*}
    \mathbf{OPT} = \mathbf{OPT}(g,c,U,\mathbf{V},m) := \argmax_{\mathbf{S} \preceq \mathbf{V};~ |\mathrm{supp}(\mathbf{S})| \leq m} g(\mathbf{S}) - c(\mathbf{S}).
\end{align*}

We propose a generalized distorted greedy algorithm (Algorithm \ref{alg:generalized_distorted_greedy}) for solving \eqref{eq:mathbfOPT} \textcolor{black}{with theoretical lower bound in Theorem~\ref{thm:lb_gen_distgrdy}.}
\textcolor{black}{This serves as a natural extension of Algorithm \ref{alg:distorted_greedy} and Theorem \ref{thm:lb_distgrdy} to be used in the subsequent sections to approximately solve $k$-submodular optimization problems in Markov chain theory with theoretical guarantee.}

\begin{algorithm}[h]
\caption{\textbf{Generalized distorted greedy algorithm for maximizing the difference of $k$-submodular function and a modular function}}\label{alg:generalized_distorted_greedy}
\begin{algorithmic}[1]
\Require $k$-submodular monotonically non-decreasing $g$ with $g(\emptyset) \geq 0$, non-negative modular $c$ with $c(\emptyset) = 0$, cardinality $m$, ground set $U$, $\mathbf{V} = (V_1, \ldots, V_k) \in (k+1)^U$. 
\State Initialize $\mathbf{S}_0 = (S_{0,1},\ldots,S_{0,k}) \gets \emptyset$
\For{$i = 0$ to $m-1$}
    \State $(j^*,e^*) \gets \argmax\limits_{j \in \llbracket k \rrbracket, e \in V_j \backslash S_{i, j}} \left\{ \left( 1 - \frac{1}{m} \right)^{m - (i+1)} \Delta_{e,j}g(\mathbf{S}_i) - c(\{e\}) \right\}$
    \If{$\left( 1 - \frac{1}{m} \right)^{m - (i+1)} \Delta_{e^*,j^*}g(\mathbf{S}_i) - c(\{e^*\}) > 0$}
        \State $S_{i+1,j^*} \gets S_{i,j^*} \cup \{e^*\}$ 
    \Else
        \State $S_{i+1,j^*} \gets S_{i,j^*}$
    \EndIf
    \For{$l \neq j^*$}
        \State $S_{i+1,l} \gets S_{i,l}$
    \EndFor
\EndFor
\State \textbf{Output:} $\mathbf{S}_m = (S_{m,1},\ldots,S_{m,k})$.
\end{algorithmic}
\end{algorithm}

{\color{black}
\begin{theorem}[Lower bound for generalized distorted greedy algorithm]
\label{thm:lb_gen_distgrdy}
    Algorithm \ref{alg:generalized_distorted_greedy} provides the following lower bound: $$g(\mathbf{S}_m) - c(\mathbf{S}_m) \geq (1-e^{-1}) g(\mathbf{OPT}) - c(\mathbf{OPT}),$$
    where $\mathbf{S}_m = (S_{m, 1}, \ldots, S_{m, k})$ is the final output set.
\end{theorem}}

{\color{black}
    We first give an example on the application of Algorithm \ref{alg:generalized_distorted_greedy} to demonstrate it is of interest to $k$-submodular maximization problem beyond the context of Markov chains. We consider the sensor placement problem as discussed earlier in Example \ref{eg:sensor_placement}. As $f$ is monotone and $k$-submodular, we choose $g=f$ and $c=0$ in Algorithm~\ref{alg:generalized_distorted_greedy}, where Theorem~\ref{thm:lb_gen_distgrdy} gives a $(1 - e^{-1})$-approximation guarantee.
}

The rest of this section is devoted to giving a lower bound for the generalized distorted greedy algorithm. We assume that $g$ is monotonically non-decreasing, $k$-submodular, $g(\emptyset) \geq 0$, while $c$ is non-negative, modular and $c(\emptyset) = 0$.

In order to prove the lower bound for the generalized distorted greedy algorithm, we first define the distorted objective function $\Phi_i: (k+1)^U \to \mathbb{R}$, for $m \in \mathbb{N}$ and $0 \leq i \leq m-1$, that 
$$\Phi_i (\mathbf{S}) := (1 - m^{-1})^{m-i} g(\mathbf{S}) - c(\mathbf{S}).$$

We also denote $\Psi_i : (k+1)^U \times \llbracket k \rrbracket \times U \to \mathbb{R}$ that
$$\Psi_i (\mathbf{S},j, e) := \max\{0, (1 - m^{-1})^{m - (i+1)} \Delta_{e, j}g(\mathbf{S}) - c(\{e\})\}.$$

\begin{lemma}\label{lem:gen_dist_grdy_lem1}
The difference of the distorted objective function of two iterations can be written as
    $$\Phi_{i+1}(\mathbf{S}_{i+1}) - \Phi_i(\mathbf{S}_i) = \Psi_i (\mathbf{S}_i, j^*,  e^*) + \frac{1}{m} \left(1 - \frac{1}{m}\right)^{m - (i+1)} g(\mathbf{S}_i).$$
\end{lemma}

\begin{proof}
    Similar to Lemma 1 of \cite{pmlr-v97-harshaw19a}, we can show
    \begin{align*}
        \Phi_{i+1}(\mathbf{S}_{i+1}) - \Phi_i(\mathbf{S}_i) &= \left(1-\dfrac{1}{m}\right)^{m- (i+1)} g(\mathbf{S}_{i+1}) - c(\mathbf{S}_{i+1}) - \left(1-\dfrac{1}{m}\right)^{m-i} g(\mathbf{S}_i) + c(\mathbf{S}_i)\\
        &= \left(1-\dfrac{1}{m}\right)^{m- (i+1)} g(\mathbf{S}_{i+1}) - c(\mathbf{S}_{i+1}) - \left(1-\dfrac{1}{m}\right)^{m- (i+1)} \left(1-\dfrac{1}{m}\right) g(\mathbf{S}_i) + c(\mathbf{S}_i)\\
        &= \left(1-\dfrac{1}{m}\right)^{m- (i+1)} (g(\mathbf{S}_{i+1}) - g(\mathbf{S}_i)) - (c(\mathbf{S}_{i+1}) - c(\mathbf{S}_i))\\
        &\quad + \frac{1}{m} \left(1-\dfrac{1}{m}\right)^{m-(i+1)} g(\mathbf{S}_i).
    \end{align*}
    If $(1-m^{-1})^{m - (i+1)} \Delta_{e^*, j^*}g(\mathbf{S}) - c(\{e^*\}) > 0$, then $e^*$ is added to the solution set. In the algorithm we have $e^* \in V_{j^*} \backslash S_{i, j^*}$, $g(\mathbf{S}_{i+1}) - g(\mathbf{S}_i) = \Delta_{e^*, j^*} g(\mathbf{S}_i)$, $c(\mathbf{S}_{i+1}) - c(\mathbf{S}_i) = c(\{e^*\})$, hence
    $$\Phi_{i+1}(\mathbf{S}_{i+1}) - \Phi_i(\mathbf{S}_i) = \Psi_i(\mathbf{S}_i, j^*, e^*) + \frac{1}{m} \left(1-\dfrac{1}{m}\right)^{m-(i+1)} g(\mathbf{S}_i).$$
    If $(1 - m^{-1})^{m - (i+1)} \Delta_{e^*, j^*}g(\mathbf{S}) - c(\{e_i\}) \leq 0$, the algorithm does not add $e^*$ into the solution set, hence $\mathbf{S}_{i+1} = \mathbf{S}_i$. In this case, we also have
    $$\Phi_{i+1}(\mathbf{S}_{i+1}) - \Phi_i(\mathbf{S}_i) = 0 + \frac{1}{m} \left(1-\dfrac{1}{m}\right)^{m-(i+1)} g(\mathbf{S}_i) = \Psi_i(\mathbf{S}_i, j^*, e^*) + \frac{1}{m} \left(1-\dfrac{1}{m}\right)^{m-(i+1)} g(\mathbf{S}_i).$$
    Summarizing these two cases, we see that
    $$\Phi_{i+1}(\mathbf{S}_{i+1}) - \Phi_i(\mathbf{S}_i) = \Psi_i (\mathbf{S}_i, j^*,  e^*) + \frac{1}{m} \left(1-\dfrac{1}{m}\right)^{m - (i+1)} g(\mathbf{S}_i).$$
\end{proof}

\begin{lemma} \label{lem:gen_dist_grdy_lem2}
A lower bound for $\Psi_i$ is
    $$\Psi_i (\mathbf{S}_i, j^*, e^*) \geq \frac{1}{m} \Bigg(\left(1-\dfrac{1}{m}\right)^{m - (i+1)} \big(g(\mathbf{OPT}) - g(\mathbf{S}_i)\big) - c(\mathbf{OPT})\Bigg).$$
\end{lemma}

\begin{proof}
    For $j \in \llbracket k \rrbracket$, let
    \begin{align*}
        U_{i,j} &:= (V_j \backslash S_{i,j}) \cap \mathrm{OPT}_j, \\
        U_i &:= \bigcup_{j=1}^k U_{i,j}, \\
        \mathbf{U}_i &:= (U_{i,1},U_{i,2},\ldots,U_{i,k}),
    \end{align*}
    and hence
    \begin{align}\label{eq:SijcupOPTj}
        S_{i,j} \cup U_{i,j} = S_{i,j} \cup \mathrm{OPT}_j.
    \end{align}
    We then have
    \begin{align*}
        m \Psi_i (\mathbf{S}_i,  j^*, e^*) &= m \max_{j \in \llbracket k \rrbracket, e\in V_j \backslash S_{i, j}}\Bigg\{0, \left(1-\dfrac{1}{m}\right)^{m - (i+1)} \Delta_{e, j}g(\mathbf{S}_i) - c(\{e\})\Bigg\}\\
        & \geq |\mathrm{supp}(\mathbf{OPT})| \max_{j \in \llbracket k \rrbracket, e\in U_{i,j}} \Bigg\{0, \left(1-\dfrac{1}{m}\right)^{m - (i+1)} \Delta_{e, j}g(\mathbf{S}_i) - c(\{e\})\Bigg\} \\
        & \geq |U_i| \max_{j \in \llbracket k \rrbracket, e\in U_{i,j}} \Bigg\{\left(1-\dfrac{1}{m}\right)^{m - (i+1)} \Delta_{e, j}g(\mathbf{S}_i) - c(\{e\})\Bigg\}\\
        & \geq \sum_{j=1}^k \sum_{e \in U_{i,j}} \Bigg(\left(1-\dfrac{1}{m}\right)^{m - (i+1)} \Delta_{e, j}g(\mathbf{S}_i) - c(\{e\})\Bigg)\\
        &= \left(1-\dfrac{1}{m}\right)^{m - (i+1)} \sum_{j=1}^k \sum_{e\in U_{i,j}}\Delta_{e, j}g(\mathbf{S}_i) - c(\mathbf{U}_i) \\
        &\geq \left(1-\dfrac{1}{m}\right)^{m - (i+1)} \sum_{j=1}^k \sum_{e\in U_{i,j}}\Delta_{e, j}g(\mathbf{S}_i) - c(\mathbf{OPT}),
    \end{align*}
    where the last inequality follows from the fact that $c$ is non-negative. Then, the desired result follows if we show that 
    \begin{align*}
        \sum_{j=1}^k \sum_{e\in U_{i,j}}\Delta_{e, j}g(\mathbf{S_i}) \geq g(\mathbf{OPT}) - g(\mathbf{S}_i).
    \end{align*} 
    Since $g$ is orthant submodular, by Lemma 1.1 of \cite{lee2010submodular}, we have
    \begin{align*}
        \sum_{e\in U_{i,j}}\Delta_{e, j}g(\mathbf{S_i}) \geq g(S_{i, 1}, \ldots, S_{i, j-1}, S_{i, j} \cup U_{i,j}, S_{i, j+1}, \ldots, S_k) - g(\mathbf{S}_i),
    \end{align*}
    and hence it further suffices to prove
    \begin{align}\label{eq:sum}
        \sum_{j=1}^k g(S_{i, 1}, \ldots, S_{i, j-1}, S_{i, j} \cup U_{i,j}, S_{i, j+1}, \ldots, S_k) \geq g(\mathbf{OPT}) + (k-1) g(\mathbf{S}_i).
    \end{align}
    Since $g$ is $k$-submodular, then $$g(\mathbf{X}) + g(\mathbf{Y}) \geq g(\mathbf{X} \sqcup \mathbf{Y}) + g(\mathbf{X} \sqcap \mathbf{Y}),$$ for any $\mathbf{X}, \mathbf{Y} \in (k+1)^U$.
    We seek to apply this definition to update each of the $k$ coordinates by adding $(U_{i,j})_{j=1}^k$ sequentially. For the first step, we have
    \begin{align*}
        &\quad g(S_{i, 1} \cup U_{i,1}, S_{i,2}, \ldots, S_{i, k}) + g(S_{i, 1}, S_{i, 2} \cup U_{i,2}, S_{i, 3}, \ldots, S_{i, k}) \\
        &\geq g((S_{i, 1} \cup U_{i,1}) \backslash (\cup_{l\neq 1}^k S_{i,l} \cup U_{i,2}), (S_{i, 2} \cup U_{i,2}) \backslash (\cup_{l\neq 2}^k S_{i,l} \cup U_{i,1}),  S_{i, 3}, \ldots, S_{i, k}) + g(\mathbf{S}_i) \\
        &= g(S_{i, 1} \cup U_{i,1},S_{i, 2} \cup U_{i,2},S_{i,3},\ldots,S_{i,k}) + g(\mathbf{S}_i),
    \end{align*}
    where the last equality uses the fact that with $n \in \llbracket k \rrbracket$,
    \begin{align*}
        (S_{i, n} \cup U_{i,n}) = (S_{i, n} \cup U_{i,n}) \backslash (\cup_{l\neq n}^k (S_{i,l} \cup U_{i,l})).
    \end{align*}
    In the $n$-th step with $n \in \llbracket k \rrbracket$, we thus have
    \begin{align*}
        &\quad g(S_{i, 1}\cup U_{i,1}, \ldots, S_{i, n}\cup U_{i,n}, \ldots, S_{i, k}) + g(S_{i, 1}, \ldots, S_{i, n}, S_{i, n+1}\cup U_{i,n+1}, \ldots, S_{i, k})\\
        &\geq g(S_{i, 1} \cup U_{i,1}, \ldots, S_{i, n+1} \cup U_{i,n+1}, \ldots, S_{i, k}) + g(\mathbf{S}_i).
    \end{align*}
    Repeating the above analysis leads to $$\sum_{j=1}^k g(S_{i, 1}, \ldots, S_{i, j-1}, S_{i, j} \cup U_{i,j}, S_{i, j+1}, \ldots, S_k) \geq g(\mathbf{S}_i \sqcup \mathbf{U}_i) + (k-1) g(\mathbf{S}_i).$$
    Finally, using the assumption that $g$ is monotonically non-decreasing and $\mathbf{OPT} \preceq \mathbf{S}_i \sqcup \mathbf{U}_i$ in view of \eqref{eq:SijcupOPTj}, we have $$\sum_{j=1}^k g(S_{i, 1}, \ldots, S_{i, j-1}, S_{i, j} \cup U_{i,j}, S_{i, j+1}, \ldots, S_k) \geq g(\mathbf{OPT}) + (k-1) g(\mathbf{S}_i),$$
    and hence \eqref{eq:sum} holds.
\end{proof}

Finally, we prove the lower bound for the generalized distorted greedy algorithm \textcolor{black}{in Theorem~\ref{thm:lb_gen_distgrdy}.}

\begin{proof}[\textcolor{black}{Proof of Theorem~\ref{thm:lb_gen_distgrdy}}]
    According to our assumptions, we have $$\Phi_0 (\mathbf{S}_0) = \left(1-\dfrac{1}{m}\right)^m g(\emptyset) - c(\emptyset) \geq 0$$ and $$\Phi_m(\mathbf{S}_m) = \left(1-\dfrac{1}{m}\right)^0 g(\mathbf{S}_m) - c(\mathbf{S}_m) = g(\mathbf{S}_m) - c(\mathbf{S}_m).$$
    Therefore, we have 
    \begin{equation}\label{eq:lw_bd}
        g(\mathbf{S}_m) - c(\mathbf{S}_m) \geq \Phi_m (\mathbf{S}_m) - \Phi_0 (\mathbf{S}_0) = \sum_{i=0}^{m-1} \Phi_{i+1}(\mathbf{S}_{i+1}) - \Phi_i (\mathbf{S}_i).
    \end{equation}
    We apply Lemma \ref{lem:gen_dist_grdy_lem1} and \ref{lem:gen_dist_grdy_lem2} to yield
    \begin{align*}
        \Phi_{i+1}(\mathbf{S}_{i+1}) - \Phi_i(\mathbf{S}_i) &= \Psi_i (\mathbf{S}_i, j^*,  e^*) + \frac{1}{m} \left(1-\dfrac{1}{m}\right)^{m - (i+1)} g(\mathbf{S}_i)\\
        &\geq \frac{1}{m} \left(1-\dfrac{1}{m}\right)^{m - (i+1)} g(\mathbf{OPT}) - \frac{1}{m} c(\mathbf{OPT}).
    \end{align*}
    We plug the above bound into \eqref{eq:lw_bd} to obtain
    \begin{align*}
        g(\mathbf{S}_m) - c(\mathrm{supp}(\mathbf{S}_m)) &\geq \sum_{i=0}^{m-1} \Bigg[\frac{1}{m} \left(1-\dfrac{1}{m}\right)^{m - (i+1)} g(\mathbf{OPT}) - \frac{1}{m} c(\mathbf{OPT})\Bigg]\\
        &= \Bigg[\frac{1}{m} \sum_{i=0}^{m-1}\left(1-\dfrac{1}{m}\right)^i\Bigg] g(\mathbf{OPT}) - c(\mathbf{OPT})\\
        &= \Bigg(1 - \left(1-\dfrac{1}{m}\right)^m\Bigg) g(\mathbf{OPT}) - c(\mathbf{OPT}) \\
        &\geq (1 - e^{-1})g(\mathbf{OPT}) - c(\mathbf{OPT}).
    \end{align*}
\end{proof}

\section{Submodular maximization of the entropy rate $H(P^{(S)})$}\label{sec:entropy_rate}

Given $P\in \mathcal{L(\mathcal{X})}$ and $m \in \mathbb{N}$, we aim to investigate the following submodular maximization problem with cardinality constraint: 
\begin{align}\label{eq:maxentropyrate}
    \max_{S \subseteq \llbracket d \rrbracket;~ |S| \leq m} H(P^{(S)}).
\end{align}
From Theorem \ref{thm:submod_mc}, the map
$S \mapsto H(P^{(S)})$ is submodular but generally not monotonically non-decreasing. Since the widely-used heuristic greedy algorithm is near-optimal only when the objective submodular function is monotonically non-decreasing (see Section 4 of \cite{MR503866}), in this regard our problem does not have a classical greedy-based approximation guarantee. On the other hand, since $H(P^{(S)}) \geq 0$ and $H(P^{(\emptyset)}) = 0$, if we consider the unconstrained maximization problem of \eqref{eq:maxentropyrate}, we can apply Algorithm~\ref{alg:local_search} with $\left(\frac{1}{3} - \frac{\epsilon}{d}\right)$-approximation guarantee (see Theorem~\ref{thm:lb_local_search}).

Instead, we consider \begin{align*}
    H(P) &= H(\pi \boxtimes P) - H(\pi),
\end{align*}
where we define the edge measure of $P$ with respect to $\pi$ as $(\pi \boxtimes P)(x,y) := \pi(x) P(x,y)$ and \textcolor{black}{$\pi \boxtimes P$ is a probability distribution on $\mathcal{X} \times \mathcal{X}$.}

Then, the map 
\begin{align}\label{eq:map_productform_entropyrate}
    S \mapsto H(P^{(S)}) = H(\pi^{(S)} \boxtimes P^{(S)}) - H(\pi^{(S)})
\end{align} 
can be considered as a monotonically non-decreasing submodular function $H(\pi^{(S)} \boxtimes P^{(S)})$ \textcolor{black}{(see e.g. \cite{F78})} minus a non-negative modular function $H(\pi^{(S)})$ if we assume $\pi$ to be of product form. This fits into the setting of the distorted greedy as in Algorithm \ref{alg:distorted_greedy}, and leads us to Corollary~\ref{cor:productform_entropyrate}.

\begin{corollary}\label{cor:productform_entropyrate}
    Let $P \in \mathcal{L}(\mathcal{X})$ be $\pi$-stationary where $\pi$ is of product form. In Algorithm \ref{alg:distorted_greedy}, we take $g(S) = H(\pi^{(S)} \boxtimes P^{(S)})$, $c(S) = H(\pi^{(S)})$, and $\mathrm{OPT} = \argmax_{S \subseteq \llbracket d \rrbracket;~ |S| \leq m} H(P^{(S)})$. Therefore, Theorem \ref{thm:lb_distgrdy} gives
    \begin{align*}
        H(P^{(S_m)}) \geq (1-e^{-1}) H(\pi^{(\mathrm{OPT})} \boxtimes P^{(\mathrm{OPT})}) - H(\pi^{(\mathrm{OPT})}), 
    \end{align*}
    where $S_m$ is the output of Algorithm \ref{alg:distorted_greedy}.
\end{corollary}

More generally for $P$ with non-product-form $\pi$ as stationary distribution, in view of Theorem \ref{thm:monotonize}, for any $\beta \in \mathbb{R}$ we have a monotonically non-decreasing submodular $g$ given by
\begin{align}\label{eq:gentropyrate}
    g(S) = H(P^{(S)}) - \beta + \sum_{e \in S} (H(P^{(-e)}) - H(P)),
\end{align} 
and we also denote the following modular function
\begin{align}\label{eq:centropyrate}
    c(S) &= - \beta + \sum_{e \in S} (H(P^{(-e)}) - H(P)) \\
         &= - \beta + \sum_{e \in S} (D(P \| P^{(e)} \otimes P^{(-e)}) - H(P^{(e)})). \nonumber
\end{align} 
As $H(P^{(e)}) \leq \log |\mathcal{X}^{(e)}|$, $c$ is ensured to be non-negative if $\beta \leq - \sum_{i=1}^d \log |\mathcal{X}^{(i)}|$. Since 
$$H(P^{(S)}) = g(S) - c(S),$$ we can employ Algorithm \ref{alg:distorted_greedy} to perform distorted greedy maximization with a lower bound.
\begin{corollary}\label{cor:entropyrate}
    Let $P \in \mathcal{L}(\mathcal{X})$ be $\pi$-stationary. In Algorithm \ref{alg:distorted_greedy}, we take $g$ as in \eqref{eq:gentropyrate}, $c$ as in \eqref{eq:centropyrate}, $\beta \leq - \sum_{i=1}^d \log |\mathcal{X}^{(i)}|$, and $\mathrm{OPT} = \argmax_{S \subseteq \llbracket d \rrbracket;~ |S| \leq m} H(P^{(S)})$. Therefore, Theorem \ref{thm:lb_distgrdy} gives
    \begin{align*}
        H(P^{(S_m)}) \geq (1-e^{-1}) g(\mathrm{OPT}) - c(\mathrm{OPT}), 
    \end{align*}
    where $S_m$ is the output of Algorithm \ref{alg:distorted_greedy}.
\end{corollary}

Note that the lower bound of Corollary \eqref{cor:entropyrate} depends on $\beta$ through $g$ and $c$. If $\beta$ is chosen to be too small, then the lower bound might be too loose as the right hand side might be negative.

\subsection{$k$-submodular maximization of the entropy rate of the tensorized keep-$S_i$-in matrices $H(\otimes_{i=1}^k P^{(S_i)})$}

In this subsection, we \textcolor{black}{examine} the following map
\begin{align}\label{map:gen_entropy_rate}
    (k+1)^{\llbracket d \rrbracket} \ni \mathbf{S} = (S_1, \ldots, S_k) \mapsto f(\mathbf{S}) = H(\otimes_{i=1}^k P^{(S_i)}) = \sum_{i=1}^k H(P^{(S_i)}),
\end{align}
and consider maximization problems of the form, for given $\mathbf{V} \in (k+1)^{\llbracket d \rrbracket}$,
\begin{align}
    \max_{\mathbf{S} \preceq \mathbf{V};~ |\mathrm{supp}(\mathbf{S})| \leq m} H(\otimes_{i=1}^k P^{(S_i)}).
\end{align}
\textcolor{black}{In other words, we are interested in designing the partition $\mathbf{S}$ subject to cardinality constraint to maximize the entropy rate of the product chain $\otimes_{i=1}^k P^{(S_i)}$.}
In the special case of $k = 1$ and $\mathbf{V} = \llbracket d \rrbracket$, we recover the problem \eqref{eq:maxentropyrate}.

First, we consider the special case where $P$ is $\pi$-stationary with $\pi$ taking on a product form. Similar to the map~\eqref{eq:map_productform_entropyrate}, we re-write the map~\eqref{map:gen_entropy_rate} as \begin{align}
    \mathbf{S} \mapsto f(\mathbf{S}) = \sum_{i=1}^k H(\pi^{(S_i)} \boxtimes P^{(S_i)}) - \sum_{i=1}^k H(\pi^{(S_i)}).
\end{align}
\textcolor{black}{As $H(\pi^{(S_i)} \boxtimes P^{(S_i)})$ is monotonically non-decreasing and submodular (recall again \cite{F78}), we make use of Corollary~\ref{lem:sum_nondecr_submod} to see that} the following function $g$ is monotonically non-decreasing and $k$-submodular
\begin{align}\label{eq:g_productform_gen_entropyrate}
    g(\mathbf{S}) = \sum_{i=1}^k H(\pi^{(S_i)} \boxtimes P^{(S_i)}).
\end{align}
Since $\pi$ is of product form, we denote the non-negative modular function $c$ as 
\begin{align}\label{eq:c_productform_gen_entropyrate}
    c(\mathbf{S}) = \sum_{i=1}^k H(\pi^{(S_i)}).
\end{align}
Therefore, we have \begin{align*}
    f(\mathbf{S}) = g(\mathbf{S}) - c(\mathbf{S}),
\end{align*}
and the distorted greedy algorithm yields an approximate \textcolor{black}{solution} with a lower bound as in Theorem~\ref{thm:lb_gen_distgrdy}.
\begin{corollary}\label{cor:productform_gen_entropyrate}
    Let $P\in \mathcal{L}(\mathcal{X})$ be $\pi$-stationary where $\pi$ is of product form. In Algorithm~\ref{alg:generalized_distorted_greedy}, we take $g$ as in~\eqref{eq:g_productform_gen_entropyrate} and $c$ as in~\eqref{eq:c_productform_gen_entropyrate}, and $\mathbf{OPT} = \argmax_{\mathbf{S} \preceq \mathbf{V};~ |\mathrm{supp}(\mathbf{S})| \leq m} f(\mathbf{S})$. Then by Theorem~\ref{thm:lb_gen_distgrdy}, we have the following lower bound \begin{align*}
        f(\mathbf{S}_m) = H(\otimes_{i=1}^k P^{(S_{m,i})}) \geq (1 - e^{-1})g(\mathbf{OPT}) - c(\mathbf{OPT}),
    \end{align*}
    where $\mathbf{S}_m = (S_{m, 1}, \ldots, S_{m, k})$ is the output of Algorithm~\ref{alg:generalized_distorted_greedy}.
\end{corollary}
In the special case where $k=1$ and $\mathbf{V} = \llbracket d \rrbracket$, we recover Corollary~\ref{cor:productform_entropyrate}.

Next, we \textcolor{black}{look into the more general case} where $P$ is $\pi$-stationary for general $\pi$ which may not be of product form. We first prove an orthant submodularity result.

\begin{lemma}
    The map \eqref{map:gen_entropy_rate} is orthant submodular.
\end{lemma}

\begin{proof}
    We shall prove that $\Delta_{e, i} f(\mathbf S) \geq \Delta_{e, i} f(\mathbf T)$, where we choose $\mathbf{S} \preceq \mathbf{T}$ and $e \notin \mathrm{supp}(\mathbf{T})$.
    Given the submodularity of $S \mapsto H(P^{(S)})$, we have 
    $$H(P^{(S_i \cup \{e\})}) - H(P^{(S_i)}) \geq H(P^{(T_i \cup \{e\})}) - H(P^{(T_i)}), $$ which is equivalent to $\Delta_{e, i} f(\mathbf S) \geq \Delta_{e, i} f(\mathbf T)$.
\end{proof}

In view of Theorem~\ref{thm:monotonize-k}, since the map \eqref{map:gen_entropy_rate} is orthant submodular, then for any $\beta \in \mathbb R$, if $\mathbf S \preceq \mathbf V$, we have a monotonically non-decreasing $k$-submodular function $g$ given by
\begin{align}\label{eq:g_k_submod_entropyrate}
    g(\mathbf{S}) = \sum_{i=1}^k H(P^{(S_i)}) - \beta + \sum_{i=1}^k \sum_{e\in S_i} (H(P^{(V_i \backslash \{e\})}) - H(P^{(V_i)})),
\end{align}
and we also denote the following modular function
\begin{align}\label{eq:c_k-submod_entropyrate}
    c(\mathbf{S}) &= - \beta + \sum_{i=1}^k \sum_{e\in S_i} (H(P^{(V_i \backslash \{e\})}) - H(P^{(V_i)}))\\
    &= - \beta + \sum_{i=1}^k \sum_{e\in S_i} (D(P^{(V_i)} \| P^{(e)} \otimes P^{(V_i \backslash \{e\})}) - H(P^{(e)})).\nonumber
\end{align}
As $H(P^{(e)}) \leq \log |\mathcal{X}^{(e)}|$, $c$ is ensured to be non-negative if $\beta \leq - \sum_{i=1}^k \sum_{e \in V_i} \log |\mathcal{X}^{(e)}|$. Since 
\begin{align*}
    f(\mathbf{S}) = \sum_{i=1}^k H(P^{(S_i)}) = g(\mathbf{S}) - c(\mathbf{S}),
\end{align*} 
then we can apply Algorithm~\ref{alg:generalized_distorted_greedy} to perform distorted greedy maximization with a guaranteed lower bound.
\begin{corollary}\label{cor:gen_entropyrate}
    Let $P\in \mathcal{L}(\mathcal{X})$ be $\pi$-stationary and $\mathbf{V} \in (k+1)^{\llbracket d \rrbracket}$. In Algorithm~\ref{alg:generalized_distorted_greedy}, we take $g$ as in \eqref{eq:g_k_submod_entropyrate} and $c$ as in \eqref{eq:c_k-submod_entropyrate}, $\beta \leq - \sum_{i=1}^k \sum_{e \in V_i} \log |\mathcal{X}^{(e)}|$, and $\mathbf{OPT} = \argmax_{\mathbf{S} \preceq \mathbf{V};~ |\mathrm{supp}(\mathbf{S})| \leq m} f(\mathbf{S})$. Therefore, Theorem~\ref{thm:lb_gen_distgrdy} gives \begin{align*}
        f(\mathbf{S}_m) = H(\otimes_{i=1}^k P^{(S_{m,i})}) \geq (1 - e^{-1})g(\mathbf{OPT}) - c(\mathbf{OPT}),
    \end{align*}
    where $\mathbf{S}_m = (S_{m,1},S_{m,2},\ldots,S_{m,k})$ is the output of Algorithm~\ref{alg:generalized_distorted_greedy}.
\end{corollary}
Note that the lower bound of Corollary \ref{cor:gen_entropyrate} depends on $\beta$ through $g$ and $c$. If $\beta$ is chosen to be too small, then the lower bound might be too loose as the right hand side might be negative.

\section{Submodular optimization of distance to factorizability $D(P \| P^{(S)} \otimes P^{(-S)})$}\label{sec:dist2fact}
\subsection{Submodular minimization of the distance to factorizability}
For $$2^{\llbracket d \rrbracket} \ni S \mapsto D(P\| P^{(S)} \otimes P^{(-S)}),$$
we first recall that this map is submodular (see Theorem \ref{thm:submod_mc}). Since $D(P\| P^{(S)} \otimes P^{(-S)}) = D(P\| P^{(-S)} \otimes P^{(S)})$, then this map is also symmetric. In this case, there exists an algorithm for minimizing non-negative symmetric submodular functions (see Theorem 14.25 of \cite{korte2011combinatorial}) that gives
\begin{align*}
    S^* \in \argmin_{\emptyset \neq S \subset \llbracket d \rrbracket;~ |S| \leq m} D(P\| P^{(S)} \otimes P^{(-S)})
\end{align*}
with time complexity $\mathcal{O}(d^3 \theta)$. Here, $\theta$ denotes the worst case time needed to evaluate $D(P\| P^{(S)} \otimes P^{(-S)})$ for any given subset $S$.

\subsection{Submodular maximization of the distance to factorizability}
Given $P \in \mathcal{L}(\mathcal{X})$ and $m \in \mathbb{N}$, \textcolor{black}{in this subsection we investigate the following submodular (recall Theorem \ref{thm:submod_mc}) maximization problem subjected to a cardinality constraint} 
\begin{align}\label{eq:max_dist2fact}
    \max_{S \subseteq \llbracket d \rrbracket;~ |S| \leq m} D(P \| P^{(S)} \otimes P^{(-S)}).
\end{align}

Since $D(P \| P^{(S)} \otimes P^{(-S)}) \geq 0$ and $D(P \| P^{(\emptyset)} \otimes P^{(\llbracket d\rrbracket)}) = 0$, if we consider the unconstrained version of \eqref{eq:max_dist2fact}, we can apply Algorithm~\ref{alg:local_search} with $\left(\frac{1}{2} - \frac{\epsilon}{d}\right)$-approximation guarantee (see Theorem~\ref{thm:lb_local_search}) since $D(P \| P^{(S)} \otimes P^{(-S)})$ is symmetric. 

In view of Theorem \ref{thm:monotonize}, \textcolor{black}{since $D(\cdot \| \cdot) \geq 0$,} we choose $\beta = 0$ and take
\begin{align}\label{eq:g_dist2fact}
    g(S) = D(P \| P^{(S)} \otimes P^{(-S)}) + \sum_{e\in S} D(P \| P^{(-e)} \otimes P^{(e)}),
\end{align} 
which is submodular and monotonically non-decreasing. In this case, we also take the modular and non-negative function $c$ to be\begin{align}\label{eq:c_dist2fact}
    c(S) = \sum_{e\in S} D(P \| P^{(-e)} \otimes P^{(e)}).
\end{align}
Therefore, \begin{align*}
    D(P\| P^{(S)} \otimes P^{(-S)}) = g(S) - c(S)
\end{align*}
\textcolor{black}{can be expressed as the difference of a non-negative, submodular, monotonically non-decreasing function $g$ and a non-negative modular function $c$. As a result, Algorithm~\ref{alg:distorted_greedy} can be applied to maximize $D(P \| P^{(S)} \otimes P^{(-S)})$ that yields a theoretical lower bound in the following Corollary.}
\begin{corollary}\label{cor:dist2fact}
    Let $P\in \mathcal{L}(\mathcal{X})$ be $\pi$-stationary. In Algorithm~\ref{alg:distorted_greedy}, we take $g$ as in \eqref{eq:g_dist2fact} and $c$ as in \eqref{eq:c_dist2fact}, and $\mathrm{OPT} = \argmax_{S \subseteq \llbracket d\rrbracket;~ |S| \leq m} D(P \| P^{(S)} \otimes P^{(-S)})$. By Theorem~\ref{thm:lb_distgrdy}, we have \begin{align*}
        D(P \| P^{(S_m)} \otimes P^{(-S_m)}) \geq (1 - e^{-1}) g(\mathrm{OPT}) - c(\mathrm{OPT}),
    \end{align*}
    where $S_m$ is the final output set of Algorithm~\ref{alg:distorted_greedy}.
\end{corollary}

\subsection{$k$-submodular maximization of distance to factorizability of the tensorized keep-$S_i$-in matrices $D(P \| P^{(S_1)} \otimes \ldots \otimes P^{(S_k)} \otimes P^{(-\cup_{i=1}^k S_i)})$}

In this section, we \textcolor{black}{study} the following map 
\begin{align}\label{eq:map_generalized_dist2fact}
    (k+1)^{\llbracket d \rrbracket} \ni \mathbf{S} \mapsto f(\mathbf{S}) = D(P \| P^{(S_1)} \otimes \ldots \otimes P^{(S_k)} \otimes P^{(-\cup_{i=1}^k S_i)}),
\end{align}
We consider the maximization problem of the form, for given $\mathbf{V} \in (k+1)^{\llbracket d \rrbracket}$, \begin{align}
    \max_{\mathbf{S} \preceq \mathbf{V};~ |\mathrm{supp}(\mathbf{S})| \leq m} D(P \| P^{(S_1)} \otimes \ldots \otimes P^{(S_k)} \otimes P^{(-\cup_{i=1}^k S_i)}).
\end{align}
\textcolor{black}{In other words, we aim to select a coordinate partition $\mathbf{S}$ to determine a factorized chain being the ``furthest'' from the original chain $P$ in terms of KL divergence.}
In the special case of $k=1$ and $\mathbf{V} = \llbracket d \rrbracket$, we recover problem~\eqref{eq:max_dist2fact}.

\begin{lemma}
    The map \eqref{eq:map_generalized_dist2fact} is orthant submodular.
\end{lemma}
\begin{proof}
    We shall prove that $\Delta_{e, i} f(\mathbf S) \geq \Delta_{e, i} f(\mathbf T)$, where we choose $\mathbf{S} \preceq \mathbf{T}$ and $e \notin \mathrm{supp}(\mathbf{T})$. We compute that
    \begin{align*}
        \Delta_{e, i} f(\mathbf S) - \Delta_{e, i} f(\mathbf T) &= H(P^{(S_i \cup \{e\})}) - H(P^{(S_i)}) + H(P^{(-\mathrm{supp}(\mathbf{S}) \cup \{e\})}) - H(P^{(-\mathrm{supp}(\mathbf{S}))})\\
        &\quad - H(P^{(T_i \cup \{e\})}) + H(P^{(T_i)}) - H(P^{(-\mathrm{supp}(\mathbf{T}) \cup \{e\})}) + H(P^{(-\mathrm{supp}(\mathbf{T}))}) \\
        &= \Big[\big(H(P^{(S_i \cup \{e\})}) - H(P^{(S_i)})\big) - \big(H(P^{(T_i \cup \{e\})}) - H(P^{(T_i)})\big)\Big] \\
        &\quad + \Big[\big(H(P^{(-\mathrm{supp}(\mathbf{T}))}) - H(P^{(-\mathrm{supp}(\mathbf{T}) \cup \{e\})})\big) \\
        &\quad - \big(H(P^{(-\mathrm{supp}(\mathbf{S}))}) - H(P^{(-\mathrm{supp}(\mathbf{S}) \cup \{e\})})\big)\Big],
    \end{align*}
    where each of the two terms above are non-negative given the submodularity of $S \mapsto H(P^{(S)})$ (recall Theorem \ref{thm:submod_mc}).
\end{proof}

\textcolor{black}{However, since the map \eqref{eq:map_generalized_dist2fact} is in general not pairwise monotone, then by Theorem \ref{thm:k-submod}, the map \eqref{eq:map_generalized_dist2fact} is in general not $k$-submodular.}
In view of Theorem~\ref{thm:monotonize-k}, since the map \eqref{eq:map_generalized_dist2fact} is orthant submodular, for any $\beta \in \mathbb R$, if $\mathbf{S} \preceq \mathbf{V}$, we have a monotonically non-decreasing $k$-submodular function given by 
\begin{align}\label{eq:g_gen_dist2fact}
    g(\mathbf{S}) &= f(\mathbf{S}) - \beta + \sum_{i=1}^k \sum_{e \in S_i} \Big[D(P\| P^{(V_1)} \otimes \ldots \otimes P^{(V_i \backslash \{e\})} \otimes \ldots \otimes P^{(V_k)} \otimes P^{(-\mathrm{supp}(\mathbf{V}) \backslash \{e\})}) \\
    &\quad - D(P\| P^{(V_1)} \otimes \ldots \otimes P^{(V_i)} \otimes \ldots \otimes P^{(V_k)} \otimes P^{(-\mathrm{supp}(\mathbf{V}))})\Big] \nonumber\\
    &= f(\mathbf{S}) - \beta + \sum_{i=1}^k \sum_{e \in S_i}\Big[H(P^{(V_i \backslash \{e\})}) + H(P^{(-\mathrm{supp}(\mathbf{V}) \backslash \{e\})}) - H(P^{(V_i)}) - H(P^{(-\mathrm{supp}(\mathbf{V}))})\Big] \nonumber \\
    &= f(\mathbf{S}) - \beta + \sum_{i=1}^k \sum_{e \in S_i} \Big[D(P^{(V_i)} \| P^{(V_i \backslash \{e\})} \otimes P^{(e)}) - D(P^{(-\mathrm{supp}(\mathbf{V}) \backslash \{e\})} \| P^{(-\mathrm{supp}(\mathbf{V}))} \otimes P^{(e)}) \nonumber \Big], \nonumber
\end{align}
and we also obtain the following modular function
\begin{align}\label{eq:c_gen_dist2fact}
    c(\mathbf{S}) = - \beta + \sum_{i=1}^k \sum_{e \in S_i} \Big[D(P^{(V_i)} \| P^{(V_i \backslash \{e\})} \otimes P^{(e)}) - D(P^{(-\mathrm{supp}(\mathbf{V}) \backslash \{e\})} \| P^{(-\mathrm{supp}(\mathbf{V}))} \otimes P^{(e)})\Big].
\end{align}
Thus, if we choose \begin{align*}
    \beta \leq - \sum_{i=1}^k \sum_{e \in V_i} \left(H(P^{(-\mathrm{supp}(\mathbf{V}) \backslash \{e\})}) +  H(P^{(e)})\right),
\end{align*}
then $c$ is non-negative. \textcolor{black}{With these choices, $f$ can then be written as a difference of $g$ and $c$ given by}
\begin{align*}
    f(\mathbf{S}) = D(P \| P^{(S_1)} \otimes \ldots \otimes P^{(S_k)} \otimes P^{(-\cup_{i=1}^k S_i)}) = g(\mathbf{S}) - c(\mathbf{S}).
\end{align*}
We can then apply Algorithm~\ref{alg:generalized_distorted_greedy} to perform distorted greedy maximization with a lower bound.

\begin{corollary}\label{cor:gen_dist2fact}
    Let $P \in \mathcal{L}(\mathcal{X})$ be $\pi$-stationary and $\mathbf{V} \in (k+1)^{\llbracket d \rrbracket}$. In Algorithm~\ref{alg:generalized_distorted_greedy}, we take $g$ as in \eqref{eq:g_gen_dist2fact} and $c$ as in \eqref{eq:c_gen_dist2fact}. We choose \begin{align*}
         \beta \leq - \sum_{i=1}^k \sum_{e \in V_i} \left(H(P^{(-\mathrm{supp}(\mathbf{V}) \backslash \{e\})}) +  H(P^{(e)})\right),
    \end{align*} and let $\mathbf{OPT} = \argmax_{\mathbf{S} \preceq \mathbf{V};~ |\mathrm{supp}(\mathbf{S})| \leq m} f(\mathbf{S})$. Therefore, Theorem~\ref{thm:lb_gen_distgrdy} gives 
    \begin{align*}
        f(\mathbf{S}_m) = D(P \| P^{(S_{m, 1})} \otimes \ldots \otimes P^{(S_{m, k})} \otimes P^{(-\cup_{i=1}^k S_{m, i})}) \geq (1 - e^{-1})g(\mathbf{OPT}) - c(\mathbf{OPT}),
    \end{align*}
    where $\mathbf{S}_m = (S_{m,1}, \ldots, S_{m,k})$ is the output of Algorithm \ref{alg:generalized_distorted_greedy}.
\end{corollary}
Note that the lower bound of Corollary \ref{cor:gen_dist2fact} depends on $\beta$ through $g$ and $c$. If $\beta$ is chosen to be too small, then the lower bound might be too loose as the right hand side might be negative, \textcolor{black}{and hence the lower bound may be too loose to be informative.}

\section{Supermodular minimization of distance to independence $\mathbb{I} (P^{(S)})$} \label{sec:dist2indp}

Given $P\in \mathcal{L}(\mathcal{X})$ and $d, m \geq 2$, we aim to \textcolor{black}{analyze} the following supermodular (recall Theorem \ref{thm:submod_mc}) minimization problem 
\begin{align}\label{eq:dist2indp}
    \min_{S\subseteq \llbracket d \rrbracket;~ |S| = m} \mathbb{I}(P^{(S)}).
\end{align}
\textcolor{black}{Note that} we shall \textcolor{black}{consider} the constraint $|S|=m$ rather than $|S| \leq m$ as in Section~\ref{sec:entropy_rate} and Section~\ref{sec:dist2fact} because $S\mapsto \mathbb I (P^{(S)})$ is monotonically non-decreasing \textcolor{black}{(see Theorem~\ref{thm:submod_mc})}.

The supermodular minimization problem~\eqref{eq:dist2indp} is equivalent to the following submodular maximization problem
\begin{align}\label{eq:dist2indp_submod}
    \max_{S\subseteq \llbracket d \rrbracket;~ |S|=m} f(S) = -\mathbb{I}(P^{(S)}) = H(P^{(S)}) - \sum_{e \in S} H(P^{(e)}).
\end{align}
Note that we restrict $m$ to be at least $2$, since we have the trivial result that $\mathbb{I}(P^{(e)}) = \mathbb{I}(P^{(\emptyset)}) = 0$ if the constraint is $m = 0$ or $m = 1$. From Theorem~\ref{thm:submod_mc}, $f(S)$ is monotonically non-increasing and submodular. Therefore, the heuristic greedy algorithm (see Section 4 of \cite{MR503866}) cannot provide a theoretical guarantee.

\textcolor{black}{Based on} Theorem~\ref{thm:monotonize}, for any $\beta \in \mathbb R$, we \textcolor{black}{consider} a monotonically non-decreasing submodular function $g$ given by
\begin{align}\label{eq:g_dist2indp}
    g(S) &= f(S) - \beta + \sum_{e\in S} (H(P^{(-e)}) + H(P^{(e)}) - H(P))\\
    &= f(S) - \beta + \sum_{e\in S} D(P \| P^{(e)} \otimes P^{(-e)}). \nonumber
\end{align}
We choose $\beta = 0$ and \textcolor{black}{take} the following non-negative, modular function \textcolor{black}{to} be
\begin{align}\label{eq:c_dist2indp}
    c(S) = \sum_{e\in S} D(P \| P^{(e)} \otimes P^{(-e)})
\end{align}
so that $f(S) = g(S) - c(S)$. By Theorem~\ref{thm:lb_distgrdy}, we can apply Algorithm~\ref{alg:distorted_greedy} to obtain a lower bound.
\begin{corollary}\label{cor:dist2indp}
    Let $P\in \mathcal{L}(\mathcal{X})$ be $\pi$-stationary along with $d,m \geq 2$. In Algorithm~\ref{alg:distorted_greedy}, we take $g$ as in~\eqref{eq:g_dist2indp}, $c$ as in~\eqref{eq:c_dist2indp}, and $\mathrm{OPT} = \max_{S\subseteq \llbracket d \rrbracket;~ |S|=m} f(S)$. By Theorem~\ref{thm:lb_distgrdy}, we have the following lower bound \begin{align*}
        f(S_m) = -\mathbb{I}(P^{(S_m)}) \geq (1 - e^{-1}) g(\mathrm{OPT}) - c(\mathrm{OPT}),
    \end{align*}
    where $S_m$ is the output of Algorithm~\ref{alg:distorted_greedy}.
\end{corollary}

\subsection{Supermodular minimization of distance to independence of the complement set $\mathbb{I} (P^{(-S)})$} \label{sec:dist2indp_c}

From Theorem~\ref{thm:dist2indp_complement}, $\mathbb{I} (P^{(-S)})$ is monotonically non-increasing and supermodular. Given $P \in \mathcal{L}(\mathcal{X})$, $d \geq 2$, and $m \leq d - 2$, we shall investigate the following optimization problem 
\begin{align*}
    \max_{S \subseteq \llbracket d \rrbracket;~ |S| \leq m} f(S) = -\mathbb{I} (P^{(-S)}).
\end{align*}
Note that we restrict $m$ to be at most $d-2$, since we have the trivial result that $\mathbb{I}(P^{(e)}) = \mathbb{I}(P^{(\emptyset)}) = 0$ if the constraint is $m = d$ or $m = d-1$.

Since $f(S) = -\mathbb{I} (P^{(-S)})$ is monotonically non-decreasing and submodular, we can \textcolor{black}{then} apply the heuristic greedy algorithm (see Section 4 of \cite{MR503866}) that comes along with a $(1 - e^{-1})$-approximation guarantee.

\subsection{$k$-supermodular minimization of distance to independence of the tensorized keep-$S_i$-in matrices $\mathbb{I}(\otimes_{i=1}^k P^{(S_i)})$}
In this section, we \textcolor{black}{look into} the following map
\begin{align}\label{eq:generalized_dist2indep}
    (k+1)^{\llbracket d \rrbracket} \ni \mathbf{S} = (S_1, \ldots, S_k) \mapsto \mathbb{I}(\otimes_{i=1}^k P^{(S_i)}).
\end{align}

\textcolor{black}{We first state a result that decomposes the distance to independence of the tensorized matrices into a sum of distances to independence of individual matrix.}
\begin{lemma}\label{lem:ind_dist2indp}
    For $k \in \mathbb{N}$ and $\mathbf{S} \in (k+1)^{\llbracket d \rrbracket}$, we have
    $$\mathbb{I}(\otimes_{i=1}^k P^{(S_i)}) = \sum_{i=1}^k \mathbb{I}(P^{(S_i)}).$$
\end{lemma}
\begin{proof}
    We shall prove by induction on $k$. When $k=1$, the equality trivially holds. When $k=2$, according to the chain rule of KL divergence (see Theorem 2.15 of \cite{Polyanskiy_Wu_2025}), 
    \begin{align*}
        \mathbb{I}(P^{(S_1)} \otimes P^{(S_2)}) &= D(P^{(S_1)} \otimes P^{(S_2)} \| \otimes_{i\in S_1 \cup S_2} P^{(i)}) \\
        &= D(P^{(S_1)} \| \otimes_{i\in S_1} P^{(i)}) + D(P^{(S_2)} \| \otimes_{i\in S_2} P^{(i)})\\
        &= \mathbb{I}(P^{(S_1)}) + \mathbb{I}(P^{(S_2)}).
    \end{align*}
    Suppose $\mathbb{I}(\otimes_{i=1}^m P^{(S_i)}) = \sum_{i=1}^m \mathbb{I}(P^{(S_i)})$ holds ($k=m$), then using the chain rule of KL divergence again (Theorem 2.15 of \cite{Polyanskiy_Wu_2025}), we have
    \begin{align*}
        \mathbb{I}(\otimes_{i=1}^{m+1} P^{(S_i)}) &= D(\otimes_{i=1}^m P^{(S_i)} \otimes P^{(S_{m+1})} \| \otimes_{i\in (\cup_{i=1}^m S_i) \cup S_{m+1}} P^{(i)})\\
        &= D(\otimes_{i=1}^m P^{(S_i)} \| \otimes_{i\in \cup_{i=1}^m S_i} P^{(i)}) + D(P^{(S_{m+1})} \| \otimes_{i\in S_{m+1}} P^{(i)})\\
        &= \sum_{i=1}^{m+1} \mathbb{I}(P^{(S_i)}).
    \end{align*}
\end{proof}

\textcolor{black}{In the second auxiliary result of this subsection, we analyze the monotonicity of \eqref{eq:generalized_dist2indep}.}
\begin{lemma}
    The map \eqref{eq:generalized_dist2indep} is pairwise monotonically non-decreasing. In particular, when $P$ is non-factorizable and $\pi$-stationary, the map \eqref{eq:generalized_dist2indep} is pairwise monotonically strictly increasing for all pairs.
\end{lemma}
\begin{proof}
    Let $f(\mathbf{S}) = \mathbb{I}(\otimes_{i=1}^k P^{(S_i)})$. We shall prove that $\Delta_{e, i} f(\mathbf S) + \Delta_{e, j} f(\mathbf S) \geq 0$, where $i \neq j \in \llbracket d \rrbracket$ and $e \notin \mathrm{supp}(\mathbf{T})$. Since $\mathbb{I}(P^{(S)}) = \sum_{i\in S} H(P^{(i)}) - H(P^{(S)})$, we note that
    \begin{align*}
        \Delta_{e, i} f(\mathbf S) + \Delta_{e, j} f(\mathbf S) &= \mathbb{I}(P^{(S_i \cup \{e\})}) - \mathbb{I}(P^{(S_i)}) + \mathbb{I}(P^{(S_j \cup \{e\})}) - \mathbb{I}(P^{(S_i)}) \\
        &= \big[H(P^{(e)}) + H(P^{(S_i)}) - H(P^{(S_i \cup \{e\})})\big]\\
        &\quad + \big[H(P^{(e)}) + H(P^{(S_j)}) - H(P^{(S_j \cup \{e\})})\big] \\
        &= D(P^{(S_i \cup \{e\})} \| P^{(S_i)} \otimes P^{(e)}) + D(P^{(S_j \cup \{e\})} \| P^{(S_j)} \otimes P^{(e)}),
    \end{align*}
    which is non-negative. In particular, when $P$ is non-factorizable, it is strictly positive.
\end{proof}
\textcolor{black}{The third result of this subsection examines the orthant supermodularity of \eqref{eq:generalized_dist2indep}.}
\begin{lemma}\label{lem:dist2indp_orthant_supermodular}
    The map \eqref{eq:generalized_dist2indep} is orthant supermodular.
\end{lemma}
\begin{proof}
    Let $f(\mathbf{S}) = \mathbb{I}(\otimes_{i=1}^k P^{(S_i)})$. For any $\mathbf S \preceq \mathbf T$, we shall prove that $\Delta_{e, i} f(\mathbf S) \leq \Delta_{e, i} f(\mathbf T)$, where $i \in \llbracket d \rrbracket$ and $e \in \llbracket d\rrbracket \backslash \mathrm{supp}(\mathbf{T})$.
    \begin{align*}
         \Delta_{e, i} f(\mathbf S) - \Delta_{e, i} f(\mathbf T) &= \big[H(P^{(e)}) + H(P^{(S_i)}) - H(P^{(S_i \cup \{e\})})\big] \\
         &\quad - \big[H(P^{(e)})+ H(P^{(T_i)}) - H(P^{(T_i \cup \{e\})})\big]\\
        &= \big[ H(P^{(T_i \cup \{e\})}) - H(P^{(T_i)})\big] - \big[ H(P^{(S_i \cup \{e\})}) - H(P^{(S_i)})\big] \leq 0,
    \end{align*}
    where the inequality holds owing to the submodularity of $S \mapsto H(P^{(S)})$ in view of Theorem \ref{thm:submod_mc}.
\end{proof}

Collecting the previous two results, we see that, for non-factorizable $P$, the map \eqref{eq:generalized_dist2indep} is not $k$-supermodular as $k$-supermodularity requires both the pairwise monotonically non-increasing property and orthant supermodularity (see Theorem \ref{thm:k-submod}). 

Given $P \in \mathcal{L}(\mathcal{X})$, $d, m \geq k+1$ and $\mathbf{V} \in (k+1)^{\llbracket d \rrbracket}$, since the map~\eqref{eq:generalized_dist2indep} is orthant supermodular, we are interested in the following orthant submodular maximization problem
\begin{align*}
    \max_{\mathbf{S} \preceq \mathbf{V}; \, |\mathrm{supp}(\mathbf{S})| = m} f(\mathbf{S}) = -\mathbb{I}(\otimes_{i=1}^k P^{(S_i)}) = -\sum_{i=1}^k \mathbb{I}(P^{(S_i)}).
\end{align*}
\textcolor{black}{Intuitively, we are seeking a coordinate partition $\mathbf{S}$ subject to cardinality constraint $m$ in order to \emph{minimize} the distance to independence of the product chain $\otimes_{i=1}^k P^{(S_i)}$.}
We are restricting $m$ to be at least $k+1$ following the pigeonhole principle, as we need at least one $S_i$ with $|S_i| > 1$. If $m \leq k$, we can take either $S_i = \{e\}$ or $S_i = \emptyset$ for all $i\in \llbracket k \rrbracket$ so that the optimization problem becomes trivial.

{\color{black}Making use of Theorem~\ref{thm:monotonize-k}, we see that the function $g$ given by 
\begin{align}\label{eq:g_gen_dist2indp}
    g(\mathbf{S}) &= f(\mathbf{S}) - \beta + \sum_{i=1}^k \sum_{e\in S_i} [H(P^{(V_i \backslash \{e\})}) + H(P^{(e)}) - H(P^{(V_i)})]\\
    &= f(\mathbf{S}) - \beta + \sum_{i=1}^k \sum_{e\in S_i} D(P^{(V_i)} \| P^{(V_i \backslash \{e\})} \otimes P^{(e)})\nonumber
\end{align}
is monotonically non-decreasing and $k$-submodular. We now take $\beta = 0$ and the function $c$ to be  
\begin{align}\label{eq:c_gen_dist2indp}
    c(\mathbf{S}) = \sum_{i=1}^k \sum_{e\in S_i} D(P^{(V_i)} \| P^{(V_i \backslash \{e\})} \otimes P^{(e)}),
\end{align}
which is non-negative and modular. This allows us to express $f(\mathbf{S}) = g(\mathbf{S}) - c(\mathbf{S})$.} By applying Algorithm~\ref{alg:generalized_distorted_greedy}, we obtain a result with the following lower bound by Theorem~\ref{thm:lb_gen_distgrdy}. 
\begin{corollary}\label{cor:gen_dist2indp}
    Let $P\in \mathcal{L}(\mathcal{X})$ be $\pi$-stationary along with $d, m \geq k+1$ and $\mathbf{V} \in (k+1)^{\llbracket d \rrbracket}$. In Algorithm~\ref{alg:generalized_distorted_greedy}, we take $g$ as in~\eqref{eq:g_gen_dist2indp}, $c$ as in~\eqref{eq:c_gen_dist2indp}, and $\mathbf{OPT} = \argmax_{\mathbf{S} \preceq \mathbf{V};~ |\mathrm{supp}(\mathbf{S})| = m} f(\mathbf{S})$, then by Theorem~\ref{thm:lb_gen_distgrdy}, we have the following lower bound
    \begin{align*}
        f(\mathbf{S}_m) = -\mathbb{I}(\otimes_{i=1}^k P^{(S_{m,i})}) \geq (1 - e^{-1}) g(\mathbf{OPT}) - c(\mathbf{OPT}),
    \end{align*}
    where $\mathbf{S}_m = (S_{m, 1}, \ldots, S_{m, k})$ is the output of Algorithm~\ref{alg:generalized_distorted_greedy}.
\end{corollary}
In the special case where $k=1$ and $\mathbf{V} = \llbracket d \rrbracket$, we recover Corollary~\ref{cor:dist2indp}.

\subsection{$k$-supermodular minimization of distance to independence of the tensorized keep-$V_i\backslash S_i$-in matrices $\mathbb{I}(\otimes_{i=1}^k P^{(V_i \backslash S_i)})$} \label{sec:gen_dist2indp_c}

For given $\mathbf{V} \in (k+1)^{\llbracket d \rrbracket}$, we consider the following map in view of Lemma~\ref{lem:ind_dist2indp},
\begin{align}\label{eq:map_gen_dist2indp_c}
    \{\mathbf{S} \in (k+1)^{\llbracket d \rrbracket};~\mathbf{S} \preceq \mathbf{V} \} \ni \mathbf{S} =(S_1, \ldots, S_k) \mapsto \mathbb{I}(\otimes_{i=1}^k P^{(V_i \backslash S_i)}) = \sum_{i=1}^k \mathbb I (P^{(V_i \backslash S_i)}).
\end{align}

We first prove a result concerning monotonicity and $k$-supermodularity of the map above.
\begin{lemma}
The map~\eqref{eq:map_gen_dist2indp_c} is monotonically non-increasing and $k$-supermodular.
\end{lemma}

\begin{proof}
    In view of Theorem~\ref{thm:dist2indp_complement}, for each component $S_i$, we take $V_i$ as the ground set, hence $\mathbb I (P^{(V_i \backslash S_i)})$ is monotonically non-increasing and supermodular. From Lemma \ref{lem:ind_dist2indp}, this function is the sum of $k$ monotonically non-increasing and supermodular functions. From Lemma \ref{lem:sum_nonincr_supmod}, we conclude that this map is $k$-supermodular and monotonically non-increasing.
\end{proof}
Therefore, we denote the following monotonically non-decreasing, $k$-submodular function $g$ as 
\begin{align}\label{eq:g_gen_dist2indp_c}
    g(\mathbf{S}) = -\mathbb{I}(\otimes_{i=1}^k P^{(V_i \backslash S_i)}) = -\sum_{i=1}^k \mathbb I(P^{(V_i \backslash S_i)}).
\end{align}
Given $d \geq k+1$, $m \leq d - k - 1$, we are interested in the following maximization problem given by 
\begin{align*}
    \max_{\mathbf{S} \preceq \mathbf{V};~ |\mathrm{supp}(\mathbf{S})| \leq m} g(\mathbf{S}).
\end{align*}

We are restricting $m$ by $m \leq d - k - 1$ following the pigeonhole principle, as we want $|V_i \backslash S_i| \geq 2$ for at least one $i$. If $m \geq d -k$, we can choose either $V_i \backslash S_i = \{e\}$ or $V_i \backslash S_i = \emptyset$ so that the optimization problem is trivial.

By \textcolor{black}{specializing into} $c=0$ as a non-negative modular function, we apply Algorithm~\ref{alg:generalized_distorted_greedy} to obtain an optimization result with $(1 - e^{-1})$-approximation guarantee.
\begin{corollary}
    Let $P \in \mathcal{L}(\mathcal{X})$ be $\pi$-stationary along with $d \geq k + 1$, $m \leq d - k - 1$ and $\mathbf{V} \in (k+1)^{\llbracket d \rrbracket}$. In Algorithm~\ref{alg:generalized_distorted_greedy}, we take $g$ as in \eqref{eq:g_gen_dist2indp_c}, $c=0$ and denote \begin{align*}
        \mathbf{OPT} = \argmax_{\mathbf{S} \preceq \mathbf{V};~ |\mathrm{supp}(\mathbf{S})| \leq m} g(\mathbf{S}).
    \end{align*}
    From Theorem~\ref{thm:lb_gen_distgrdy}, we obtain the following lower bound
    \begin{align*}
        g(\mathbf{S}_m) \geq (1 - e^{-1}) g(\mathbf{OPT}),
    \end{align*}
    where $\mathbf{S}_m = (S_{m, 1}, \ldots, S_{m, k})$ is the output of Algorithm~\ref{alg:generalized_distorted_greedy}.
\end{corollary}

\section{Supermodular minimization of distance to stationarity $D(P^{(S)} \| \Pi ^{(S)})$}\label{sec:dist2stat}

In this section, we \textcolor{black}{examine} the following map:
\begin{align}\label{eq:map_dist2stationarity}
    2^{\llbracket d \rrbracket} \ni S \mapsto D(P^{(S)} \| \Pi ^{(S)}),
\end{align}
where $\Pi$ is the matrix of stationary distribution with each row of $\Pi$ being $\pi$. We first show that this map is monotonically non-decreasing.

\begin{lemma}\label{lem:dist2statnondecrease}
    The map \eqref{eq:map_dist2stationarity} is monotonically non-decreasing.
\end{lemma}
\begin{proof}
    We choose $S \subseteq T \subseteq \llbracket d \rrbracket$. By the partition lemma (Theorem \ref{thm:part_lem}), we have 
    $$ D(P^{(S)} \| \Pi ^{(S)}) \leq  D(P^{(T)} \| \Pi ^{(T)}),$$ 
    and hence this map is monotonically non-decreasing.
\end{proof}

{\color{black}Intuitively, the monotonicity of the map \eqref{eq:map_dist2stationarity} means that the larger the set, the projected multivariate Markov chain is further away to stationarity in one-step.}

We are interested in the following optimization problem
\begin{align*}
    \max_{S \subseteq \llbracket d \rrbracket;~ |S|=m} D(P^{(S)} \| \Pi^{(S)}),
\end{align*}
as solving the above can help to identify coordinates which are furthest away from the equilibrium in one step. \textcolor{black}{This is for instance useful in improving sampling via MCMC using a heuristic that we propose in Section \ref{subsec:MCMC} below.}

To solve this optimization problem with a theoretical guarantee, we recall the batch greedy algorithm (Algorithm~\ref{alg:batch_greedy}, see Theorem 7 of~\cite{jagalur2021batch}).

\begin{algorithm}
\caption{\textbf{Batch greedy algorithm}}\label{alg:batch_greedy}
\begin{algorithmic}[1]
\Require monotonically non-decreasing set function $f$; ground set $U$; total cardinality constraint $m$; number of steps $l$ and cardinality constraints $q_i$ such that $\sum_{i=1}^l q_i = m$
\State Initialize $S_0 = \emptyset$
\For{$i = 1$ to $l$}
    \State Determine incremental gains $f(S_{i-1} \cup \{e\}) - f(S_{i-1})$, $\forall e\in U \backslash S_{i-1}$
    \State Find $Q$, comprising the elements with top-$q_i$ incremental gains
    \State $S_{i} \gets S_{i-1} \cup Q$
\EndFor
\State \textbf{Output}: $S_l$
\end{algorithmic}
\end{algorithm}

It turns out that the theoretical guarantee depends on the supermodularity ratio and submodularity ratio of a set function $f$, that we shall now briefly recall. The \textbf{supermodularity ratio} of a non-negative set function $f$ (Definition 6 of \cite{jagalur2021batch}) with respect to the set $U$ and a cardinality constraint $m \geq 1$ is 
\begin{align*}
    \eta_{U, m} := \min_{S \subseteq U;~ T: |T| \leq m,~ S\cap T = \emptyset} \frac{f(S \cup T) - f(S)}{\sum_{e \in T} [f(S\cup \{e\}) - f(S)]},
\end{align*}
while the \textbf{submodularity ratio} of $f$ (Definition 32 of \cite{jagalur2021batch}) with respect to the set $U$ and a cardinality constraint $k \geq 1$ is 
\begin{align*}
    \gamma_{U, m} := \min_{S \subseteq U;~ T: |T| \leq m,~ S\cap T = \emptyset} \frac{\sum_{e \in T} [f(S\cup \{e\}) - f(S)]}{f(S \cup T) - f(S)}.
\end{align*}
We then state the lower bound pertaining to Algorithm~\ref{alg:batch_greedy} (see Theorem 7 of~\cite{jagalur2021batch}).

\begin{theorem}[Lower bound for batch greedy algorithm] \label{thm:lb_batch_greedy}
    Let $P \in \mathcal{L}(\mathcal{X})$ be $\pi$-stationary and $U$ be the ground set. Let $f$ be a monotonically non-decreasing set function with $f(\emptyset) = 0$. Algorithm~\ref{alg:batch_greedy} yields the following lower bound
    \begin{align*}
        f(S_l) \geq \left(1 - \prod_{i=1}^l \left(1 - \frac{q_i \cdot \eta_{U, q_i} \cdot \gamma_{U, m}}{m}\right)\right) \max_{S \subseteq U;~ |S| = m} f(S),
    \end{align*}
    where $S_l$ is the output set of Algorithm~\ref{alg:batch_greedy}.
\end{theorem}

Since we have a monotonically mon-decreasing map~\eqref{eq:map_dist2stationarity} with $D(P^{(\emptyset)} \| \Pi^{(\emptyset)}) = 0$, we can apply the Algorithm~\ref{alg:batch_greedy} (see Theorem 7 of~\cite{jagalur2021batch}) with the following lower bound.

\begin{corollary}
    Let $P \in \mathcal{L}(\mathcal{X})$ be $\pi$-stationary and $U = \llbracket d \rrbracket$ be the ground set. Let $f$ be \eqref{eq:map_dist2stationarity} which is a monotonically non-decreasing set function with $f(\emptyset) = 0$. \textcolor{black}{In view of Theorem~\ref{thm:lb_batch_greedy},} Algorithm~\ref{alg:batch_greedy} yields the following lower bound
    \begin{align*}
        f(S_l) \geq \left(1 - \prod_{i=1}^l \left(1 - \frac{q_i \cdot \eta_{U, q_i} \cdot \gamma_{U, m}}{m}\right)\right) \max_{S \subseteq \llbracket d \rrbracket;~ |S| = m} f(S),
    \end{align*}
    where $S_l$ is the output set of Algorithm~\ref{alg:batch_greedy}.
\end{corollary}

We now consider the special case where the stationary distribution $\pi$ is of product form. In this case, we can show the supermodularity of the map~\eqref{eq:map_dist2stationarity}.

\begin{lemma}\label{lem:dist2stationarity}
    The map \eqref{eq:map_dist2stationarity} is supermodular if $P$ is $\pi$-stationary where $\pi$ is of product form.
\end{lemma}
\begin{proof}
    \begin{align*}
        D(P^{(S)} \| \Pi ^{(S)}) &= \sum_{x^{(S)}} \sum_{y^{(S)}} \pi^{(S)}(x^{(S)}) P^{(S)}(x^{(S)}, y^{(S)}) \ln{\frac{P^{(S)} (x^{(S)}, y^{(S)})}{\pi^{(S)} (y^{(S)})}} \\
        &= -H(P^{(S)}) - \sum_{x^{(S)}} \sum_{y^{(S)}} \pi^{(S)}(x^{(S)}) P^{(S)}(x^{(S)}, y^{(S)}) \ln{\pi^{(S)} (y^{(S)})} \\
        &= -H(P^{(S)}) - \sum_{y^{(S)}} \ln{\pi^{(S)} (y^{(S)})} \sum_{x^{(S)}} \pi^{(S)}(x^{(S)}) P^{(S)}(x^{(S)}, y^{(S)})\\
        &= -H(P^{(S)}) + H(\pi^{(S)}).
    \end{align*}
    The last equation holds since $P$ is $\pi$-stationary and hence $$\pi^{(S)}(y^{(S)}) = \sum_{x^{(S)}} \pi^{(S)}(x^{(S)}) P^{(S)}(x^{(S)}, y^{(S)}).$$
    Since the stationary distribution $\pi$ is of product form, then 
    $\pi = \otimes_{i=1}^d \pi^{(i)}$, hence $H(\pi^{(S)}) = \sum_{i \in S} H(\pi^{(i)})$, which is a modular function.
    Also, since $H(P^{(S)})$ is submodular, then $-H(P^{(S)})$ is supermodular. Therefore, $D(P^{(S)} \| \Pi ^{(S)})$ is supermodular because it is a sum of a supermodular function and a modular function. 
\end{proof}

{\color{black} The supermodularity of map \eqref{eq:map_dist2stationarity} means that, when the stationary distribution $\pi$ of $P$ is of product form, the marginal gain of distance to stationarity by adding an element to a subset is at most equal to the marginal gain by adding the same element to a larger set. That is, when $S \subseteq T \subseteq \llbracket d \rrbracket$ and $e \notin T$, we have 
\[D(P^{(S \cup \{e\})} \| \Pi^{(S \cup \{e\})}) - D(P^{(S)} \| \Pi^{(S)})\leq D(P^{(T \cup \{e\})} \| \Pi^{(T \cup \{e\})}) - D(P^{(T)} \| \Pi^{(T)}).\]
}

We proceed to investigate the following \textcolor{black}{maximization task} when $P$ is $\pi$-stationary with product form $\pi$\textcolor{black}{:}
\begin{align*}
    \max_{S \subseteq \llbracket d \rrbracket;~|S| \leq m} f(S) = -D(P^{(S)} \| \Pi ^{(S)}).
\end{align*}

\textcolor{black}{In light of Theorem~\ref{thm:monotonize}, the following function $g$
\begin{align}\label{eq:g_dist2sta}
    g(S) &= f(S) - \beta + \sum_{e\in S} (H(P^{(-e)}) - H(\pi^{(-e)}) - H(P) + H(\pi)) \nonumber \\
    &= f(S) - \beta + \sum_{e\in S}(D(P\| P^{(e)} \otimes P^{(-e)}) + D(P^{(e)} \| \Pi^{(e)})).
\end{align}
is monotonically non-decreasing and submodular as $f$ is submodular. We again choose $\beta = 0$ and denote the following non-negative modular function as
\begin{align}\label{eq:c_dist2sta}
    c(S) = \sum_{e\in S}(D(P\| P^{(e)} \otimes P^{(-e)}) + D(P^{(e)} \| \Pi^{(e)})).
\end{align}
so that $f(S) = g(S) - c(S)$. It allows us to apply} Algorithm~\ref{alg:distorted_greedy} to obtain a result with the following lower bound:
\begin{corollary}\label{cor:dist2stat}
    Let $P\in \mathcal{L}(\mathcal{X})$ be $\pi$-stationary with $\pi$ to be product form. In Algorithm~\ref{alg:distorted_greedy}, we take $g$ as in~\eqref{eq:g_dist2sta}, $c$ as in~\eqref{eq:c_dist2sta}, and $\mathrm{OPT} = \argmax_{S \subseteq \llbracket d \rrbracket;~|S| \leq m} f(S)$. By Theorem~\ref{thm:lb_distgrdy}, we have the following lower bound
    \begin{align*}
        f(S_m) = -D(P^{(S_m)} \| \Pi ^{(S_m)}) \geq (1-e^{-1}) g(\mathrm{OPT}) - c(\mathrm{OPT}),
    \end{align*}
    where $S_m$ is the output set of Algorithm~\ref{alg:distorted_greedy}.
\end{corollary}

\subsection{Supermodular minimization of distance to stationarity of the complement set $D(P^{(-S)} \| \Pi ^{(-S)})$} \label{sec:dist2stat_c}

In this section, we investigate the following map:
\begin{align}\label{eq:map_dist2sta_c}
    2^{\llbracket d\rrbracket} \ni S \mapsto D(P^{(-S)} \| \Pi ^{(-S)}).
\end{align}

Owing to Lemma \ref{lem:dist2statnondecrease}, we first see that the map~\eqref{eq:map_dist2sta_c} is monotonically non-increasing. In addition, the map~\eqref{eq:map_dist2sta_c} is supermodular if $P$ is $\pi$-stationary with product form $\pi$ in view of Lemma~\ref{lem:complement_submodularity} and Lemma~\ref{lem:dist2stationarity}.

We are interested in the following optimization problem \begin{align*}
    \max_{S \subseteq \llbracket d \rrbracket;~|S|\leq m} f(S) = -D(P^{(-S)} \| \Pi ^{(-S)}),
\end{align*}
as solving the above allows us to identify coordinates whose complement set is the closest to equilibrium in one step.

Under the assumption of product form $\pi$, as the map~\eqref{eq:map_dist2sta_c} is monotonically non-increasing and supermodular, $f$ is monotonically non-decreasing and submodular. We apply the heuristic greedy algorithm (Section 4 of \cite{MR503866}) to obtain an approximate maximizer along with a $(1 - e^{-1})$-approximation guarantee.

\subsection{$k$-supermodular minimization of distance to stationarity of tensorized keep-$S_i$-in matrices $D(\otimes_{i=1}^k P^{(S_i)} \| \otimes_{i=1}^k \Pi^{(S_i)})$} \label{sec:gen_dist2stat}

In this section, for given $\mathbf{V} \in (k+1)^{\llbracket d \rrbracket}$, we \textcolor{black}{study} the following map:
\begin{align}\label{eq:map_gen_dist2sta}
    (k+1)^{\llbracket d \rrbracket} \ni \mathbf{S} &= (S_1,\ldots,S_k) \mapsto f(\mathbf{S}) = D(\otimes_{i=1}^k P^{(S_i)} \| \otimes_{i=1}^k \Pi^{(S_i)}).
\end{align}
We first give an orthant supermodularity result.
\begin{lemma}\label{lem:orthantsuperproduct}
    The map~\eqref{eq:map_gen_dist2sta} is orthant supermodular if $P$ is $\pi$-stationary where $\pi$ is of product form.
\end{lemma}
\begin{proof}
    By the chain rule or tensorization property of KL divergence (see Theorem 2.15 and 2.16 of \cite{Polyanskiy_Wu_2025}), we see that
    \begin{align*}
        D(\otimes_{i=1}^k P^{(S_i)} \| \otimes_{i=1}^k \Pi^{(S_i)}) = \sum_{i=1}^k D(P^{(S_i)} \| \Pi^{(S_i)}).
    \end{align*}
    We now take $\mathbf{S} \preceq \mathbf{T}$ and $e\in \llbracket d \rrbracket \backslash T_i$. By \eqref{eq:ort_submod}, we aim to show that $\Delta_{e, i} f(\mathbf{S}) \leq \Delta_{e, i} f(\mathbf{T})$, which indeed holds since
    \begin{align*}
        \Delta_{e, i} f(\mathbf{S}) &= D(P^{(S_i \cup \{e\})} \| \Pi^{(S_i \cup \{e\})}) - D(P^{(S_i)} \| \Pi^{(S_i)})\\
        &\leq D(P^{(T_i \cup \{e\})} \| \Pi^{(T_i \cup \{e\})}) - D(P^{(T_i)} \| \Pi^{(T_i)}) = \Delta_{e, i} f(\mathbf{T}),
    \end{align*}
    because $S \mapsto D(P^{(S)} \| \Pi^{(S)})$ is supermodular (see Theorem~\ref{lem:dist2stationarity}). 
\end{proof}
\textcolor{black}{We now confine ourselves into} the following optimization problem \begin{align*}
    \max_{\mathbf{S} \preceq \mathbf{V};~|\mathrm{supp}(\mathbf{S}_m)| \leq m} -f(\mathbf{S}),
\end{align*}
where $f$ is orthant supermodular \textcolor{black}{when $\pi$ is of product form in view of Lemma \ref{lem:orthantsuperproduct}.} \textcolor{black}{In other words, we are looking for a coordinate partition $\mathbf{S}$ so that the factorized transition kernel $\otimes_{i=1}^k P^{(S_i)}$ is ``closest'' to the tensorized equilibrium of the factorization $\otimes_{i=1}^k \Pi^{(S_i)}$ given the total cardinality constraint $m$.}

{\color{black}As a consequence of Theorem~\ref{thm:monotonize-k}, the following function $g$ is
\begin{align}\label{eq:g_gen_dist2stat}
    g(\mathbf{S}) = -f(\mathbf{S}) - \beta + \sum_{i=1}^k \sum_{e\in S_i} (D(P^{(V_i)}\| P^{(e)} \otimes P^{(V_i \backslash e)}) + D(P^{(e)} \| \Pi^{(e)})).
\end{align}
monotonically non-decreasing and $k$-submodular. We further take $\beta = 0$, and choose $c$ to be
\begin{align}\label{eq:c_gen_dist2stat}
    c(\mathbf{S}) = \sum_{i=1}^k \sum_{e\in S_i} (D(P^{(V_i)}\| P^{(e)} \otimes P^{(V_i \backslash e)}) + D(P^{(e)} \| \Pi^{(e)})).
\end{align}
which is non-negative and modular.} Since $-f(\mathbf{S}) = g(\mathbf{S}) - c(\mathbf{S})$, we apply Algorithm~\ref{alg:generalized_distorted_greedy} to obtain an approximate maximizer along with a lower bound.

\begin{corollary}\label{cor:gen_distgrdy_dist2stat}
    Let $P\in \mathcal{L}(\mathcal{X})$ be $\pi$-stationary with $\pi$ be of product form and $\mathbf{V} \in (k+1)^{\llbracket d \rrbracket}$. In Algorithm~\ref{alg:generalized_distorted_greedy}, we take $g$ as in~\eqref{eq:g_gen_dist2stat}, $c$ as in~\eqref{eq:c_gen_dist2stat}, and $\mathbf{OPT} = \argmax_{\mathbf{S} \preceq \mathbf{V};~|\mathrm{supp}(\mathbf{S}_m)| \leq m} -f(\mathbf{S})$. Then Theorem~\ref{thm:lb_gen_distgrdy} gives the following lower bound
    \begin{align*}
        -f(\mathbf{S}_m) \geq (1 - e^{-1}) g(\mathbf{OPT}) - c(\mathbf{OPT}),
    \end{align*}
    where $\mathbf{S}_m = (S_{m, 1}, \ldots, S_{m, k})$ is the output of Algorithm~\ref{alg:generalized_distorted_greedy}.
\end{corollary}

\subsection{$k$-supermodular minimization of distance to stationarity of tensorized keep-$V_i\backslash S_i$-in matrices $D(\otimes_{i=1}^k P^{(V_i \backslash S_i)} \| \otimes_{i=1}^k \Pi^{(V_i \backslash S_i)})$} \label{sec:gen_dist2stat_c}

For given $\mathbf{V} \in (k+1)^{\llbracket d \rrbracket}$, we \textcolor{black}{concentrate on} the following map:
\begin{align}
    \{\mathbf{S} \in (k+1)^{\llbracket d \rrbracket};~\mathbf{S} \preceq \mathbf{V} \} \ni \mathbf{S} &=(S_1, \ldots, S_k) \mapsto D(\otimes_{i=1}^k P^{(V_i \backslash S_i)} \| \otimes_{i=1}^k \Pi^{(V_i \backslash S_i)}). \label{eq:map_gen_dist2sta_c}
\end{align}

\textcolor{black}{The auxiliary result of this subsection looks into $k$-supermodularity of the above map.}
\begin{lemma}\label{lem:ksupermodulardist2stat}
    The map~\eqref{eq:map_gen_dist2sta_c} is monotonically non-increasing and $k$-supermodular if $P$ is $\pi$-stationary where $\pi$ is of product form.
\end{lemma}
\begin{proof}
    By the chain rule or tensorization property of KL divergence (see Theorem 2.15 and 2.16 of \cite{Polyanskiy_Wu_2025}), we see that $$D(\otimes_{i=1}^k P^{(V_i \backslash S_i)} \| \otimes_{i=1}^k \Pi^{(V_i \backslash S_i)}) = \sum_{i=1}^k D(P^{(V_i \backslash S_i)} \| \Pi^{(V_i \backslash S_i)}),$$ \textcolor{black}{can be expressed as a sum of $k$ monotonically non-increasing (recall the partition lemma (Theorem \ref{thm:part_lem})) and supermodular functions, and hence $k$-supermodular making use of Lemma \ref{lem:sum_nonincr_supmod}.}
\end{proof}
\textcolor{black}{In this subsection, we focus on} the following optimization problem 
\begin{align}\label{eq:ksupermoddist2statg}
    \max_{\mathbf{S} \preceq \mathbf{V};~ |\mathrm{supp}(\mathbf{S})| \leq m} g(\mathbf{S}) = -D(\otimes_{i=1}^k P^{(V_i \backslash S_i)} \| \otimes_{i=1}^k \Pi^{(V_i \backslash S_i)}).
\end{align}
\textcolor{black}{With a product form $\pi$ and by Lemma \ref{lem:ksupermodulardist2stat}, the map~\eqref{eq:map_gen_dist2sta_c} is monotonically non-increasing and $k$-supermodular, and hence $g$ is monotonically non-decreasing and $k$-submodular.} We \textcolor{black}{then} apply Algorithm~\ref{alg:generalized_distorted_greedy} to obtain a $(1 - e^{-1})$-approximation guarantee.

\begin{corollary}
    Let $P \in \mathcal{L}(\mathcal{X})$ be $\pi$-stationary with product form $\pi$ and $\mathbf{V} \in (k+1)^{\llbracket d \rrbracket}$. We take $g$ as in \eqref{eq:ksupermoddist2statg}, $c=0$ and $\mathbf{OPT} = \argmax_{\mathbf{S} \preceq \mathbf{V};~ |\mathrm{supp}(\mathbf{S})| \leq m} g(\mathbf{S})$. According to Theorem~\ref{thm:lb_gen_distgrdy}, we have the following lower bound for Algorithm~\ref{alg:generalized_distorted_greedy}
    \begin{align*}
        g(\mathbf{S}_m) \geq (1 - e^{-1})g(\mathbf{OPT}),
    \end{align*}
    where $\mathbf{S}_m = (S_{m,1}, \ldots, S_{m, k})$ is the output of Algorithm~\ref{alg:generalized_distorted_greedy}.
\end{corollary}

In the special case where $k=1$ and $\mathbf{V} = \llbracket d \rrbracket$, the above Corollary reduces to the $(1-e^{-1})$-approximation guarantee as in Section~\ref{sec:dist2stat_c}.

\section{Distance to factorizability over a fixed set $D(P^{(W \cup S)} \| P^{(W)} \otimes P^{(S)})$} \label{sec:dist2fact_fixed}

We fix a set $W\subseteq \llbracket d \rrbracket$ and \textcolor{black}{focus on} the following function:
\begin{align}\label{eq:map_dist2fact_fixedset}
    \{S \subseteq \llbracket d \rrbracket;~ S \cap W = \emptyset \}\ni S \mapsto f(S) = D(P^{(W \cup S)} \| P^{(W)} \otimes P^{(S)}).
\end{align}
We \textcolor{black}{now turn to} the following optimization problem with cardinality constraint \textcolor{black}{given by}
\begin{align*}
    \max_{S \subseteq \llbracket d \rrbracket;~ S \cap W = \emptyset;~|S| = m} f(S).
\end{align*}

We pick $S, T \subseteq \{S \subseteq \llbracket d \rrbracket;~ S \cap W = \emptyset \}$ with $S\subseteq T$ and compute that
\begin{align*}
    f(S) - f(T) = [H(P^{(T \cup W)}) - H(P^{(T)})] - [H(P^{(S \cup W)}) - H(P^{(S)})] \leq 0,
\end{align*}
where the inequality follows from the property that $S \mapsto H(P^{(S)})$ is submodular (see Theorem~\ref{thm:submod_mc}). Therefore $f$ is monotonically non-decreasing. Also, $f(\emptyset) = D(P^{(W)} \| P^{(W)} \otimes P^{(\emptyset)}) = 0$. As such, we can apply Algorithm~\ref{alg:batch_greedy} (see Theorem~\ref{thm:lb_batch_greedy}) with a lower bound.

\begin{corollary}\label{cor:dist2fact_fixed}
    Let $P \in \mathcal{L}(\mathcal{X})$ be $\pi$-stationary, $W \subseteq \llbracket d \rrbracket$, and $U = \llbracket d \rrbracket \backslash W$ be the ground set. Let $f$ be \eqref{eq:map_dist2fact_fixedset} which is a monotonically non-decreasing set function with $f(\emptyset) = 0$. Algorithm~\ref{alg:batch_greedy} yields the following lower bound
    \begin{align*}
        f(S_l) \geq \left(1 - \prod_{i=1}^l \left(1 - \frac{q_i \cdot \eta_{U, q_i} \cdot \gamma_{U, m}}{m}\right)\right) \max_{S \subseteq \llbracket d \rrbracket;~ S \cap W = \emptyset;~|S| = m} f(S),
    \end{align*}
    where $S_l$ is the output set of Algorithm~\ref{alg:batch_greedy}, \textcolor{black}{and $\eta$ and $\gamma$ denote the supermodularity ratio and submodularity ratio respectively \cite{jagalur2021batch}.}
\end{corollary}

\section{Numerical experiments} \label{sec:numericalexp}

We conduct a case study to evaluate the numerical performance of the submodular optimization algorithms on the information-theoretic properties of multivariate Markov chains. \textcolor{black}{We evaluate multivariate Markov chains associated with the Curie--Weiss model and the Bernoulli--Laplace level model. Both models come from probability, statistical physics and Markov chain Monte Carlo (MCMC) literature. They serve as established toy models in the analysis of multivariate Markov chains, see e.g. \cite{faulkner2024sampling}.} The code used in our numerical experiments is available at: \href{https://github.com/zheyuanlai/SubmodOptMC/}{\texttt{https://github.com/zheyuanlai/SubmodOptMC}}.

\subsection{Experiment settings - Curie--Weiss model}\label{sec:curie_weiss}

We consider a discrete $d$-dimensional hypercube state space given by
\begin{align*}
    \mathcal{X} = \{-1,+1\}^d.
\end{align*}
Let the Hamiltonian function be that of the Curie--Weiss model (see Chapter 13 of~\cite{bovier2016metastability}) on $\mathcal{X}$ with interaction coefficients $\frac{1}{2^{|j-i|}}$ and external magnetic field $h \in \mathbb{R}$, that is, for $x = (x^1, \ldots, x^d) \in \mathcal{X}$,
\begin{align*}
    \mathcal{H}(x) = - \sum_{i=1}^d \sum_{j=1}^d \dfrac{1}{2^{|j-i|}} x^i x^j - h \sum_{i=1}^d x^i.
\end{align*}
We consider a Glauber dynamics with a simple random walk proposal targeting the Gibbs distribution at temperature $T \geq 0$. At each step we pick uniformly at random one of the $d$ coordinates and flip it to the opposite sign, along with an acceptance-rejection filter, that is,
\begin{align*}
    P(x, y) = \begin{cases}
\dfrac{1}{d} e^{-\frac{1}{T} (\mathcal{H}(y) - \mathcal{H}(x))_+}, & \text{if } y = (x^1,x^2,\ldots,-x^i,\ldots,x^d), i \in \llbracket d \rrbracket, \\
1 - \sum_{y;~y \neq x} P(x, y), & \text{if } x = y, \\
0, & \text{otherwise},
\end{cases}
\end{align*}
where for $m \in \mathbb{R}$ we denote $m_+ := \max\{m,0\}$ the non-negative part of $m$. The stationary distribution of $P$ is the Gibbs distribution at temperature $T$ given by
\begin{align*}
    \pi(x) = \dfrac{e^{-\frac{1}{T} \mathcal{H}(x)}}{\sum_{z \in \mathcal{X}} e^{-\frac{1}{T} \mathcal{H}(z)}}.
\end{align*}
\paragraph{Parameters.} We aim to generate a $10$-dimensional Markov chain from the Curie--Weiss model. We choose $d=10$, and hence the state space is of product form with $\mathcal{X} = \{-1, +1\}^{10}$. The choices of Hamiltonian function $\mathcal H (x)$, transition matrix $P$, and the stationary distribution $\pi (x)$ are detailed in Section~\ref{sec:curie_weiss}, and we choose $T = 10$ as the temperature, $h = 1$ as the external magnetic field. For the numerical experiments of the generalized distorted greedy algorithm, we choose $\mathbf{V} = (V_1, V_2, V_3)$ where $V_1 = \{1, 2, 3, 4\}$, $V_2 = \{5, 6, 7\}$, and $V_3 = \{8, 9, 10\}$. 

\subsection{Experiment settings - Bernoulli--Laplace level model}
\label{sec:BLLM}
We consider a $d$-dimensional Bernoulli–Laplace level model as described in Section 4.2 of~\cite{MR2521887}. Let \begin{align*}
    \mathcal{X} = \{x = (x^1, \ldots, x^{d}) \in \mathbb{N}_0^d;~ x^1 + \ldots + x^{d} = N\}
\end{align*}
be the state space, where $x^i$ can be interpreted as the number of ``particles'' of type $i$ out of the total number $N$. The stationary distribution of such Markov chain, $\pi$, is given by the multivariate hypergeometric distribution described in Lemma 4.18 of~\cite{MR2521887}. Concretely, we have
\begin{align}\label{eq:sta_dist}
    \pi(x) = \frac{\prod_{i=1}^{d} {l_i \choose x^i}}{{l_1 + \ldots + l_{d} \choose N}}, \quad x \in \mathcal{X}, 
\end{align}
for some fixed parameters $l_1, \ldots, l_{d} \in \mathbb N$ representing the total number of ``particles'' of type $i$.

Following the spectral decomposition for reversible Markov chains (see Section 2.1 of \cite{MR2521887} for background), the transition matrix $P$ is written as:
\begin{align}\label{eq:spec_decomp}
    P(x, y) = \sum_{n=0}^N \beta_n \phi_n(x) \phi_n (y) \pi(y),
\end{align}
where $\beta_n$ are the eigenvalues and $\phi_n(x)$ is the eigenfunction.

From Definition 4.15 of~\cite{MR2521887}, in the Bernoulli--Laplace level model, $s$ is the swap size parameter satisfying $$0 \leq s \leq \min \left\{N, \sum_{i=1}^d l_i - N\right\},$$ where we consider $\sum_{i=1}^d l_i > N$. From Theorem 4.19 of~\cite{MR2521887}, the eigenvalues for the Bernoulli--Laplace level model are given by
\begin{align*}
    \beta_n = \sum_{k=0}^n {n \choose k} \frac{(N - s)_{[n-k]} s_{[k]}}{N_{[n-k]} \left(\sum_{i=1}^{d} l_i - N\right)_{[k]}}, \quad 0\leq n \leq N,
\end{align*}
where $a_{[k]} = a (a-1) \cdots (a - k + 1)$, and we apply the convention that $a_{[0]} = 1$. 

In this case, we choose the eigenfunction as \begin{align*}
    \phi_n (x) = \left\{\mathbf{Q_n}\left(x; N, -\sum_{i=1}^{d} l_i\right)\right\}_{|\mathbf{n}| = n},
\end{align*}
where $\mathbf{Q_n}$ are the multivariate Hahn polynomials for the hypergeometric distribution as defined in Proposition 2.3 of~\cite{MR2521887}.

\paragraph{Parameters.} We aim to generate a $10$-dimensional Markov chain from the Bernoulli--Laplace level model. We consider the special case where $s=1$ and choose $d=11$, $l_1 = \ldots = l_{10} = 1$, $N = 10$, and $l_{11} = 10$. We let $x^{11} = N - \sum_{i=1}^{10} x^i$, and hence the state space is of product form with $\mathcal{X} = \{0, 1\}^{10}$.

The transition probabilities follow the dynamics as in~\eqref{eq:spec_decomp}, where particles hop between coordinates while respecting capacity constraints, and the stationary distribution $\pi$ is computed as in~\eqref{eq:sta_dist}. For the numerical experiments of the generalized distorted greedy algorithm, we choose $\mathbf{V} = (V_1, V_2, V_3)$ where $V_1 = \{1, 2, 3, 4\}$, $V_2 = \{5, 6, 7\}$, and $V_3 = \{8, 9, 10\}$. 

{\color{black}
\subsection{An experiment on improving sampling of MCMC}\label{subsec:MCMC}
We give a small-scale MCMC experiment illustrating how the subset selection algorithm is able to identify optimal subset of a multivariate Markov chain under a cardinality constraint that is closest to stationarity, and how this chosen subset can be used to improve sampling of $\pi$. In this section, we consider the Curie--Weiss model from Section~\ref{sec:curie_weiss} with a dimension $d=8$ so that $\mathcal{X} = \{-1,+1\}^8$. We choose the cardinality constraint as $m=7$. In other words, we want to choose a subset containing $7$ coordinates that is closest to equilibrium among all feasible subsets. This is equivalent to the following optimization problem:
\[
    i^\star = \argmin_{i \in \llbracket 8 \rrbracket} D(P^{(-i)} \| \Pi^{(-i)}),
\]
The algorithm from Section~\ref{sec:dist2stat_c} gives $i^\star = 4$, hence the resulting subset of coordinates $\{1,2,3,5,6,7,8\}$ is the closest to equilibrium.

We proceed to compare the mixing rate of different choices of subsets. For each $i \in \llbracket 8\rrbracket$, we calculate its worst-case total variation distance between $(P^{(-i)})^n$ and $\Pi^{(-i)}$:
\[
    d_{\mathrm{TV}}^{(i)}(n) := \max_{x^{(-i)} \in \mathcal{X}^{(-i)}} \frac{1}{2} \sum_{y^{(-i)} \in \mathcal{X}^{(-i)}} \left|(P^{(-i)})^n(x^{(-i)}, y^{(-i)}) - \Pi^{(-i)}(x^{(-i)}, y^{(-i)})\right|.
\]

We plot $1000\times d_{\mathrm{TV}}^{(i)}(n)$ as a function of $n$ for all $i\in \llbracket 8 \rrbracket$, with the vertical axis representing the scaled total variation distance and the horizontal axis representing $n$, see Figure~\ref{fig:exp_mcmc}. The curve for $i=4$ decays the fastest among the curves of all coordinates, indicating that $P^{(-4)}$ mixes fastest, which justifies the subset-selection criterion and algorithm introduced in Section \ref{sec:dist2stat_c}.

\begin{figure}[H]
    \centering
    \includegraphics[width=0.85\linewidth]{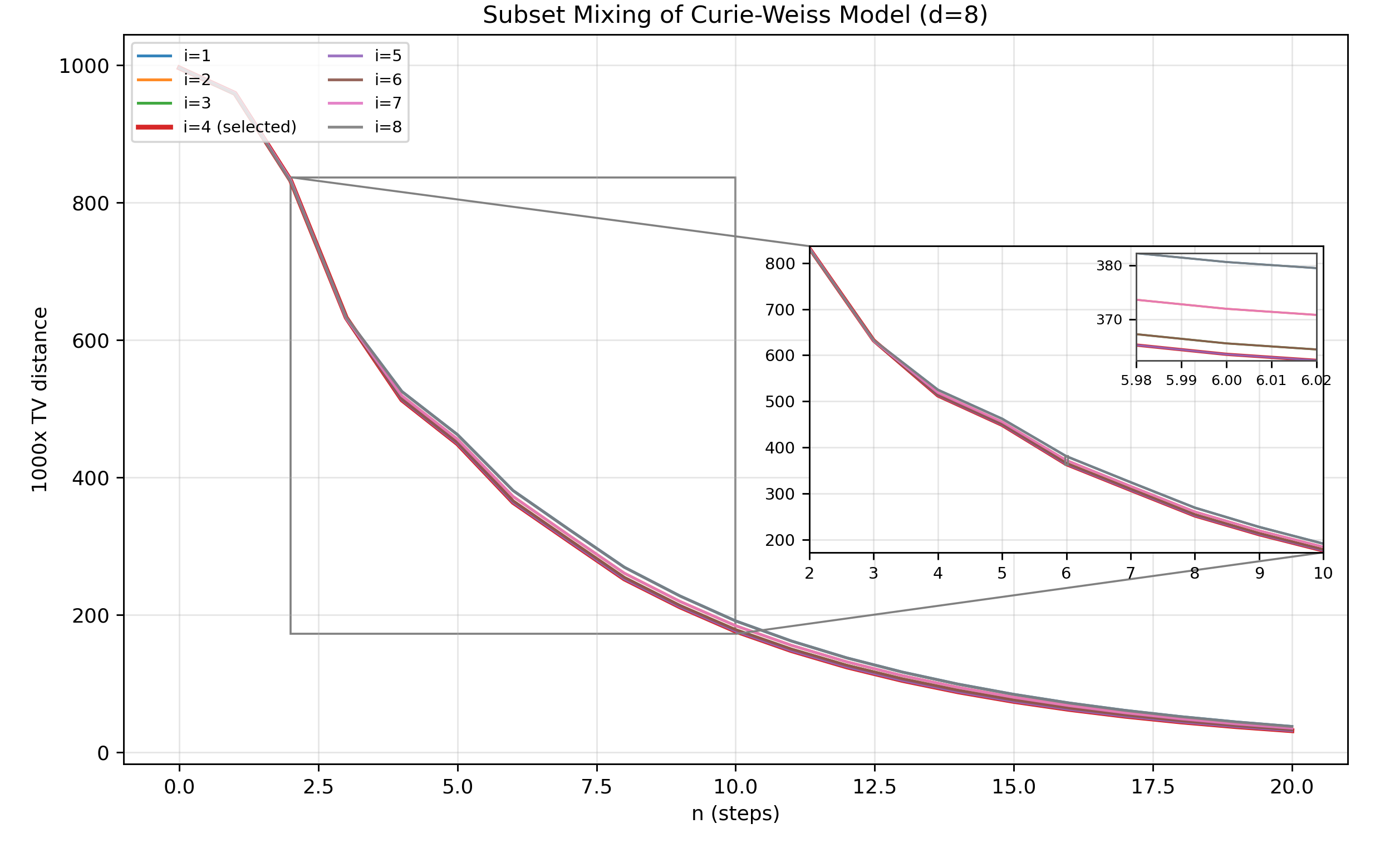}
    \caption{The leave-one-out mixing analysis of the Curie-Weiss model ($d=8$).}
    \label{fig:exp_mcmc}
\end{figure}

One heuristic argument from the above numerical results is that the coordinate 4 is ``far-off'' from equilibrium, so we separate it out and run an individual sampler $P^{(4)}$ to equilibrate it. Thus, instead of considering $P$, we run the tensorized transition matrix $P^{(-4)} \otimes P^{(4)}$. Note that this latter matrix targets $\pi^{(-4)} \otimes \pi^{(4)}$, which is different from $\pi$.

To validate the heuristics, we proceed to compare the convergence of the factorized design $P^{(-4)} \otimes P^{(4)}$ with the original chain $P$. From Figure~\ref{fig:exp_mcmc}, we see that the TV distance between $(P^{(-4)})^n$ and $\Pi^{(-4)}$ reaches below $0.2$ after 10 steps. We proceed to compute the worst-case total variation distance of both designs in comparison to stationarity, that is,
\[\max_{x \in \mathcal{X}}\|P^{10}(x,\cdot) - \pi\|_\mathrm{TV} = 0.22, \qquad \max_{x \in \mathcal{X}}\|(P^{(-4)})^{10} \otimes (P^{(4)})^{10}(x,\cdot) - \pi \|_\mathrm{TV} = 0.19,\]
where we define $\|\mu - \nu \|_\mathrm{TV} := \frac{1}{2} \sum_{x\in \mathcal{X}} | \mu(x) - \nu(x) |$ for $\mu, \nu$ being two probability distributions on $\mathcal{X}$. Thus, there is an improvement of approximately $\frac{0.03}{0.22} \approx 14\%$. We also run simulations for $1000$ times on both designs of transition kernels across four different seeds, and plot the empirical cumulative distribution function (CDF) of the results of both models in Figure~\ref{fig:exp_mcmc_cdf}. In addition, we compute the total variation distance between the empirical distributions $\hat\pi$ formed by the samples and the true target distribution $\pi$. From Figure~\ref{fig:exp_mcmc_cdf}, it can be visually inferred that the empirical CDF of the factorized design $(P^{(-4)})^{10} \otimes (P^{(4)})^{10}$ is ``closer'' to the stationary distribution $\pi$ than that of the baseline $P^{10}$, while the total variation distance of the factorized design $(P^{(-4)})^{10} \otimes (P^{(4)})^{10}$ is smaller than that of $P^{10}$ across all seeds except seed $200$ where these two are still comparable.

These results from numerical validation collectively showcase that the factorized design $P^{(-4)} \otimes P^{(4)}$ can be used as a faster substitute of $P$ for approximate sampling of $\pi$.

\begin{figure}[H]
    \centering
    \includegraphics[width=1\linewidth]{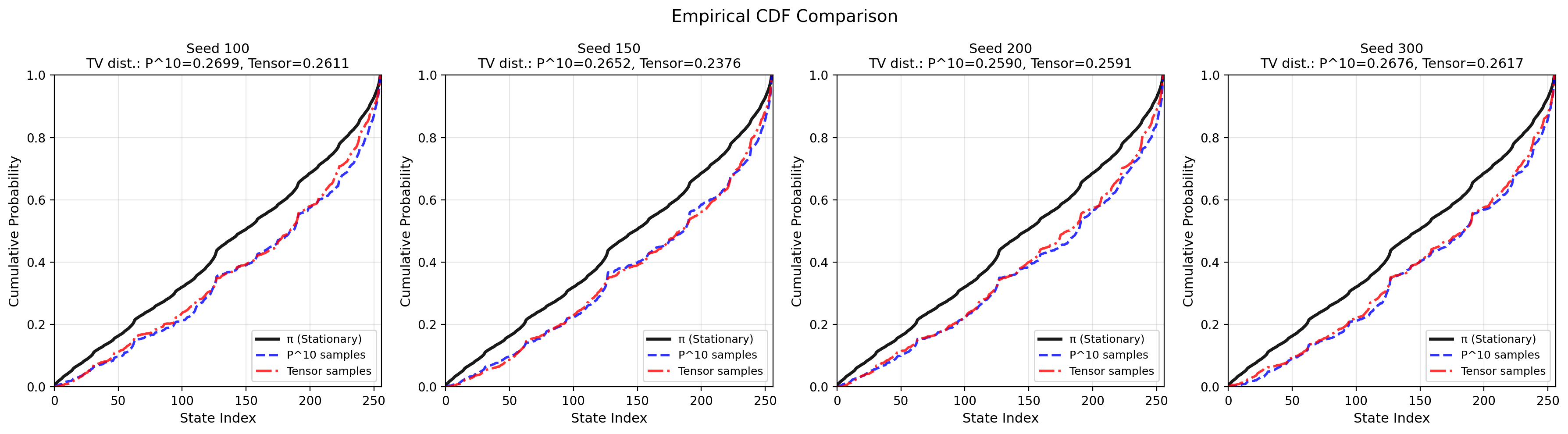}
    \caption{Empirical CDF comparison of MCMC simulation. Here, ``Tensor samples" refer to the 1000 samples generated by the transition matrix $(P^{(-4)})^{10} \otimes (P^{(4)})^{10}$.}
    \label{fig:exp_mcmc_cdf}
\end{figure}

\paragraph{Limitations.} This heuristic argument can be applied to design factorizable transition matrix when the dimension of the original Markov chain is low (e.g. $d=8$ in the above simulation). When the dimension $d$ is large, we may not be able to explicitly compute the target $\pi$, and hence submodular optimization algorithms may not be readily applicable. We also may not be able to simulate the relevant keep-$S$-in and leave-$S$-out matrices due to computational constraints when the dimension is large.
}

\subsection{Experiment results of Section~\ref{sec:entropy_rate}}

In this section, we report the numerical experiment results related to Section~\ref{sec:entropy_rate}, which contains the performance of the heuristic greedy algorithm (see Section 4 of~\cite{MR503866}), the distorted greedy algorithm (see Corollary~\ref{cor:entropyrate}), and the generalized distorted greedy algorithm (see Corollary~\ref{cor:gen_entropyrate}) on the Bernoulli--Laplace level model (see Section~\ref{sec:BLLM}) and the Curie--Weiss model (see Section~\ref{sec:curie_weiss}). For each experiment, we conduct submodular optimization with cardinality constraint $m$, with $m$ ranging from 1 to 10. 

\begin{table}[H]
\centering
\begin{tabular}{|c|c|c|c|c|}
\hline
  & \multicolumn{2}{c|}{\textbf{Greedy}}
  & \multicolumn{2}{c|}{\textbf{Distorted Greedy}} \\
\hline
$m$
 & Subset $S_m$
 & $H\bigl(P^{(S_m)}\bigr)$
 & Subset $S_m$
 & $H\bigl(P^{(S_m)}\bigr)$ \\
\hline
1  & $\{10\}$ & 0.46094 & $\{10\}$ & 0.46094 \\
2  & $\{3,\,10\}$ & 0.83616 & $\{1,\,10\}$ & 0.83573 \\
3  & $\{1,\,3,\,10\}$ & 1.17940 & $\{1,\,2,\,5\}$ & 1.18116 \\
4  & $\{1,\,2,\,3,\,10\}$ & 1.49461 & $\{1,\,2,\,3,\,5\}$ & 1.50706 \\
5  & $\{1,\,2,\,3,\,4,\,10\}$ & 1.77855 & $\{1,\,2,\,3,\,4,\,5\}$ & 1.80193 \\
6  & $\{1,\,2,\,3,\,4,\,5,\,10\}$ & 2.03516 & $\{1,\,2,\,3,\,4,\,5,\,6\}$ & 2.06105 \\
7  & $\{1,\,2,\,3,\,4,\,5,\,6,\,10\}$ & 2.25729 & $\{1,\,2,\,3,\,4,\,5,\,6,\,7\}$ & 2.28328 \\
8  & $\{1,\,2,\,3,\,4,\,5,\,6,\,7,\,10\}$ & 2.43498 & $\{1,\,2,\,3,\,4,\,5,\,6,\,7,\,8\}$ & 2.45453 \\
9  & $\{1,\,2,\,3,\,4,\,5,\,6,\,7,\,8,\,10\}$ & 2.51897 & $\{1,\,2,\,3,\,4,\,5,\,6,\,7,\,8,\,10\}$ & 2.51897 \\
10 & $\{1,\,2,\,3,\,4,\,5,\,6,\,7,\,8,\,10\}$ & 2.51897 & $\{1,\,2,\,3,\,4,\,5,\,6,\,7,\,8,\,10\}$ & 2.51897 \\
\hline
\end{tabular}
\caption{Comparison of the greedy algorithm and the distorted greedy algorithm. Entropy rate of the full chain of the Bernoulli--Laplace level model is $H(P) = 1.96068$.}
\label{tab:BLLM_greedy_comparison_entropyrate}
\end{table}

\begin{figure}[H]
\centering
\begin{subfigure}[b]{0.46\linewidth}
    \centering
    \includegraphics[width=\linewidth]{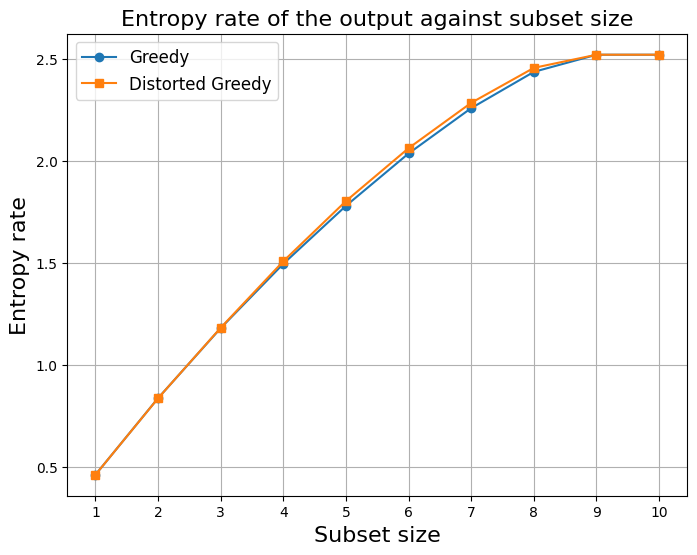}
    \caption{Greedy and Algorithm~\ref{alg:distorted_greedy}}
    \label{fig:BLLM_greedy_entropyrate}
\end{subfigure}
\begin{subfigure}[b]{0.48\linewidth}
    \centering
    \includegraphics[width=\linewidth]{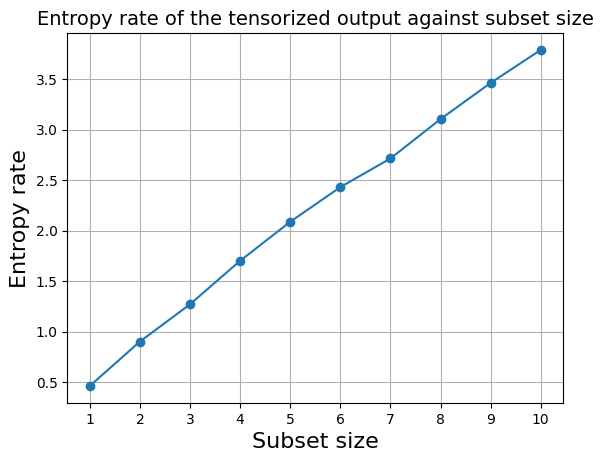}
    \caption{Algorithm~\ref{alg:generalized_distorted_greedy}}
    \label{fig:BLLM_gen_distgrdy_entropyrate}
\end{subfigure}
\caption{Entropy rate against subset size for the three algorithms (B--L model).}
\label{fig:BLLM_entropyrate}
\end{figure}

\begin{table}[H]
\centering
\begin{tabular}{|c|c|c|c|c|}
\hline
\textbf{Cardinality \(m\)} 
  & \textbf{Subset \(S_{m,1}\)} 
  & \textbf{Subset \(S_{m,2}\)} 
  & \textbf{Subset \(S_{m,3}\)} 
  & \textbf{\(H(\otimes_{i=1}^3 P^{(S_{m,i})})\)} \\
\hline
1  & \(\emptyset\) & \(\emptyset\) & \(\{10\}\) & 0.46094 \\
2  & \(\emptyset\) & \(\{7\}\)      & \(\{10\}\) & 0.90046 \\
3  & \(\emptyset\) & \(\{7\}\)      & \(\{8,\,9\}\) & 1.26966 \\
4  & \(\{4\}\)      & \(\{7\}\)      & \(\{8,\,9\}\) & 1.70072 \\
5  & \(\{4\}\)      & \(\{5,\,7\}\)  & \(\{8,\,9\}\) & 2.08692 \\
6  & \(\{4\}\)      & \(\{5,\,6,\,7\}\) & \(\{8,\,9\}\) & 2.43035 \\
7  & \(\{4\}\)      & \(\{5,\,6,\,7\}\) & \(\{8,\,9,\,10\}\) & 2.71405 \\
8  & \(\{3,\,4\}\)  & \(\{5,\,6,\,7\}\) & \(\{8,\,9,\,10\}\) & 3.10451 \\
9  & \(\{1,\,2,\,4\}\)  & \(\{5,\,6,\,7\}\) & \(\{8,\,9,\,10\}\) & 3.46267 \\
10 & \(\{1,\,2,\,3,\,4\}\)  & \(\{5,\,6,\,7\}\) & \(\{8,\,9,\,10\}\) & 3.78968 \\
\hline
\end{tabular}
\caption{Performance evaluation of the generalized distorted greedy algorithm. Entropy rate of the full chain of the Bernoulli--Laplace level model is $H(P) = 1.96068$.}
\label{tab:BLLM_gen_distgrdy_entropyrate}
\end{table}

For the Bernoulli--Laplace level model, Table \ref{tab:BLLM_greedy_comparison_entropyrate} and Figure \ref{fig:BLLM_greedy_entropyrate} show the entropy rates of the output of the greedy algorithm and the distorted greedy algorithm (Algorithm~\ref{alg:distorted_greedy}); Table \ref{tab:BLLM_gen_distgrdy_entropyrate} and Figure \ref{fig:BLLM_gen_distgrdy_entropyrate} show the entropy rates of the tensorized output of the generalized distorted greedy algorithm (Algorithm \ref{alg:generalized_distorted_greedy}).

\begin{table}[H]
\centering
\begin{tabular}{|c|c|c|c|c|}
\hline
 & \multicolumn{2}{c|}{\textbf{Greedy}}
 & \multicolumn{2}{c|}{\textbf{Distorted Greedy}} \\
\hline
$m$
 & Subset $S_m$
 & $H\bigl(P^{(S_m)}\bigr)$
 & Subset $S_m$
 & $H\bigl(P^{(S_m)}\bigr)$ \\
\hline
1  & $\{1\}$                        & 0.29085 & $\{1\}$                        & 0.29085 \\
2  & $\{1,\,10\}$                   & 0.57371 & $\{1,\,10\}$                   & 0.57371 \\
3  & $\{1,\,9,\,10\}$               & 0.83933 & $\{1,\,9,\,10\}$               & 0.83933 \\
4  & $\{1,\,2,\,9,\,10\}$           & 1.09570 & $\{1,\,2,\,9,\,10\}$           & 1.09570 \\
5  & $\{1,\,2,\,6,\,9,\,10\}$       & 1.33953 & $\{1,\,2,\,6,\,9,\,10\}$       & 1.33953 \\
6  & $\{1,\,2,\,4,\,6,\,9,\,10\}$   & 1.57098 & $\{1,\,2,\,4,\,6,\,9,\,10\}$   & 1.57098 \\
7  & $\{1,\,2,\,4,\,6,\,8,\,9,\,10\}$ & 1.78757 & $\{1,\,2,\,4,\,6,\,8,\,9,\,10\}$ & 1.78757 \\
8  & $\{1,\,2,\,3,\,4,\,6,\,8,\,9,\,10\}$ & 1.98500 
   & $\{1,\,2,\,3,\,4,\,6,\,7,\,9,\,10\}$ & 1.98458 \\
9  & $\{1,\,2,\,3,\,4,\,6,\,7,\,8,\,9,\,10\}$ & 2.15793
   & $\{1,\,2,\,3,\,4,\,6,\,7,\,8,\,9,\,10\}$ & 2.15793 \\
10 & $\{1,\,2,\,3,\,4,\,5,\,6,\,7,\,8,\,9,\,10\}$ & 2.29109
   & $\{1,\,2,\,3,\,4,\,5,\,6,\,7,\,8,\,9,\,10\}$ & 2.29109 \\
\hline
\end{tabular}
\caption{Comparison of the greedy algorithm and the distorted greedy algorithm. Entropy rate of the full chain of the Curie--Weiss model is $H(P) = 2.29109$.}
\label{tab:CW_greedy_comparison_entropyrate}
\end{table}

\begin{figure}[H]
\centering
\begin{subfigure}[b]{0.48\linewidth}
    \centering
    \includegraphics[width=\linewidth]{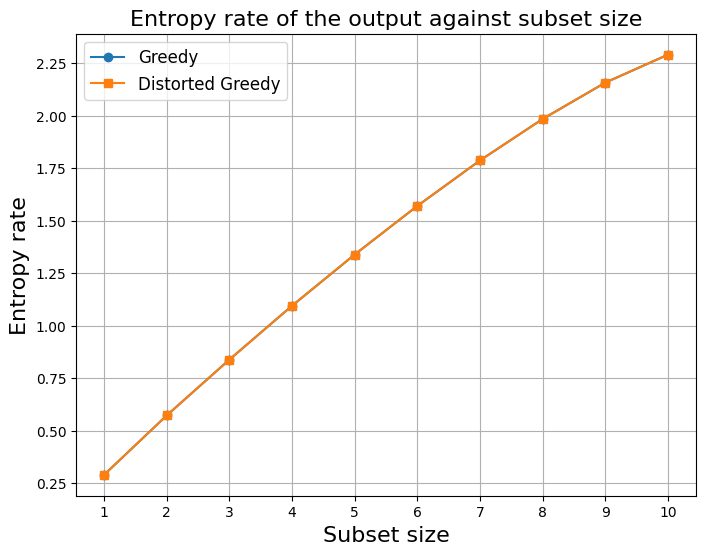}
    \caption{Greedy and Algorithm~\ref{alg:distorted_greedy}}
    \label{fig:CW_greedy_entropyrate}
\end{subfigure}
\begin{subfigure}[b]{0.49\linewidth}
    \centering
    \includegraphics[width=\linewidth]{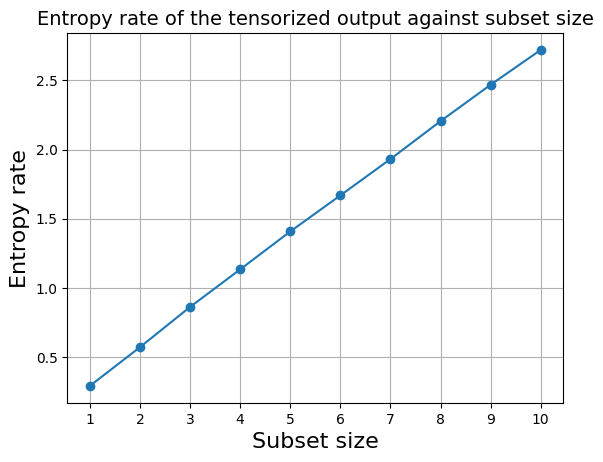}
    \caption{Algorithm~\ref{alg:generalized_distorted_greedy}}
    \label{fig:CW_gen_distgrdy_entropyrate}
\end{subfigure}
\caption{Entropy rate against subset size for the three algorithms (C--W model).}
\label{fig:CW_entropyrate}
\end{figure}

\begin{table}[H]
\centering
\begin{tabular}{|c|c|c|c|c|}
\hline
\textbf{Cardinality \(m\)} 
  & \textbf{Subset \(S_{m,1}\)} 
  & \textbf{Subset \(S_{m,2}\)} 
  & \textbf{Subset \(S_{m,3}\)} 
  & \textbf{\(H(\otimes_{i=1}^3 P^{(S_{m,i})})\)} \\
\hline
1  & \(\{1\}\)           & \(\emptyset\)      & \(\emptyset\)       & 0.29085 \\
2  & \(\{1\}\)           & \(\{7\}\)            & \(\emptyset\)       & 0.57067 \\
3  & \(\{1\}\)           & \(\{7\}\)            & \(\{10\}\)           & 0.86152 \\
4  & \(\{1\}\)           & \(\{5,7\}\)          & \(\{10\}\)           & 1.13316 \\
5  & \(\{1\}\)           & \(\{5,7\}\)          & \(\{9,10\}\)         & 1.40732 \\
6  & \(\{1\}\)           & \(\{5,6,7\}\)        & \(\{9,10\}\)         & 1.66816 \\
7  & \(\{1\}\)           & \(\{5,6,7\}\)        & \(\{8,9,10\}\)       & 1.93090 \\
8  & \(\{1,2\}\)         & \(\{5,6,7\}\)        & \(\{8,9,10\}\)       & 2.20505 \\
9  & \(\{1,2,4\}\)       & \(\{5,6,7\}\)        & \(\{8,9,10\}\)       & 2.46832 \\
10 & \(\{1,2,3,4\}\)     & \(\{5,6,7\}\)        & \(\{8,9,10\}\)       & 2.72011 \\
\hline
\end{tabular}
\caption{Performance evaluation of the generalized distorted greedy algorithm. Entropy rate of the full chain of the Curie--Weiss model is $H(P) = 2.29109$.}
\label{tab:CW_gen_distgrdy_entropyrate}
\end{table}

For the Curie--Weiss model, Table \ref{tab:CW_greedy_comparison_entropyrate} and Figure \ref{fig:CW_greedy_entropyrate} show the entropy rates of the output of the greedy algorithm and the distorted greedy algorithm (Algorithm~\ref{alg:distorted_greedy}); Table \ref{tab:CW_gen_distgrdy_entropyrate} and Figure \ref{fig:CW_gen_distgrdy_entropyrate} show the entropy rates of the tensorized output of the generalized distorted greedy algorithm (Algorithm \ref{alg:generalized_distorted_greedy}).

Notably, in Table~\ref{tab:BLLM_greedy_comparison_entropyrate} and Figure~\ref{fig:BLLM_greedy_entropyrate}, the distorted greedy algorithm outperforms the heuristic greedy algorithm when the cardinality constraint equals to $m = 3, 4, 5, 6, 7, 8$. This is because, in the distorted greedy algorithm, the distortion term $(1 - \frac{1}{m})^{m - (i+1)}$ at each step is different with different cardinality constraint $m$, which results in possibly better or different results than the heuristic greedy algorithm. However, the distorted greedy algorithm does not necessarily select better subset than the heuristic greedy algorithm, see the example of $m=2$ in Table~\ref{tab:BLLM_greedy_comparison_entropyrate} and $m=8$ in Table~\ref{tab:CW_greedy_comparison_entropyrate}.

\subsection{Experiment results of Section~\ref{sec:dist2fact}}
We report the numerical experiment results related to Section~\ref{sec:dist2fact}, which contains the performance of the heuristic greedy algorithm (Section 4 of \cite{MR503866}), the distorted greedy algorithm (Algorithm~\ref{alg:distorted_greedy}), and the generalized distorted greedy algorithm (Algorithm~\ref{alg:generalized_distorted_greedy}) on the Curie--Weiss model as detailed in Section~\ref{sec:curie_weiss}. 

\begin{table}[H]
\centering
\begin{tabular}{|c|c|c|c|c|}
\hline
 & \multicolumn{2}{c|}{\textbf{Greedy}}
 & \multicolumn{2}{c|}{\textbf{Distorted Greedy}} \\
\hline
$m$
 & Subset $S_m$
 & $D\left(P \| P^{(S_m)} \otimes P^{(-S_m)}\right)$
 & Subset $S_m$
 & $D\left(P \| P^{(S_m)} \otimes P^{(-S_m)}\right)$ \\
\hline
1 & $\{6\}$ & 0.14837 & $\{6\}$ & 0.14837 \\
2 & $\{2,6\}$ & 0.24497 & $\{3,10\}$ & 0.24496 \\
3 & $\{2,6,9\}$ & 0.30927 & $\{3,7\}$ & 0.24525 \\
4 & $\{2,5,6,9\}$ & 0.34590 & $\{2,7,10\}$ & 0.30905 \\
5 & $\{2,3,5,6,9\}$ & 0.35758 & $\{2,3,6,10\}$ & 0.34590 \\
\hline
\end{tabular}
\caption{Comparison of the greedy algorithm and the distorted greedy algorithm \textcolor{black}{(C--W model)}.}
\label{tab:CW_greedy_comparison_dist2fact}
\end{table}

\begin{table}[H]
\centering
\begin{tabular}{|c|c|c|c|c|}
\hline
\(m\)
  & \textbf{Subset \(S_{m,1}\)} 
  & \textbf{Subset \(S_{m,2}\)} 
  & \textbf{Subset \(S_{m,3}\)} 
  & $D\left(P \| \left(\otimes_{i=1}^3 P^{(S_{m, i})}\right) \otimes P^{(-\cup_{i=1}^3 S_{m, i})}\right)$ \\
\hline
1  & \(\emptyset\)    & \(\{6\}\)          & \(\emptyset\)      & 0.14836 \\
2  & \(\emptyset\)    & \(\{7\}\)          & \(\{8\}\)          & 0.25388 \\
3  & \(\{4\}\)       & \(\{7\}\)          & \(\{8\}\)          & 0.33529 \\
4  & \(\{4\}\)       & \(\{5,7\}\)        & \(\{8\}\)          & 0.39056 \\
5  & \(\{2,4\}\)     & \(\{5,7\}\)        & \(\{8\}\)          & 0.43104 \\
6  & \(\{2,4\}\)     & \(\{5,7\}\)        & \(\{8,10\}\)       & 0.45978 \\
7  & \(\{2,4\}\)     & \(\{5,6,7\}\)      & \(\{8,10\}\)       & 0.46887 \\
8  & \(\{2,4\}\)     & \(\{5,6,7\}\)      & \(\{8,10\}\)       & 0.46887 \\
9  & \(\{2,4\}\)     & \(\{5,6,7\}\)      & \(\{8,10\}\)       & 0.46887 \\
10 & \(\{2,4\}\)     & \(\{5,6,7\}\)      & \(\{8,10\}\)       & 0.46887 \\
\hline
\end{tabular}
\caption{Performance evaluation of the generalized distorted greedy algorithm \textcolor{black}{(C--W model)}.}
\label{tab:CW_gen_distgrdy_dist2fact}
\end{table}

\begin{figure}[H]
\centering
\begin{subfigure}[b]{0.46\linewidth}
    \centering
    \includegraphics[width=\linewidth]{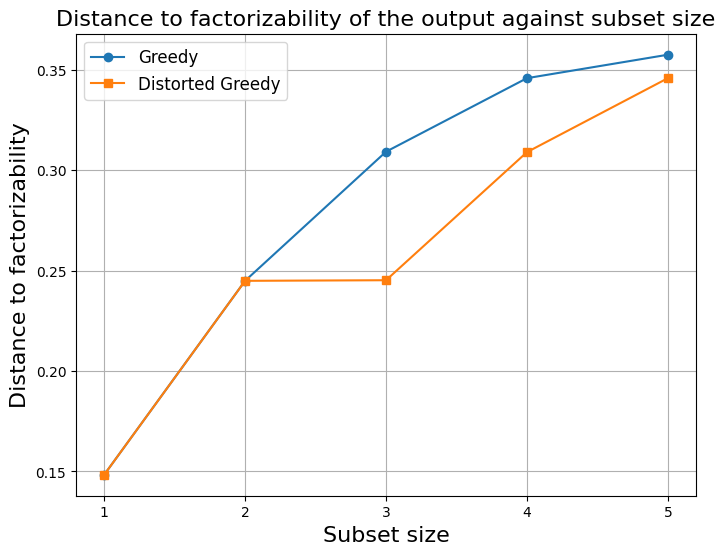}
    \caption{Greedy and Algorithm~\ref{alg:distorted_greedy}}
    \label{fig:CW_greedy_distfact}
\end{subfigure}
\begin{subfigure}[b]{0.48\linewidth}
    \centering
    \includegraphics[width=\linewidth]{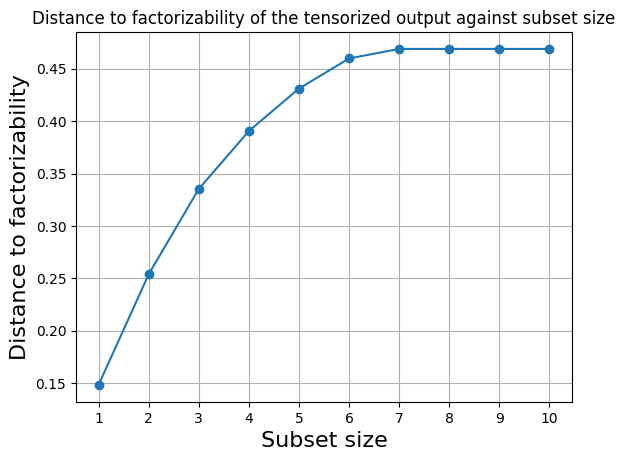}
    \caption{Algorithm~\ref{alg:generalized_distorted_greedy}}
    \label{fig:CW_gen_distgrdy_distfact}
\end{subfigure}
\caption{Distance to factorizability against subset size for the three algorithms \textcolor{black}{(C--W model)}.}
\label{fig:CW_dist2fact}
\end{figure}

For the experiments related to heuristic greedy and distorted greedy algorithms, since the map $S \mapsto D(P \| P^{(S)} \otimes P^{(-S)})$ is symmetric, we conduct submodular maximization with cardinality constraint $m$, with $m$ only ranging from 1 to 5. The results are shown on Table~\ref{tab:CW_greedy_comparison_dist2fact} and Figure~\ref{fig:CW_greedy_distfact}. 
\textcolor{black}{We observe that the distorted greedy algorithm does not consistently outperform the heuristic greedy algorithm (e.g., at $m=2$). This is expected in the context of non-monotone submodular maximization: while the distorted greedy algorithm provides a worst-case approximation guarantee (Corollary~\ref{cor:dist2fact}), the heuristic greedy algorithm may still exploit local gradients effectively in specific instances.}
We also conduct the generalized distorted greedy algorithm on the Curie--Weiss model as detailed in Corollary~\ref{cor:gen_dist2fact} with cardinality constraint $m$ ranging from 1 to 10, and the results are shown on Table~\ref{tab:CW_gen_distgrdy_dist2fact} and Figure~\ref{fig:CW_gen_distgrdy_distfact}.

We conduct similar numerical experiments \textcolor{black}{of the three algorithms} on the Bernoulli--Laplace level model (see Section~\ref{sec:BLLM}). Among all cardinality constraints, the greedy algorithm and the distorted greedy algorithm output $S_m =\{10\}$, and the generalized distorted greedy algorithm outputs $S_{m, 1} = S_{m, 2} = \emptyset$, $S_{m, 3} = \{10\}$. The reason behind it is that for a 10-dimensional Markov chain \textcolor{black}{associated with the Bernoulli--Laplace level model}, the coordinate 10 is ``far'' from other coordinates, \textcolor{black}{which leads to the sparse result.}

\subsection{Experiment results of Section~\ref{sec:dist2indp}}

We report the numerical experiment results related to Section~\ref{sec:dist2indp}, which contains the performance of the heuristic greedy algorithm (see Section 4 of~\cite{MR503866}), the distorted greedy algorithm (see Corollary~\ref{cor:dist2indp}), and the generalized distorted greedy algorithm (see Corollary~\ref{cor:gen_dist2indp}) on the Bernoulli--Laplace level model (see Section~\ref{sec:BLLM}) and the Curie--Weiss model (see Section~\ref{sec:curie_weiss}). For each experiment, we conduct supermodular minimization with different cardinality constraint $m$'s.

For the Bernoulli--Laplace level model, Table \ref{tab:BLLM_greedy_comparison_dist2indp} and Figure \ref{fig:BLLM_greedy_dist2indp} show the distance to independence of the outputs.
\textcolor{black}{We observe that the distorted greedy algorithm consistently identifies subsets with a smaller distance to independence (lower $\mathbb{I}(P^{(S)})$) compared to the standard greedy algorithm.}
Table \ref{tab:BLLM_gen_distgrdy_dist2indp} and Figure \ref{fig:BLLM_gen_distgrdy_dist2indp} show the distance to independence of the tensorized outputs of the generalized distorted greedy algorithm (Algorithm \ref{alg:generalized_distorted_greedy}) \textcolor{black}{subject to varying cardinality constraints}.

\begin{table}[H]
\centering
\begin{tabular}{|c|c|c|c|c|}
\hline
 & \multicolumn{2}{c|}{\textbf{Greedy}} & \multicolumn{2}{c|}{\textbf{Distorted Greedy}} \\
\hline
$m$ & Subset $S_m$ & $\mathbb{I}(P^{(S_m)})$ & Subset $S_m$ & $\mathbb{I}(P^{(S_m)})$ \\
\hline
2  & $\{1,\,10\}$                   & 0.05140 & $\{1,\,2\}$                   & 0.03406 \\
3  & $\{1,\,2,\,10\}$               & 0.13505 & $\{1,\,2,\,3\}$               & 0.10318 \\
4  & $\{1,\,2,\,3,\,10\}$           & 0.24989 & $\{1,\,2,\,3,\,4\}$           & 0.20793 \\
5  & $\{1,\,2,\,3,\,4,\,10\}$       & 0.39701 & $\{1,\,2,\,3,\,4,\,5\}$       & 0.34753 \\
6  & $\{1,\,2,\,3,\,4,\,5,\,10\}$   & 0.57523 & $\{1,\,2,\,3,\,4,\,5,\,6\}$   & 0.52441 \\
7  & $\{1,\,2,\,3,\,4,\,5,\,6,\,10\}$ & 0.78911 & $\{1,\,2,\,3,\,4,\,5,\,6,\,7\}$ & 0.74171 \\
8  & $\{1,\,2,\,3,\,4,\,5,\,6,\,7,\,10\}$ & 1.05094 & $\{1,\,2,\,3,\,4,\,5,\,6,\,7,\,8\}$ & 1.01576 \\
9  & $\{1,\,2,\,3,\,4,\,5,\,6,\,7,\,8,\,10\}$ & 1.41226 & $\{1,\,2,\,3,\,4,\,5,\,6,\,7,\,8,\,10\}$ & 1.41226 \\
10 & $\{1,\,2,\,3,\,4,\,5,\,6,\,7,\,8,\,9,\,10\}$ & 2.41825 & $\{1,\,2,\,3,\,4,\,5,\,6,\,7,\,8,\,9,\,10\}$ & 2.41825 \\
\hline
\end{tabular}
\caption{Comparison of the greedy algorithm and the distorted greedy algorithm (B--L model).}
\label{tab:BLLM_greedy_comparison_dist2indp}
\end{table}

\begin{table}[H]
\centering
\begin{tabular}{|c|c|c|c|c|}
\hline
\(m\) & \textbf{Subset \(S_{m,1}\)} & \textbf{Subset \(S_{m,2}\)} & \textbf{Subset \(S_{m,3}\)} & $\mathbb I \left(\otimes_{i=1}^3 P^{(S_{m, i})} \right)$ \\
\hline
4  & \(\{1,\,2\}\)       & \(\{5\}\)       & \(\{8\}\)       & 0.03406 \\
5  & \(\{1,\,2\}\)       & \(\{5,\,6\}\)   & \(\{8\}\)       & 0.07999 \\
6  & \(\{1,\,2\}\)       & \(\{5,\,6\}\)   & \(\{8,\,9\}\)   & 0.14286 \\
7  & \(\{1,\,2,\,3\}\)   & \(\{5,\,6\}\)   & \(\{8,\,9\}\)   & 0.21199 \\
8  & \(\{1,\,2,\,3\}\)   & \(\{5,\,6,\,7\}\) & \(\{8,\,9\}\)   & 0.30727 \\
9  & \(\{1,\,2,\,3,\,4\}\) & \(\{5,\,6,\,7\}\) & \(\{8,\,9\}\)   & 0.41202 \\
10 & \(\{1,\,2,\,3,\,4\}\) & \(\{5,\,6,\,7\}\) & \(\{8,\,9,\,10\}\) & 0.58925 \\
\hline
\end{tabular}
\caption{Performance evaluation of the generalized distorted greedy algorithm (B--L model).}
\label{tab:BLLM_gen_distgrdy_dist2indp}
\end{table}

\begin{figure}[H]
\centering
\begin{subfigure}[b]{0.46\linewidth}
    \centering
    \includegraphics[width=\linewidth]{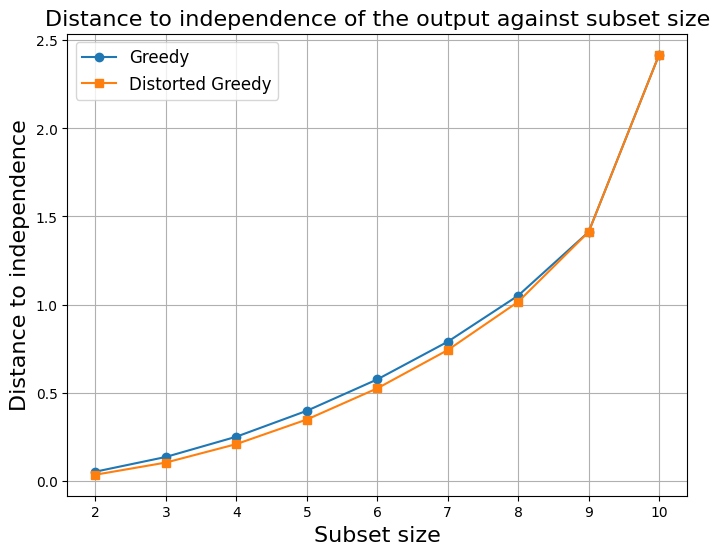}
    \caption{Greedy and Algorithm~\ref{alg:distorted_greedy}}
    \label{fig:BLLM_greedy_dist2indp}
\end{subfigure}
\begin{subfigure}[b]{0.52\linewidth}
    \centering
    \includegraphics[width=\linewidth]{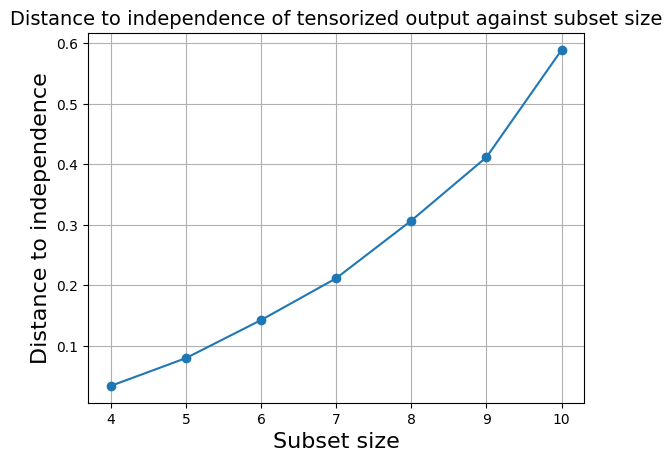}
    \caption{Algorithm~\ref{alg:generalized_distorted_greedy}}
    \label{fig:BLLM_gen_distgrdy_dist2indp}
\end{subfigure}
\caption{Distance to independence against subset size for the three algorithms (B--L model).}
\label{fig:BLLM_dist2indp}
\end{figure}

\begin{table}[H]
\centering
\begin{tabular}{|c|c|c|c|c|}
\hline
 & \multicolumn{2}{c|}{\textbf{Greedy}} & \multicolumn{2}{c|}{\textbf{Distorted Greedy}} \\
\hline
$m$ & Subset $S_m$ & $\mathbb{I}(P^{(S_m)})$ & Subset $S_m$ & $\mathbb{I}(P^{(S_m)})$ \\
\hline
2  & $\{4,\,10\}$             & 0.00757 & $\{1,\,7\}$             & 0.00757 \\
3  & $\{4,\,7,\,10\}$         & 0.02350 & $\{1,\,6,\,10\}$         & 0.02398 \\
4  & $\{2,\,4,\,7,\,10\}$      & 0.04889 & $\{1,\,5,\,7,\,10\}$      & 0.04961 \\
5  & $\{2,\,4,\,6,\,7,\,10\}$   & 0.08592 & $\{1,\,3,\,5,\,7,\,10\}$  & 0.08591 \\
6  & $\{2,\,4,\,6,\,7,\,8,\,10\}$ & 0.13555 & $\{1,\,3,\,5,\,7,\,8,\,10\}$ & 0.13533 \\
7  & $\{2,\,3,\,4,\,6,\,7,\,8,\,10\}$ & 0.19989 & $\{1,\,3,\,4,\,5,\,7,\,8,\,10\}$ & 0.20017 \\
8  & $\{2,\,3,\,4,\,5,\,6,\,7,\,8,\,10\}$ & 0.28356 & $\{1,\,3,\,4,\,5,\,6,\,7,\,8,\,10\}$ & 0.28399 \\
9  & $\{2,\,3,\,4,\,5,\,6,\,7,\,8,\,9,\,10\}$ & 0.39102 & $\{1,\,3,\,4,\,5,\,6,\,7,\,8,\,9,\,10\}$ & 0.39191 \\
10 & $\{1,\,2,\,3,\,4,\,5,\,6,\,7,\,8,\,9,\,10\}$ & 0.53813 & $\{1,\,2,\,3,\,4,\,5,\,6,\,7,\,8,\,9,\,10\}$ & 0.53813 \\
\hline
\end{tabular}
\caption{Comparison of the greedy algorithm and the distorted greedy algorithm (C--W model).}
\label{tab:CW_greedy_comparison_dist2indp}
\end{table}

\begin{table}[H]
\centering
\begin{tabular}{|c|c|c|c|c|}
\hline
\(m\) & \textbf{Subset \(S_{m,1}\)} & \textbf{Subset \(S_{m,2}\)} & \textbf{Subset \(S_{m,3}\)} & $\mathbb I \left(\otimes_{i=1}^3 P^{(S_{m, i})} \right)$ \\
\hline
4  & \(\{1\}\)           & \(\{5,\,7\}\)       & \(\{8\}\)         & 0.00778 \\
5  & \(\{1,\,4\}\)       & \(\{5,\,7\}\)       & \(\{8\}\)         & 0.01556 \\
6  & \(\{1,\,4\}\)       & \(\{5,\,7\}\)       & \(\{8,\,10\}\)     & 0.02376 \\
7  & \(\{1,\,3,\,4\}\)   & \(\{5,\,7\}\)       & \(\{8,\,10\}\)     & 0.04172 \\
8  & \(\{1,\,3,\,4\}\)   & \(\{5,\,6,\,7\}\)   & \(\{8,\,10\}\)     & 0.06029 \\
9  & \(\{1,\,3,\,4\}\)   & \(\{5,\,6,\,7\}\)   & \(\{8,\,9,\,10\}\)  & 0.07972 \\
10 & \(\{1,\,2,\,3,\,4\}\) & \(\{5,\,6,\,7\}\)   & \(\{8,\,9,\,10\}\)  & 0.10911 \\
\hline
\end{tabular}
\caption{Performance evaluation of the generalized distorted greedy algorithm (C--W model).}
\label{tab:CW_gen_distgrdy_dist2indp}
\end{table}

\begin{figure}[H]
\centering
\begin{subfigure}[b]{0.46\linewidth}
    \centering
    \includegraphics[width=\linewidth]{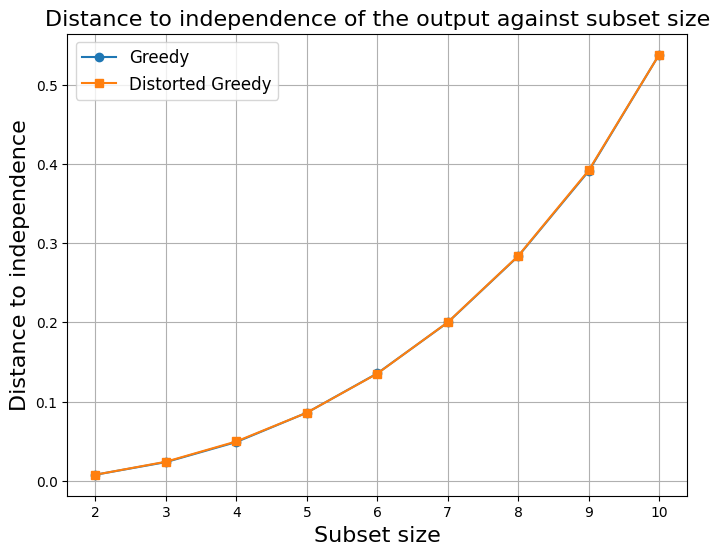}
    \caption{Greedy and Algorithm~\ref{alg:distorted_greedy}}
    \label{fig:CW_greedy_dist2indp}
\end{subfigure}
\begin{subfigure}[b]{0.52\linewidth}
    \centering
    \includegraphics[width=\linewidth]{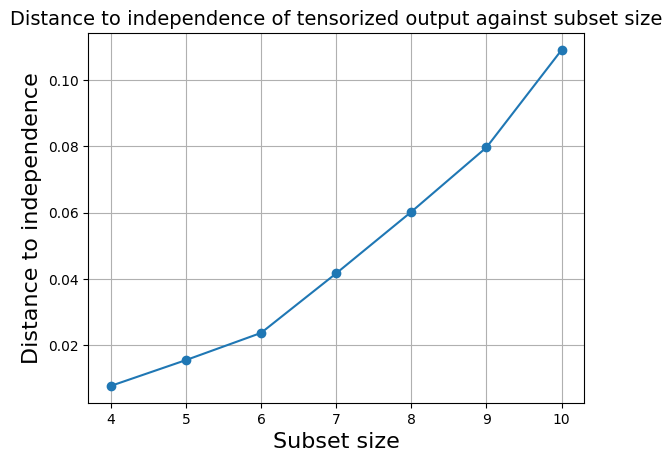}
    \caption{Algorithm~\ref{alg:generalized_distorted_greedy}}
    \label{fig:CW_gen_distgrdy_dist2indp}
\end{subfigure}
\caption{Distance to independence against subset size for the three algorithms (C--W model).}
\label{fig:CW_dist2indp}
\end{figure}

For the Curie--Weiss model, Table \ref{tab:CW_greedy_comparison_dist2indp} and Figure \ref{fig:CW_greedy_dist2indp} show the distance of independence of the outputs of the greedy algorithm and the distorted greedy algorithm (Algorithm~\ref{alg:distorted_greedy}), \textcolor{black}{where the distorted greedy algorithm is slightly worse than the heuristic greedy algorithm.} Table \ref{tab:CW_gen_distgrdy_dist2indp} and Figure \ref{fig:CW_gen_distgrdy_dist2indp} show the distance of independence of the tensorized outputs of the generalized distorted greedy algorithm (Algorithm \ref{alg:generalized_distorted_greedy}).

In addition, we report the numerical experiment results related to the distance to independence of the complement set, as detailed in Section~\ref{sec:dist2indp_c} and Section~\ref{sec:gen_dist2indp_c}. The performance of the greedy algorithm on the two models is shown in Table~\ref{tab:greedy_dist2indp_c} and Figure~\ref{fig:greedy_dist2indp_c}, while the performance of the generalized distorted greedy algorithm can be seen from Table~\ref{tab:gen_distgrdy_dist2indp_c} and Figure~\ref{fig:gen_distgrdy_dist2indp_c}. \textcolor{black}{Note that under the settings of Section~\ref{sec:dist2indp_c}, the heuristic greedy algorithm (Section 4 of \cite{MR503866}) gives $(1 - e^{-1})$-approximation guarantee as the objective function is monotone, and as such we do not conduct experiments using the distorted greedy algorithm.}

\begin{table}[H]
\centering
\begin{tabular}{|c|c|c|c|c|}
\hline
 & \multicolumn{2}{c|}{\textbf{Bernoulli--Laplace}} & \multicolumn{2}{c|}{\textbf{Curie--Weiss}} \\
\hline
\(m\) & \textbf{Subset $S_m$} & $\mathbb{I} (P^{(-S_m)})$ & \textbf{Subset $S_m$} & $\mathbb{I} (P^{(-S_m)})$ \\
\hline
1 & \(\{9\}\) & 1.41226 & \(\{1\}\) & 0.39102 \\
2 & \(\{9,\,10\}\) & 1.01576 & \(\{1,\,10\}\) & 0.28314 \\
3 & \(\{8,\,9,\,10\}\) & 0.74171 & \(\{1,\,5,\,10\}\) & 0.19981 \\
4 & \(\{7,\,8,\,9,\,10\}\) & 0.52441 & \(\{1,\,5,\,7,\,10\}\) & 0.13517 \\
5 & \(\{6,\,7,\,8,\,9,\,10\}\) & 0.34753 & \(\{1,\,3,\,5,\,7,\,10\}\) & 0.08523 \\
6 & \(\{5,\,6,\,7,\,8,\,9,\,10\}\) & 0.20793 & \(\{1,\,3,\,5,\,7,\,8,\,10\}\) & 0.04845 \\
7 & \(\{4,\,5,\,6,\,7,\,8,\,9,\,10\}\) & 0.10318 & \(\{1,\,3,\,4,\,5,\,7,\,8,\,10\}\) & 0.02304 \\
8 & \(\{3,\,4,\,5,\,6,\,7,\,8,\,9,\,10\}\) & 0.03406 & \(\{1,\,3,\,4,\,5,\,7,\,8,\,9,\,10\}\) & 0.00736 \\
\hline
\end{tabular}
\caption{Performance evaluation of greedy algorithm.}
\label{tab:greedy_dist2indp_c}
\end{table}

\begin{table}[H]
\centering
\begin{tabular}{|c|cccc|cccc|}
\hline
 & \multicolumn{4}{c|}{\textbf{Bernoulli--Laplace}} & \multicolumn{4}{c|}{\textbf{Curie--Weiss}} \\
\hline
    \(m\) & \(S_{m,1}\) & \(S_{m,2}\) & \(S_{m,3}\) & $\mathbb{I}\left(\otimes_{i=1}^3 P^{(-S_{m, i})}\right)$ & \(S_{m,1}\) & \(S_{m,2}\) & \(S_{m,3}\) & $\mathbb{I}\left(\otimes_{i=1}^3 P^{(-S_{m, i})}\right)$ \\
\hline
1 & $\emptyset$     & $\emptyset$      & \(\{10\}\)      & 0.41202 & \(\{2\}\)      & $\emptyset$      & $\emptyset$       & 0.07972 \\
2 & \(\{4\}\)     & $\emptyset$      & \(\{10\}\)      & 0.30727 & \(\{2\}\)      & $\emptyset$      & \(\{9\}\)      & 0.06029 \\
3 & \(\{4\}\)     & \(\{7\}\)     & \(\{10\}\)      & 0.21198 & \(\{2\}\)      & \(\{6\}\)     & \(\{9\}\)      & 0.04172 \\
4 & \(\{3,\,4\}\) & \(\{7\}\)     & \(\{10\}\)      & 0.14286 & \(\{2,\,3\}\) & \(\{6\}\)     & \(\{9\}\)      & 0.02376 \\
5 & \(\{3,\,4\}\) & \(\{7\}\)     & \(\{9,\,10\}\) & 0.07999 & \(\{2,\,3\}\) & \(\{6\}\)     & \(\{9,\,10\}\) & 0.01556 \\
6 & \(\{3,\,4\}\) & \(\{5,\,7\}\) & \(\{9,\,10\}\) & 0.03406 & \(\{1,\,2,\,3\}\) & \(\{6\}\) & \(\{9,\,10\}\) & 0.00778 \\
\hline
\end{tabular}
\caption{Performance evaluation of the generalized distorted greedy algorithm.}
\label{tab:gen_distgrdy_dist2indp_c}
\end{table}

\begin{figure}[H]
\centering
\begin{subfigure}[b]{0.49\linewidth}
    \centering
    \includegraphics[width=\linewidth]{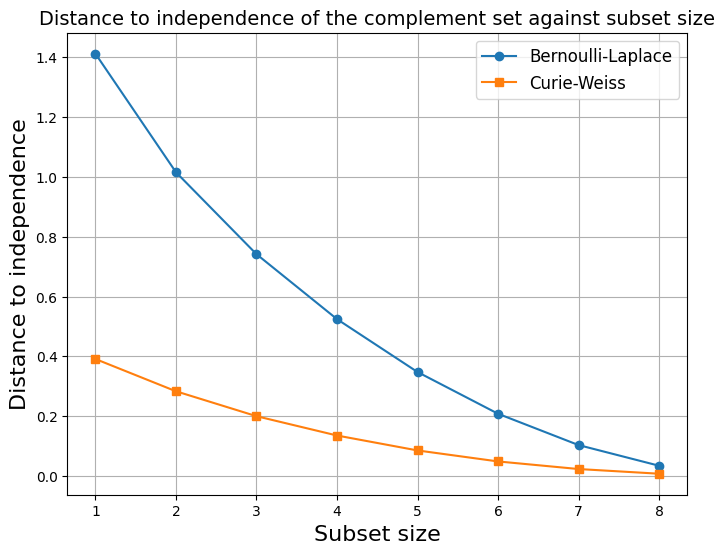}
    \caption{Greedy}
    \label{fig:greedy_dist2indp_c}
\end{subfigure}
\begin{subfigure}[b]{0.49\linewidth}
    \centering
    \includegraphics[width=\linewidth]{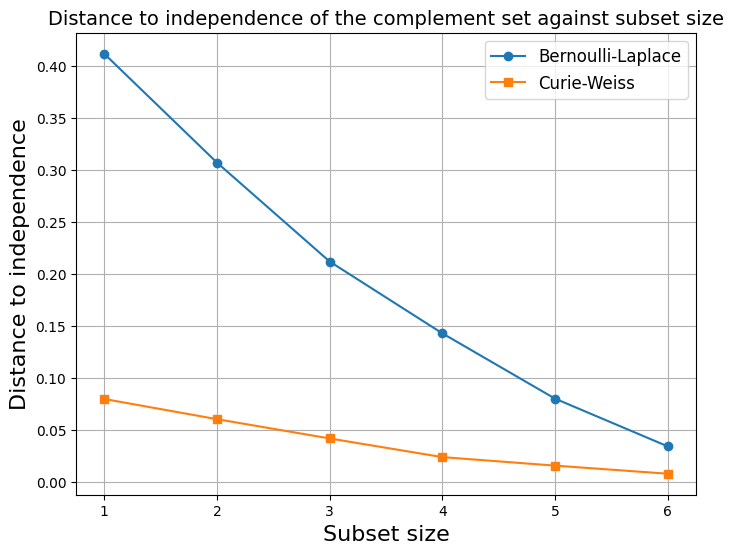}
    \caption{Algorithm~\ref{alg:generalized_distorted_greedy}}
    \label{fig:gen_distgrdy_dist2indp_c}
\end{subfigure}
\caption{Distance to independence of the complement set against subset size.}
\label{fig:dist2indp_c}
\end{figure}

\subsection{Experiment results of Section~\ref{sec:dist2stat}}
We first report the numerical experiment results related to Algorithm~\ref{alg:batch_greedy}. For both the Bernoulli--Laplace level model and the Curie--Weiss model, we consider the following two configurations of the batch greedy algorithm to maximize $D(P^{(S)} \| \Pi^{(S)})$ subject to the cardinality constraint $m$: \begin{itemize}
    \item Approach 1: $l=m$ and $q_i = 1$ for $i\in \llbracket l \rrbracket$;
    \item Approach 2: $l = \lceil \frac{m}{2} \rceil$, $q_i = 2$ for $i \in \llbracket l-1 \rrbracket$; $q_l = 2$ if $m$ is even, $q_l = 1$ if $m$ is odd.
\end{itemize}
In Approach 1, we recover the heuristic greedy algorithm since we are adding one element per iteration. We compare the performance of Approach 1 and Approach 2 for both models, and the results are shown in Table~\ref{tab:bg_BLLM_dist2stat} and Table~\ref{tab:bg_CW_dist2stat}. Although the stationary distribution $\pi$ of the Bernoulli--Laplace level model and the Curie--Weiss model are not of product form, we still apply the heuristic distorted greedy algorithm as in Corollary~\ref{cor:dist2stat}, and the results are summarized in Table~\ref{tab:distgrdy_dist2stat}, \textcolor{black}{and the theoretical guarantee in Corollary \ref{cor:dist2stat} does not apply.} The comparison of these algorithms on the two models is shown in Figure~\ref{fig:bg_dist2stat}, \textcolor{black}{where for both Bernoulli--Laplace and Curie--Weiss models, the distorted greedy algorithm gives the best results, as the resulting coordinate subset chain is closest to the equilibrium under every cardinality constraint $m$, while Approach 1 gives the worst results.}

\begin{table}[H]
\centering
\begin{tabular}{|c|cc|cc|}
\hline
 & \multicolumn{2}{c|}{\textbf{Approach 1}} & \multicolumn{2}{c|}{\textbf{Approach 2}} \\
\hline
\(m\) & \textbf{Subset \(S_l\)} & $D(P^{(S_l)} \| \Pi^{(S_l)})$ & \textbf{Subset \(S_l\)} & $D(P^{(S_l)} \| \Pi^{(S_l)})$ \\
\hline
1  & \(\{1\}\)                          & 0.26693  & \(\{1\}\)                          & 0.26693  \\
2  & \(\{1,2\}\)                        & 0.59421  & \(\{1,2\}\)                        & 0.59421  \\
3  & \(\{1,2,7\}\)                      & 0.98856  & \(\{1,2,7\}\)                      & 0.98856  \\
4  & \(\{1,2,7,10\}\)                   & 1.47330  & \(\{1,2,4,7\}\)                    & 1.46082  \\
5  & \(\{1,2,7,9,10\}\)                 & 2.07889  & \(\{1,2,4,7,10\}\)                 & 2.03226  \\
6  & \(\{1,2,7,8,9,10\}\)               & 2.85834  & \(\{1,2,4,7,9,10\}\)               & 2.73225  \\
7  & \(\{1,2,6,7,8,9,10\}\)             & 3.70196  & \(\{1,2,4,7,8,9,10\}\)             & 3.64286  \\
8  & \(\{1,2,5,6,7,8,9,10\}\)           & 4.69790  & \(\{1,2,4,6,7,8,9,10\}\)           & 4.65621  \\
9  & \(\{1,2,4,5,6,7,8,9,10\}\)         & 5.91911  & \(\{1,2,4,5,6,7,8,9,10\}\)         & 5.91911  \\
10 & \(\{1,2,3,4,5,6,7,8,9,10\}\)       & 7.56130  & \(\{1,2,3,4,5,6,7,8,9,10\}\)       & 7.56130  \\
\hline
\end{tabular}
\caption{Comparison of different configurations of the batch greedy algorithm (B--L model).}
\label{tab:bg_BLLM_dist2stat}
\end{table}

\begin{table}[H]
\centering
\begin{tabular}{|c|cc|cc|}
\hline
 & \multicolumn{2}{c|}{\textbf{Approach 1}} & \multicolumn{2}{c|}{\textbf{Approach 2}} \\
\hline
\(m\) & \textbf{Subset \(S_l\)} & $D(P^{(S_l)} \| \Pi^{(S_l)})$ & \textbf{Subset \(S_l\)} & $D(P^{(S_l)} \| \Pi^{(S_l)})$ \\
\hline
1  & \(\{6\}\)                           & 0.40245  & \(\{6\}\)                           & 0.40245  \\
2  & \(\{3,\,6\}\)                       & 0.81082  & \(\{5,\,6\}\)                       & 0.80739  \\
3  & \(\{3,\,6,\,8\}\)                    & 1.22606  & \(\{5,\,6,\,8\}\)                    & 1.22234  \\
4  & \(\{3,\,4,\,6,\,8\}\)                 & 1.64626  & \(\{3,\,5,\,6,\,8\}\)                 & 1.64615  \\
5  & \(\{3,\,4,\,6,\,8,\,9\}\)              & 2.07613  & \(\{2,\,3,\,5,\,6,\,8\}\)              & 2.07601  \\
6  & \(\{2,\,3,\,4,\,6,\,8,\,9\}\)           & 2.51741  & \(\{2,\,3,\,5,\,6,\,8,\,9\}\)           & 2.51771  \\
7  & \(\{2,\,3,\,4,\,5,\,6,\,8,\,9\}\)        & 2.97051  & \(\{2,\,3,\,4,\,5,\,6,\,8,\,9\}\)        & 2.97051  \\
8  & \(\{1,\,2,\,3,\,4,\,6,\,8,\,9\}\)        & 3.44141  & \(\{2,\,3,\,4,\,5,\,6,\,7,\,8,\,9\}\)    & 3.44085  \\
9  & \(\{1,\,2,\,3,\,4,\,6,\,8,\,9,\,10\}\)    & 3.93647  & \(\{1,\,2,\,3,\,4,\,5,\,6,\,7,\,8,\,9\}\) & 3.93568  \\
10 & \(\{1,\,2,\,3,\,4,\,5,\,6,\,7,\,8,\,9,\,10\}\) & 4.46975  & \(\{1,\,2,\,3,\,4,\,5,\,6,\,7,\,8,\,9,\,10\}\) & 4.46975  \\
\hline
\end{tabular}
\caption{Comparison of different configurations of the batch greedy algorithm (C--W model).}
\label{tab:bg_CW_dist2stat}
\end{table}

\begin{table}[H]
\centering
\begin{tabular}{|c|cc|cc|}
\hline
 & \multicolumn{2}{c|}{\textbf{Bernoulli--Laplace level model}} & \multicolumn{2}{c|}{\textbf{Curie--Weiss model}} \\
\hline
\(m\) & \textbf{Subset \(S_m\)} & $D(P^{(S_m)} \| \Pi^{(S_m)})$ & \textbf{Subset \(S_m\)} & $D(P^{(S_m)} \| \Pi^{(S_m)})$ \\
\hline
1  & \(\{10\}\)                             & 0.23219  & \(\{1\}\)                               & 0.39435  \\
2  & \(\{1,\,10\}\)                         & 0.57719    & \(\{1,\,10\}\)                          & 0.79669   \\
3  & \(\{1,\,2,\,10\}\)                      & 0.98552   & \(\{1,\,2,\,10\}\)                       & 1.20915   \\
4  & \(\{1,\,2,\,3,\,5\}\)                   & 1.45314   & \(\{1,\,2,\,9,\,10\}\)                   & 1.63086    \\
5  & \(\{1,\,2,\,3,\,4,\,5\}\)                & 1.99871   & \(\{1,\,2,\,3,\,9,\,10\}\)               & 2.06307    \\
6  & \(\{1,\,2,\,3,\,4,\,5,\,6\}\)             & 2.63821   & \(\{1,\,2,\,3,\,8,\,9,\,10\}\)            & 2.50704   \\
7  & \(\{1,\,2,\,3,\,4,\,5,\,6,\,7\}\)          & 3.39168    & \(\{1,\,2,\,3,\,4,\,8,\,9,\,10\}\)         & 2.96498    \\
8  & \(\{1,\,2,\,3,\,4,\,5,\,6,\,7,\,8\}\)       & 4.30094    & \(\{1,\,2,\,3,\,4,\,5,\,8,\,9,\,10\}\)      & 3.43971    \\
9  & \(\{1,\,2,\,3,\,4,\,5,\,6,\,7,\,8,\,10\}\)   & 5.46950    & \(\{1,\,2,\,3,\,4,\,5,\,6,\,8,\,9,\,10\}\)   & 3.93647    \\
10 & \(\{1,\,2,\,3,\,4,\,5,\,6,\,7,\,8,\,9,\,10\}\) & 7.56130     & \(\{1,\,2,\,3,\,4,\,5,\,6,\,7,\,8,\,9,\,10\}\) & 4.46975    \\
\hline
\end{tabular}
\caption{Performance evaluation of the distorted greedy algorithm.}
\label{tab:distgrdy_dist2stat}
\end{table}

\begin{figure}[H]
\centering
\begin{subfigure}[b]{0.47\linewidth}
    \centering
    \includegraphics[width=\linewidth]{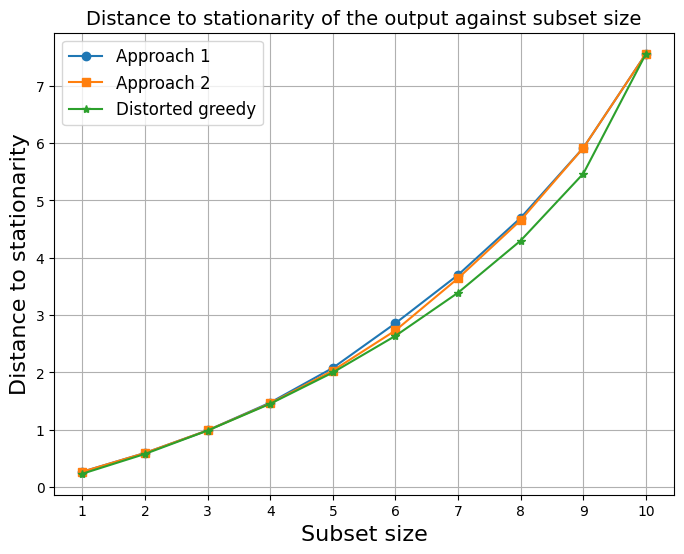}
    \caption{Bernoulli--Laplace level model}
    \label{fig:bg_BLLM_dist2stat}
\end{subfigure}
\begin{subfigure}[b]{0.48\linewidth}
    \centering
    \includegraphics[width=\linewidth]{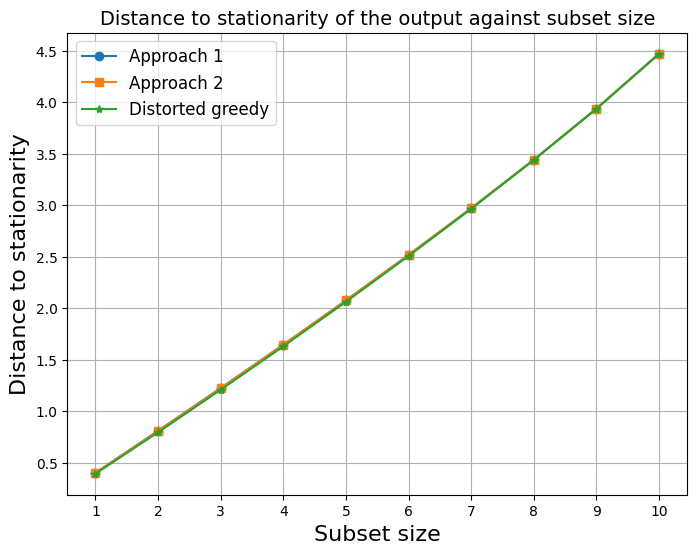}
    \caption{Curie--Weiss model}
    \label{fig:bg_CW_dist2stat}
\end{subfigure}
\caption{Distance to stationarity of the output against subset size.}
\label{fig:bg_dist2stat}
\end{figure}

We then report the numerical experiment results in Section~\ref{sec:gen_dist2stat}, see Table~\ref{tab:gen_distgrdy_dist2stat} and Figure~\ref{fig:gen_dist2stat}. Note that since the stationary distributions of the Bernoulli--Laplace level model (see Section~\ref{sec:BLLM}) and the Curie--Weiss model (see Section~\ref{sec:curie_weiss}) are not of product form, these simulations are heuristic in nature, \textcolor{black}{where} Corollary~\ref{cor:gen_distgrdy_dist2stat} does not provide a theoretical guarantee in this setting.

\begin{table}[H]
\centering
\begin{tabular}{|c|cccc|cccc|}
\hline
 & \multicolumn{4}{c|}{\textbf{Bernoulli--Laplace level model}} & \multicolumn{4}{c|}{\textbf{Curie--Weiss model}} \\
\hline
\(m\) & \(S_{m,1}\) & \(S_{m,2}\) & \(S_{m,3}\) & Value & \(S_{m,1}\) & \(S_{m,2}\) & \(S_{m,3}\) & Value \\
\hline
1  & \(\emptyset\)    & \(\emptyset\)    & \(\{10\}\)       & 0.23191  & \(\{1\}\)       & \(\emptyset\)   & \(\emptyset\)    & 0.39436 \\
2  & \(\emptyset\)    & \(\{7\}\)        & \(\{10\}\)       & 0.48566  & \(\{1\}\)       & \(\emptyset\)   & \(\{10\}\)       & 0.78871 \\
3  & \(\{4\}\)        & \(\{7\}\)        & \(\{10\}\)       & 0.74787  & \(\{1\}\)       & \(\{7\}\)      & \(\{10\}\)       & 1.19100 \\
4  & \(\{3,4\}\)      & \(\{7\}\)        & \(\{10\}\)       & 1.07820  & \(\{1\}\)       & \(\{7\}\)      & \(\{9,10\}\)    & 1.59492 \\
5  & \(\{3,4\}\)      & \(\{5,7\}\)      & \(\{10\}\)       & 1.41218  & \(\{1,2\}\)     & \(\{7\}\)      & \(\{9,10\}\)    & 1.99886 \\
6  & \(\{3,4\}\)      & \(\{5,7\}\)      & \(\{8,10\}\)     & 1.76157  & \(\{1,2\}\)     & \(\{6,7\}\)    & \(\{9,10\}\)    & 2.40381 \\
7  & \(\{1,3,4\}\)    & \(\{5,7\}\)      & \(\{8,10\}\)     & 2.15778  & \(\{1,2\}\)     & \(\{5,6,7\}\)  & \(\{9,10\}\)    & 2.81582 \\
8  & \(\{1,3,4\}\)    & \(\{5,6,7\}\)    & \(\{8,10\}\)     & 2.56632  & \(\{1,2,3\}\)   & \(\{5,6,7\}\)  & \(\{9,10\}\)    & 3.22828 \\
9  & \(\{1,3,4\}\)    & \(\{5,6,7\}\)    & \(\{8,9,10\}\)   & 3.02745  & \(\{1,2,3\}\)   & \(\{5,6,7\}\)  & \(\{8,9,10\}\)   & 3.64075 \\
10 & \(\{1,2,3,4\}\)  & \(\{5,6,7\}\)    & \(\{8,9,10\}\)   & 3.49326  & \(\{1,2,3,4\}\) & \(\{5,6,7\}\)  & \(\{8,9,10\}\)   & 4.06242 \\
\hline
\end{tabular}
\caption{Performance evaluation of Algorithm~\ref{alg:generalized_distorted_greedy}. ``Value" refers to $D(\otimes_{i=1}^3 P^{(S_{m, i})} \| \otimes_{i=1}^3 \Pi^{(S_{m, i})})$.}
\label{tab:gen_distgrdy_dist2stat}
\end{table}

\begin{figure}[H]
\centering
\begin{subfigure}[b]{0.48\linewidth}
    \centering
    \includegraphics[width=\linewidth]{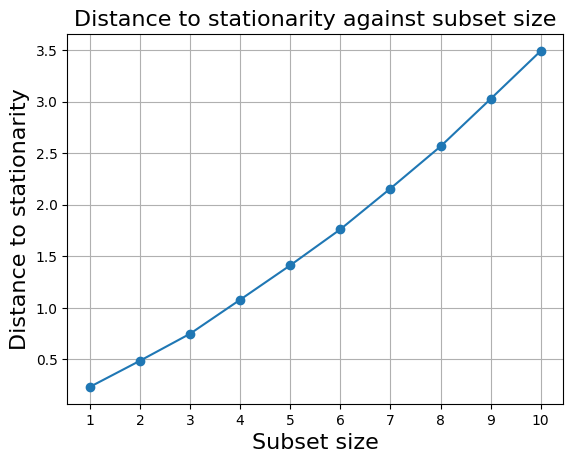}
    \caption{Bernoulli--Laplace level model}
    \label{fig:gen_distgrdy_BLLM_dist2stat}
\end{subfigure}
\begin{subfigure}[b]{0.48\linewidth}
    \centering
    \includegraphics[width=\linewidth]{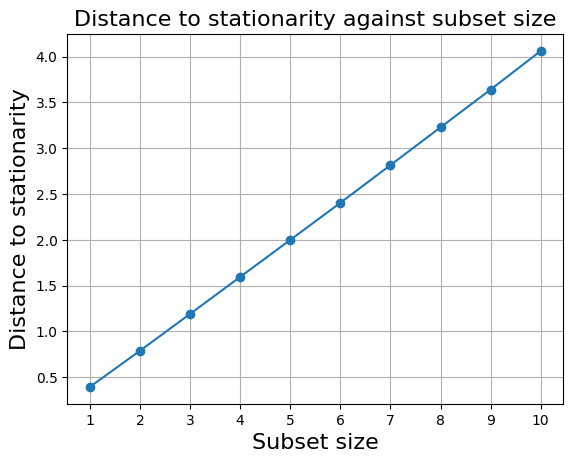}
    \caption{Curie--Weiss model}
    \label{fig:gen_distgrdy_CW_dist2stat}
\end{subfigure}
\caption{Performance evaluation of the generalized distorted greedy algorithm.}
\label{fig:gen_dist2stat}
\end{figure}

We proceed to present the numerical experiment results in Section~\ref{sec:dist2stat_c} and Section~\ref{sec:gen_dist2stat_c} (see Table~\ref{tab:greedy_dist2stat_c}, Table~\ref{tab:gen_distgrdy_dist2stat_c}, and Figure~\ref{fig:dist2stat_c}). Note that since the stationary distribution $\pi$ of both models is not of product form, \textcolor{black}{the application of greedy algorithm is thus a heuristic in this context, and} we do not have the $(1 - e^{-1})$-approximation guarantee.

\begin{table}[H]
\centering
\begin{tabular}{|c|cc|cc|}
\hline
 & \multicolumn{2}{c|}{\textbf{Bernoulli--Laplace level model}} & \multicolumn{2}{c|}{\textbf{Curie--Weiss model}} \\
\hline
\(m\) & \textbf{Subset \(S_m\)} & $D(P^{(-S_m)} \| \Pi^{(-S_m)})$ & \textbf{Subset \(S_m\)} & $D(P^{(-S_m)} \| \Pi^{(-S_m)})$ \\
\hline
1  & \(\{9\}\)                           & 5.46950    & \(\{10\}\)                          & 3.93568  \\
2  & \(\{9,\,10\}\)                       & 4.30094    & \(\{9,\,10\}\)                       & 3.43908   \\
3  & \(\{8,\,9,\,10\}\)                    & 3.39168    & \(\{8,\,9,\,10\}\)                    & 2.96487  \\
4  & \(\{7,\,8,\,9,\,10\}\)                & 2.63821   & \(\{7,\,8,\,9,\,10\}\)                & 2.507645  \\
5  & \(\{6,\,7,\,8,\,9,\,10\}\)             & 1.99871   & \(\{6,\,7,\,8,\,9,\,10\}\)             & 2.06420  \\
6  & \(\{4,\,6,\,7,\,8,\,9,\,10\}\)         & 1.45314   & \(\{5,\,6,\,7,\,8,\,9,\,10\}\)         & 1.63242  \\
7  & \(\{3,\,4,\,6,\,7,\,8,\,9,\,10\}\)      & 0.98630   & \(\{4,\,5,\,6,\,7,\,8,\,9,\,10\}\)      & 1.21075  \\
8  & \(\{1,\,3,\,4,\,6,\,7,\,8,\,9,\,10\}\)   & 0.58961   & \(\{3,\,4,\,5,\,6,\,7,\,8,\,9,\,10\}\)   & 0.79828   \\
9  & \(\{1,\,2,\,3,\,4,\,6,\,7,\,8,\,9,\,10\}\) & 0.25830  & \(\{2,\,3,\,4,\,5,\,6,\,7,\,8,\,9,\,10\}\) & 0.39435 \\
10 & \(\{1,\,2,\,3,\,4,\,5,\,6,\,7,\,8,\,9,\,10\}\)& 0.00000                    & \(\{1,\,2,\,3,\,4,\,5,\,6,\,7,\,8,\,9,\,10\}\)& 0.00000 \\
\hline
\end{tabular}
\caption{Performance evaluation of the greedy algorithm.}
\label{tab:greedy_dist2stat_c}
\end{table}

\begin{table}[H]
\centering
\begin{tabular}{|c|cccc|cccc|}
\hline
 & \multicolumn{4}{c|}{\textbf{Bernoulli--Laplace level model}} & \multicolumn{4}{c|}{\textbf{Curie--Weiss model}} \\
\hline
\(m\) & \(S_{m,1}\) & \(S_{m,2}\) & \(S_{m,3}\) & Value & \(S_{m,1}\) & \(S_{m,2}\) & \(S_{m,3}\) & Value \\
\hline
1  & \(\{4\}\)           & \(\emptyset\)      & \(\emptyset\)      & 3.02668  & \(\{4\}\)           & \(\emptyset\)      & \(\emptyset\)      & 3.64075 \\
2  & \(\{4\}\)           & \(\emptyset\)      & \(\{9\}\)         & 2.56554  & \(\{4\}\)           & \(\emptyset\)      & \(\{8\}\)         & 3.22828 \\
3  & \(\{4\}\)           & \(\{6\}\)         & \(\{9\}\)         & 2.15700  & \(\{3,4\}\)         & \(\emptyset\)      & \(\{8\}\)         & 2.81582 \\
4  & \(\{1,4\}\)         & \(\{6\}\)         & \(\{9\}\)         & 1.76235  & \(\{3,4\}\)         & \(\{5\}\)         & \(\{8\}\)         & 2.40381 \\
5  & \(\{1,4\}\)         & \(\{6\}\)         & \(\{8,9\}\)       & 1.41297  & \(\{3,4\}\)         & \(\{5,6\}\)       & \(\{8\}\)         & 1.99886 \\
6  & \(\{1,4\}\)         & \(\{5,6\}\)       & \(\{8,9\}\)       & 1.07899  & \(\{2,3,4\}\)       & \(\{5,6\}\)       & \(\{8\}\)         & 1.59492 \\
7  & \(\{1,2,4\}\)       & \(\{5,6\}\)       & \(\{8,9\}\)       & 0.74955  & \(\{2,3,4\}\)       & \(\{5,6\}\)       & \(\{8,9\}\)       & 1.19099 \\
8  & \(\{1,2,3,4\}\)     & \(\{5,6\}\)       & \(\{8,9\}\)       & 0.48566  & \(\{2,3,4\}\)       & \(\{5,6,7\}\)     & \(\{8,9\}\)       & 0.78871 \\
9  & \(\{1,2,3,4\}\)     & \(\{5,6,7\}\)     & \(\{8,9\}\)       & 0.23191  & \(\{2,3,4\}\)       & \(\{5,6,7\}\)     & \(\{8,9,10\}\)    & 0.39436 \\
10 & \(\{1,2,3,4\}\)     & \(\{5,6,7\}\)     & \(\{8,9,10\}\)    & 0.00000  & \(\{1,2,3,4\}\)     & \(\{5,6,7\}\)     & \(\{8,9,10\}\)    & 0.00000 \\
\hline
\end{tabular}
\caption{Performance evaluation of Algorithm~\ref{alg:generalized_distorted_greedy}. ``Value" refers to $D(\otimes_{i=1}^3 P^{(V_i \backslash S_{m, i})} \| \otimes_{i=1}^3 \Pi^{(V_i \backslash S_{m, i})})$.}
\label{tab:gen_distgrdy_dist2stat_c}
\end{table}

\begin{figure}[H]
\centering
\begin{subfigure}[b]{0.48\linewidth}
    \centering
    \includegraphics[width=\linewidth]{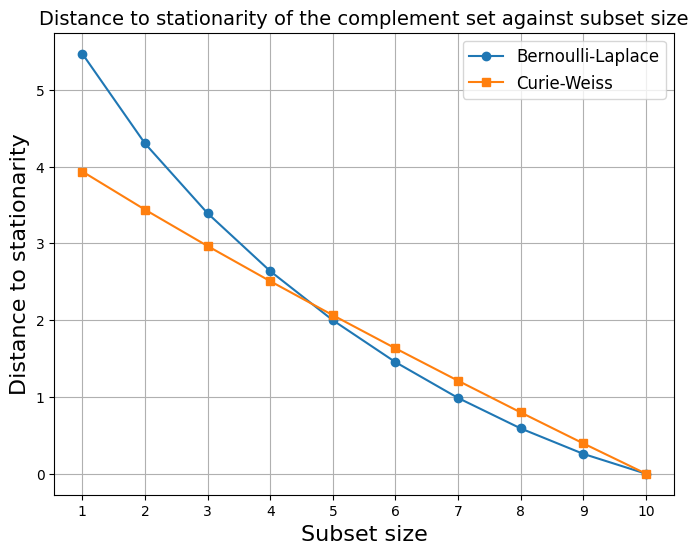}
    \caption{Greedy}
    \label{fig:greedy_dist2stat_c}
\end{subfigure}
\begin{subfigure}[b]{0.49\linewidth}
    \centering
    \includegraphics[width=\linewidth]{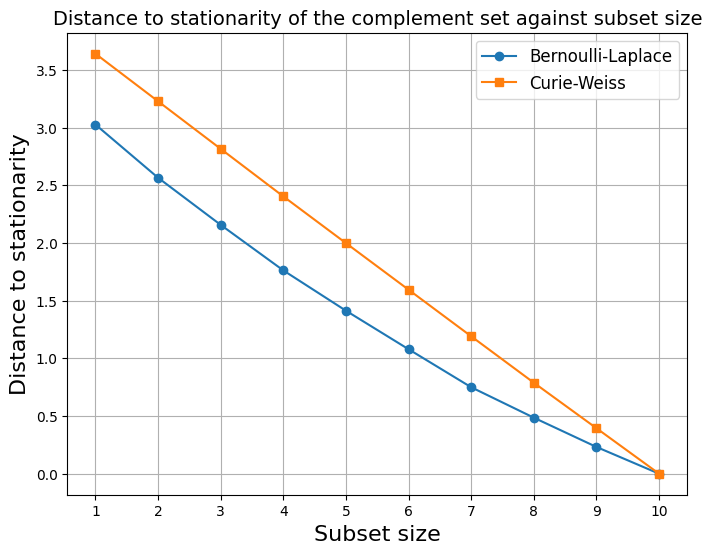}
    \caption{Algorithm~\ref{alg:generalized_distorted_greedy}}
    \label{fig:gen_distgrdy_dist2stat_c}
\end{subfigure}
\caption{Distance to stationarity of the complement set against subset size.}
\label{fig:dist2stat_c}
\end{figure}

\subsection{Experiment results of Section~\ref{sec:dist2fact_fixed}}
We perform Algorithm~\ref{alg:batch_greedy} with the following configuration: $l = \lceil \frac{m}{2} \rceil$, $q_i = 2$ for $i \in \llbracket l-1 \rrbracket$; $q_l = 2$ if $m$ is even, $q_l = 1$ if $m$ is odd. We choose the fixed subset as $W = \{1, 2, 3\}$. The performance of the batch greedy algorithm on the two models is shown in Table~\ref{tab:batch_greedy_dist2fact_fixed} and Figure~\ref{fig:batch_greedy_dist2fact_fixed}.

\begin{table}[H]
\centering
\begin{tabular}{|c|cc|cc|}
\hline
 & \multicolumn{2}{c|}{\textbf{Bernoulli--Laplace level model}} & \multicolumn{2}{c|}{\textbf{Curie--Weiss model}} \\
\hline
\(m\) & Subset $S_l$ & $D(P^{(W \cup S_l)} \| P^{(W)} \otimes P^{(S_l)})$ & Subset $S_l$ & $D(P^{(W \cup S_l)} \| P^{(W)} \otimes P^{(S_l)})$ \\
\hline
1 & \(\{10\}\) & 0.14671 & \(\{4\}\) & 0.02751 \\
2 & \(\{9,10\}\) & 0.26354 & \(\{4,10\}\) & 0.05651 \\
3 & \(\{8,9,10\}\) & 0.37787 & \(\{4,5,10\}\) & 0.08919 \\
4 & \(\{7,8,9,10\}\) & 0.49198 & \(\{4,5,9,10\}\) & 0.12616 \\
5 & \(\{6,7,8,9,10\}\) & 0.61908 & \(\{4,5,6,9,10\}\) & 0.17028 \\
6 & \(\{5,6,7,8,9,10\}\) & 0.79889 & \(\{4,5,6,8,9,10\}\) & 0.22527 \\
7 & \(\{4,5,6,7,8,9,10\}\) & 1.06993 & \(\{4,5,6,7,8,9,10\}\) & 0.30491 \\
\hline
\end{tabular}
\caption{Performance evaluation of the batch greedy algorithm.}
\label{tab:batch_greedy_dist2fact_fixed}
\end{table}

\begin{figure}[H]
    \centering
    \includegraphics[width=0.5\linewidth]{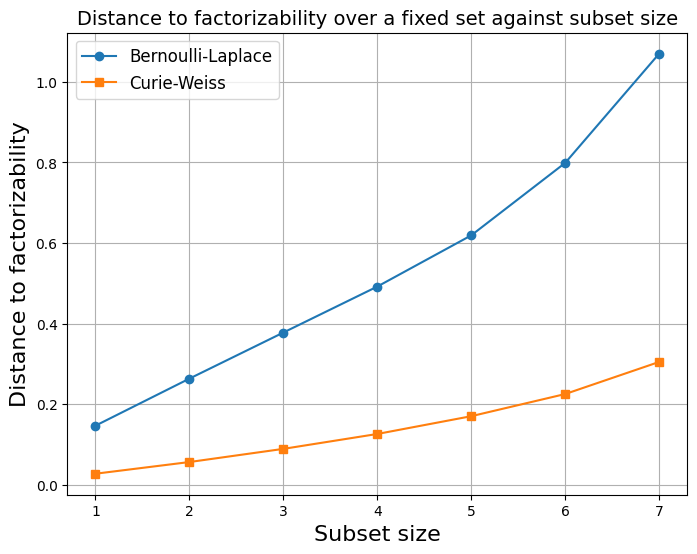}
    \caption{Performance evaluation of the batch greedy algorithm.}
    \label{fig:batch_greedy_dist2fact_fixed}
\end{figure}

\section*{Acknowledgements}
\textcolor{black}{We thank the careful reading and constructive comments of three reviewers that have significantly improved the quality of the manuscript.
Michael Choi acknowledges the financial support of the projects A-8001061-00-00, NUSREC-HPC-00001, NUSREC-CLD-00001, A-0000178-02-00 and A-8003574-00-00 at National University of Singapore. }

\section*{Data availability}

No data was used for the research described in the article.

\section*{Declarations}

\textbf{Conflict of interests} The authors have no relevant financial or non-financial interests to disclose.

\bibliographystyle{plain}
\bibliography{submodular}
\end{document}